\documentclass[reqno,12pt,letterpaper]{amsart}

\usepackage{amsmath,amssymb,amsthm,graphicx,mathrsfs,url,bm,subfigure,accents,paralist,xargs,slashed,enumitem,xparse,paralist}
\usepackage[usenames,dvipsnames]{color}
\usepackage[colorlinks=true,linkcolor=Red,citecolor=Green]{hyperref}
\usepackage{bbm}

\def\?[#1]{\textbf{[#1]}\marginpar{\Large{\textbf{??}}}}

\newlist{inlineroman}{enumerate*}{1}
\setlist[inlineroman]{itemjoin*={{, and }},afterlabel=~,label=\roman*.}

\newcommand{\inlineitem}[1][]{%
	\ifnum\enit@type=\tw@
	{\descriptionlabel{#1}}
	\hspace{\labelsep}
	\else
	\ifnum\enit@type=\z@
	\refstepcounter{\@listctr}\fi
	\quad\@itemlabel\hspace{\labelsep}
	\fi}

\DeclareFontEncoding{FMS}{}{}
\DeclareFontSubstitution{FMS}{futm}{m}{n}
\DeclareFontEncoding{FMX}{}{}
\DeclareFontSubstitution{FMX}{futm}{m}{n}
\DeclareSymbolFont{fouriersymbols}{FMS}{futm}{m}{n}
\DeclareSymbolFont{fourierlargesymbols}{FMX}{futm}{m}{n}
\DeclareMathDelimiter{\VERT}{\mathord}{fouriersymbols}{152}{fourierlargesymbols}{147}

\theoremstyle{plain}
\newtheorem{theo}{Theorem}[section]
\newtheorem{prop}{Proposition}[section]

\newtheorem{lem}[prop]{Lemma}
\newtheorem{cor}[prop]{Corollary}

\theoremstyle{remark}

\theoremstyle{definition}
\newtheorem{rem}[prop]{Remark}
\newtheorem{defi}[prop]{Definition}
\newtheorem{example}[prop]{Example}
\numberwithin{equation}{section}

\DeclareMathOperator{\supp}{supp}

\newcommand{\TP}{{\partial}_\nu}
\newcommand{\TD}{D_\nu}
\newcommand{\RNP}{\mathbb{T}^n_+}
\newcommand{\RP}{\mathbb{R}_+}
\newcommand{\RN}{\mathbb{T}^{n-1}}

\newcommand{\Cnu}{\mathcal{F}_\nu}
\newcommand{\Sob}{\mathcal{H}}
\newcommand{\TSob}{\widetilde{\mathcal{H}}}

\newcommand{\TSHARP}{\mathbb{T}^n_\sharp}

\newcommand{\overbar}[1]{\mkern 1.5mu\overline{\mkern-1.5mu#1\mkern-1.5mu}\mkern 1.5mu}

\newcommand{\UL}[1]
{
	\underline{#1}
}

\newcommand{\OL}[1]
{
	\overbar{#1}
}

\newcommand{\Ran}[3]
{
	\Sob^{#1 - 2}(#2) \times H^{#1-\UL{\mu}}(#3)
}

\newcommand*\colvec[3][]{
	\begin{pmatrix}\ifx\relax#1\relax\else#1\\\fi#2\\#3\end{pmatrix}
}

\renewcommand{\Im}{\operatorname{Im}}
\renewcommand{\Re}{\operatorname{Re}}
\graphicspath{{img/}}

\begin{document}
	
	\title
	[Elliptic boundary value problems for Bessel operators]
	{Elliptic boundary value problems for Bessel operators, with applications to anti-de Sitter spacetimes}
	\author{Oran Gannot}
	\email{gannot@northwestern.edu}
\address{Department of Mathematics, Lunt Hall, Northwestern University,
	Evanston, IL 60208, USA}

	\begin{abstract}
		
		This paper considers boundary value problems for a class of singular elliptic operators which appear naturally in the study of asymptotically anti-de Sitter (aAdS) spacetimes. These problems involve a singular Bessel operator acting in the normal direction. After formulating a Lopatinski\v{\i} condition, elliptic estimates are established for functions supported near the boundary. The Fredholm property follows from additional hypotheses in the interior. This paper provides a rigorous framework for mode analysis on aAdS spacetimes for a wide range of boundary conditions considered in the physics literature. Completeness of eigenfunctions for some Bessel operator pencils is shown.
	\end{abstract}
	
	\maketitle

\section{Introduction}
The study of linear fields on asymptotically anti-de Sitter (aAdS) spaces has stimulated new interest in boundary value problems for a class of singular elliptic equations wherein the operator $D_x^2 + (\nu^2 -1/4)x^{-2}$ acts on one of the variables \cite{enciso:2013,holzegel:2012:jhde,holzegel:2015,holzegel:2013:cmp,vasy:2012:apde,warnick:2013:cmp,warnick:2015:cmp}. To formulate this class of operators more precisely, consider a product manifold $[0,\varepsilon) \times \partial X$, where $\partial X$ is compact. The model for what we call a Bessel operator has the form
\[
P(x,y,D_x,D_y) = D_x^2 + (\nu^2 - 1/4)x^{-2} + A(x,y,D_y),
\]
where $(x,y) \in (0,\varepsilon)\times \partial X$ and $A$ is a family of second order differential operators on $\partial X$ depending smoothly on $x\in [0,\varepsilon)$. The parameter $\nu$ is required to be real and strictly positive. In the study of linear waves on aAdS spacetimes, $\nu$ is related to the mass of a scalar field --- see Section \ref{sect:motivation} for more details. The condition $\nu > 0$ corresponds to the Breitenlohner--Freedman bound \cite{breitenlohner:1982:plb,breitenlohner:1982:ap}.

Boundary data for this problem are formally defined by the following weighted restrictions:
\[
\gamma_- u = x^{\nu-1/2}u|_{\partial X}, \quad \gamma_+ u = x^{1-2\nu}\partial_x (x^{\nu - 1/2}u)|_{\partial X}.
\]
Some care is needed to give precise meaning to these restrictions --- see Section \ref{subsect:traces}, along with an earlier discussion in \cite{warnick:2013:cmp}. The boundary operators in this paper are of the form $T = T^- \gamma_- + T^+ \gamma_+ $, where $T^-, \, T^+$ are differential operators on $\partial X$ of order at most one and zero, respectively. This paper is concerned with solvability of the boundary value problem
\begin{equation} \label{eq:introBVP}
\begin{cases}
P u  = f  \text{ on $X$} \\ 
T u = g \text{ on $\partial X$}
\end{cases}
\end{equation}
when $0 < \nu <1$, and the simpler equation
\begin{equation} \label{eq:introeqn}
Pu = f \text{ on $X$}
\end{equation}
when $\nu \geq 1$. No boundary conditions are imposed when $\nu \geq 1$. The difference between the cases $0 < \nu < 1$ and $\nu \geq 1$ is explained in more detail in the introduction to Section \ref{sect:ellipticBVP}.  

Ellipticity of the Bessel operator $P$ is defined in Section \ref{subsect:ellipticity}. As in the study of smooth boundary value problems, there is also a notion of ellipticity for \eqref{eq:introBVP} given by a natural Lopatinski\v{\i} condition on the pair $(P,T)$. This condition is introduced in Section \ref{subsect:lopatinskii}. Elliptic estimates are proved in Theorem \ref{theo:bvptheo}. When the operators $P,\, T$ depend polynomially on a spectral parameter $\lambda$, there is a notion of parameter-ellipticity for both $P$ and the boundary value problem \eqref{eq:introBVP}. Theorem \ref{theo:bvptheosemiclassical} provides elliptic estimates in terms of parameter-dependent norms which are uniform as $|\lambda|\rightarrow \infty$ in the cone of ellipticity.

For the global problem, consider a compact manifold  $\OL X$ where $[0,\varepsilon) \times \partial X$ is identified with a collar neighborhood of $\partial X$. Suppose that the restriction of $P$ to this collar is a Bessel operator --- see Section \ref{sect:prelim} for details. As in the case of smooth boundary value problems, estimates for $P$ near $\partial X$ may often be combined with estimates in the interior $X$ to establish the Fredholm property (including some cases where $P$ fails to be everywhere elliptic on $X$). In Section \ref{sect:fredholm}, a sufficient condition of this type is discussed. Furthermore, in the presence of a spectral parameter $\lambda$, unique solvability is established for $\lambda$ in the cone of ellipticity provided $|\lambda|$ is sufficiently large.

Section \ref{sect:motivation} recalls the notion of an aAdS metric, which is the primary motivation for this paper. It is shown that the Fourier transformed Klein--Gordon operator is indeed a Bessel operator whose order $\nu$ depends on the Klein--Gordon mass. One goal is to study the Klein--Gordon equation by Fourier synthesis once its spectral family is understood, corresponding to the study of normal or quasinormal modes. This paper provides a rigorous framework for future work in that direction; see \cite{gannot2017existence,gannot:2014:kerr} for some recent progress.

For stationary aAdS spacetimes with compact time slices and an everywhere timelike Killing field $\partial_t$, Section \ref{sect:completeness} describes a class of boundary conditions which yield a complete set of normal modes associated to a discrete set of eigenvalues. Of particular interest are boundary conditions which depend on $\partial_t$ (hence depend on the spectral parameter after a Fourier transform). This is important for the study of modes with transparent or dissipative boundary conditions, along with superradiance phenomena \cite{avis:1978:prd,holzegel:2015,winstanley:2001}.

The approach of this paper is inspired by the texts of Roitberg \cite{roitberg:1996} and Kozlov--Maz'ya--Rossman \cite{kozlov:1997} in the smooth setting. This approach is particularly suited to the singular nature of Bessel operators, and allows for the study of boundary value problems in low regularity spaces as needed in applications to general relativity --- see Section \ref{sect:fredholm}. All the methods are classical, using only homogeneity properties of differential operators. The key is exploiting the theory of ``twisted'' derivatives as first emphasized in \cite{warnick:2013:cmp}. This is based on the classical observation that the one-dimensional Bessel operator $D_x^2 + (\nu^2 - 1/4)x^{-2}$ admits a factorization as the product of a first order order operator and its adjoint; this first order operator is then treated as an elementary derivative. 

Using a variational approach, similar elliptic estimates were studied  by Holzegel and Warnick \cite{holzegel:2012:jhde,holzegel:2012wt,warnick:2013:cmp,warnick:2015:cmp}. However, only the ``classical'' self-adjoint boundary conditions were handled when $0 < \nu < 1$; these are the Dirichlet ($T = \gamma_-$) and Robin boundary conditions ($T = \gamma_+ + \beta \gamma_- $ with $\beta$ a real-valued function). The approach taken here accounts for a larger class of non-self-adjoint boundary conditions.

Using results of this paper, discreteness of quasinormal frequencies on Kerr--AdS black holes with arbitrary rotation speed is established in \cite{gannot:2014:kerr}. These frequencies replace eigenvalues in scattering problems \cite{cardoso:2014:jhep,dias:2012:cqg:b,gannot:2014:kerr,horowitz:2000:prd,konoplya:2011,warnick:2015:cmp}.  When $0 < \nu < 1$, arbitrary boundary conditions satisfying the Lopatinski\v{\i} condition may be imposed on the field (although of course one does not have any completeness statement). This generalizes earlier work of Warnick \cite{warnick:2015:cmp}, where, as noted above, only Dirichlet or Robin boundary conditions were considered.

The results of this paper should also be compared to earlier works of Vasy \cite{vasy:2012:apde} and  Holzegel  \cite{holzegel:2012:jhde} on aAdS spaces, where a more restrictive measure is used to define the space of square integrable functions. In those works, the square integrability condition is equivalent to the generalized Dirichlet boundary condition. This limits the range of applications, since different boundary conditions are employed throughout the physics literature on aAdS spaces \cite{avis:1978:prd,berkooz:2002:jhep,breitenlohner:1982:plb,breitenlohner:1982:ap,dias:2013:jhep,ishibashi:2004:cqg,witten:2001}.

There is also a general microlocal approach to degenerate boundary value problems developed by Mazzeo--Melrose \cite{mazzeo:1987} and Mazzeo \cite{mazzeo:1991}, culminating in the work of Mazzeo--Vertman \cite{mazzeo:2013} on general boundary value problems. In particular, the elliptic theory in \cite{mazzeo:2013} could likely reproduce the results of this paper, and is also applicable to much more general classes of elliptic operators. On the other hand, the approach developed here is directly motivated by the physics literature. For instance, the Sobolev spaces used in this paper were originally defined in \cite{warnick:2013:cmp} to give precise meaning to the energy renormalization implicit in the work of Breitenlohner--Freedman \cite{breitenlohner:1982:plb}. Furthermore, the parameter-dependent (or semiclassical) theory is not developed in \cite{mazzeo:2013}. There is also a simplicity advantage in using physical space methods, rather than a more sophisticated microlocal approach. It should also be stressed that Bessel operators of the precise kind studied here arise in numerous contexts outside of general relativity with negative cosmological constant, both mathematical and physical.

\section{Preliminaries} \label{sect:prelim}

\subsection{Conventions for differential operators} \label{subsect:conventions}
If $P$ is a smooth differential operator on a manifold $Y$, then in local coordinates,
\begin{equation} \label{eq:difflocalcoords}
P = \sum_{|\alpha| \leq m} a_{\alpha}(y) D_y^\alpha.
\end{equation}
The order of $P$ is said to no greater than $m$, written $\mathrm{ord}(P) \leq m$. If $\mathrm{ord}(P) \leq m$, then the symbol $\sigma_m(P)$ of $P$ with respect to $m$ is the polynomial function on $T^*Y$ given in local coordinates by
\[
\sigma_m(P)(y,\eta) = \sum_{|\alpha|=m} a_\alpha(y) \eta^\alpha, \quad (y,\eta) \in T^*Y.
\]
The space of smooth differential operators of order no greater than $m$ is denoted $\mathrm{Diff}^m(Y)$. The following convention will be useful throughout Section \ref{sect:ellipticBVP}:  if $\mathrm{ord}(P) \leq m$ with $m < 0$, then $P = 0$; conversely if $P = 0$, then $P$ can be assigned any negative order. 

The class of parameter-dependent differential operators on a manifold $Y$ is defined as follows: $P \in \mathrm{Diff}_{(\lambda)}^m(Y)$ if in local coordinates,
\[
P(y,D_y;\lambda) = \sum_{j + |\alpha| \leq m} a_{\alpha,j}(x) \lambda^j D_y^\alpha.
\]
The statement  $\mathrm{ord}^{(\lambda)}(P) \leq m$  about the parameter-dependent order of $P$ is defined by assigning to $\lambda^j$ the same weight as a derivative of order $j$. Thus the parameter-dependent principal symbol of $P$ is given by
\[
\sigma^{(\lambda)}_{m}(A) = \sum_{j + |\alpha| = m} a_{\alpha,j}(t)\lambda^j \eta^\alpha ,\quad (y,\eta,\lambda) \in T^*Y \times \mathbb{C}.
\]
If $Y$ is a compact manifold without boundary, the parameter -dependent Sobolev norms on $Y$ are defined for $s\geq 0$ by
\[
\VERT u \VERT^2_{H^s(Y)} = |\lambda|^{2s} \| u \|^2_{H^0(Y)} + \| u \|^2_{H^s(Y)},
\]
and an operator $P \in \mathrm{Diff}^m_{(\lambda)}(Y)$ is bounded $H^s(Y) \rightarrow H^{s-m}(Y)$ uniformly with respect to $|\lambda|$ in these norms.

\subsection{Manifolds with boundary} \label{subsect:manifoldwithboundary}
Let $\OL X = X \cup \partial X$ denote an $n$-dimensional manifold with compact boundary $\partial X$ and interior $X$. A boundary defining function for $\partial X$ is a function $x \in C^\infty(\OL X)$ satisfying
\[
x^{-1}(0) = \partial X, \quad x > 0 \text{ on } X, \quad {dx}|_{\partial X} \neq 0.
\]
Given $x$, there exists an open subset $\OL{W} \supseteq \partial X$, a number $\varepsilon > 0$, and a diffeomorphism $\phi : [0,\varepsilon) \times \partial X \rightarrow \OL{W}$ such that $x \circ \phi$ agrees with the projection $[0,\varepsilon) \times \partial X \rightarrow [0,\varepsilon)$. A collar of this type is said to be compatible with $x$. 
Unless otherwise specified, a manifold with boundary $\OL{X}$ will always refer to $\OL{X}$ equipped with a distinguished boundary defining function $x$ and a choice of compatible collar diffeomorphism $\phi$.

\subsection{Bessel operators} \label{subsect:bessel}
Given $\nu \in \mathbb{R}$, formally define the differential operator $\partial_\nu$ by the formula
\[
\TP = \partial_x + (\nu-1/2){x^{-1}} = x^{1/2-\nu}\partial_x x^{\nu-1/2}.
\]
Furthermore, let $\TP^* = -x^{\nu-1/2}\partial_x x^{1/2-\nu}$, which is the formal adjoint of $\TP$ with respect to Lebesgue measure on $\RP$. Similarly, let $\TD = -i \TP$ and $\TD^* = i\TP^*$. Note that
\[
|\TD|^2 := \TD^*\TD = D_x^2 + (\nu^2 - 1/4)x^{-2}
\]
is the one-dimensional Bessel operator in Schr\"odinger form.

\begin{defi}
	Let $\OL{X}$ denote a manifold with compact boundary. A differential operator $P \in \mathrm{Diff}^2(X)$ is called a Bessel operator of order $\nu > 0$ if there exist
	\[
	A = A(x,y,D_y) \in \mathrm{Diff}^2(\partial X), \quad B= B(x,y,D_y) \in \mathrm{Diff}^1(\partial X)
	\]
	depending smoothly on $x \in [0,\varepsilon)$, such that $B(0,y,D_y) = 0$ and
	\begin{equation} \label{eq:besseloperator}
	\phi^* P = |\TD|^2 + B(x,y,D_y)\TD + A(x,y,D_y).
	\end{equation}
The set of such operators is denoted by $\mathrm{Bess}_\nu(X)$.
	\end{defi}

The requirement that $|\TD|^2$ appears with unit coefficient is not at all essential. If the coefficient is a nonvanishing function smooth up to $x = 0$, then the quotient of $P$ by this coefficient is a Bessel operator as above, and this normalization does not affect any of the arguments in Sections \ref{sect:ellipticBVP}, \ref{sect:fredholm}.

The class of Bessel operators depends strongly on the pair $(x,\phi)$, where $x$ is a boundary defining function and $\phi$ is a collar diffeomorphism compatible with $x$ in the sense of Section \ref{subsect:manifoldwithboundary}. This dependence will be written as $\mathrm{Bess}_\nu(X;x,\phi)$ when necessary. On the other hand, let $(\tilde x, \tilde y)$ denote arbitrary local coordinates near $\partial X$, where $\tilde{x}$ is a local boundary defining function. If $(x,y)$ are local coordinates induced by $\phi$, then $P$ is still of the form \eqref{eq:besseloperator} (up to a smooth nonvanishing multiple) in $(\tilde x,\tilde{y})$ coordinates provided that $\tilde{x}/x$ and $\tilde{y}$ are even functions of $x$ modulo $\mathcal{O}(x^3)$. 

A (smooth, positive) density $\mu$ on $\OL X$ is said to be of product type near $\partial X$ if 
\[
\phi^*\mu = |dx| \otimes \mu_{\partial X}
\]
for a fixed density $\mu_{\partial X}$ on $\partial X$. It is of course always possible to choose a density of product type near $\partial X$. This is useful in light of the next lemma. If $\OL{X}$ is compact, then $L^2(X)$ may be defined as the space of square integrable functions with respect to any smooth density $\mu$ on $\OL X$, in particular one of product type near $\partial X$. 

\begin{lem} \label{lem:closedunderadjoint}
Suppose that $\mu$ is of product type near $\partial X$. If $P \in \mathrm{Bess}_\nu(X)$, then $P^* \in \mathrm{Bess}_\nu(X)$, where $P^*$ is the formal adjoint of $P$ with respect to $\mu$.
\end{lem}
\begin{proof}
The pullback of $P^*$ to $(0,\varepsilon)\times \partial X$ is given by	
\[
|\TD|^2 + \TD^* B^* + A^*,
\]
where $A^*,B^*$ are the formal adjoints of $A, B$ with respect to $\mu_{\partial X}$. On the other hand, 
\[
\TD^* B^* = B^* \TD^* + [D_x,B^*].
\]
Furthermore, since $B = xB_1$ for a first order operator $B_1$ on $\partial X$ depending smoothly on $x\in[0,\varepsilon)$, it follows that
\[
B^*\TD^* = xB_1^*\TD^* = B_1^*(xD_x - i(1/2-\nu)) = B^*\TD - i(1/2-\nu)B_1^*,
\]
which completes the proof since the multiple of $B_1^*$ as well as $[D_x,B^*]$ may be absorbed into $A^*$.
\end{proof}

For the local theory, it is convenient to work on $\RNP = \RP \times \RN$, where $\RN = \left( \mathbb{R}/2\pi \mathbb{Z} \right)^{n-1}$. The set of Bessel operators on $\RNP$ is defined with respect to the canonical product decomposition on $\RNP$. Thus $P \in \mathrm{Bess}_\nu(\RNP)$ if
\[
P(x,y,\TD,D_y) =  |\TD|^2 + \sum_{|\beta|\leq 1} b_\beta(x,y) D_y^\beta \TD + \sum_{|\alpha| \leq 2} a_{\alpha}(x,y)D_y^\alpha
\]
for $b_\beta \in xC^\infty(\OL{\RNP})$ and $a_{\alpha} \in C^\infty(\OL{\RNP})$. When working on $\RNP$, the functions $b_\beta, \, a_{\alpha}$ are referred to as the coefficients of $P$. 

Fix a coordinate chart $Y \subseteq \partial X$ and a diffeomorphism $\theta : Y \rightarrow V$, where $V$ is an open subset of $\mathbb{T}^{n-1}$. Setting 
\[
U = \phi([0,\varepsilon) \times Y),
\]
the map $\psi : U \rightarrow [0,\varepsilon) \times V$ given by $\psi = (1\times \theta)\circ \phi^{-1}$ defines a boundary coordinate chart on $\OL X$. Given $P \in \mathrm{Bess}_\nu(X)$, there clearly exists $P_U \in \mathrm{Bess}_\nu(\RNP)$ such that
\[
Pu = P_U(u \circ \psi)
\]
for each $u \in {C}_c^\infty(U^\circ)$. Furthermore, it is always possible to arrange it so that the coefficients of $P_U$ (in the sense of the previous paragraph) are constant outside a compact subset of $\OL{\RNP}$.

\subsection{Ellipticity and the boundary symbol} \label{subsect:ellipticity}

Given $P \in \mathrm{Bess}_\nu(X)$ which near $\partial X$ has the form
\[
P = |\TD|^2 + B\TD + A,
\] 
let $A_0(y,D_y) = A(0,y,D_y)$. Ellipticity of $P$ at a point $p \in \partial X$ is defined via the function
\begin{equation} \label{eq:besselsymbol}
\xi^2 + \sigma_2(A_0)(p,\eta),
\end{equation}
which is a homogeneous polynomial of degree two in $(\xi,\eta) \in \mathbb{R} \times T_p^* \partial X$.

\begin{defi} \label{def:elliptic}
The Bessel operator $P \in \mathrm{Bess}_\nu(X)$ is said to be (properly) elliptic at $p \in \partial X$ if for each $\eta \in T_p^*\partial X \setminus 0$ the polynomial
\begin{equation} \label{eq:properlyellipticpolynomial}
\xi \mapsto \xi^2 + \sigma_2(A_0)(p,\eta)
\end{equation}
has no real roots.
\end{defi}

Thus, ellipticity implies the existence of nonreal roots $\pm \xi(p,\eta)$, where $\Im \xi(p,\eta) < 0$ by convention. For each $(p,\eta) \in T^* \partial X \setminus 0$, the symbol $\sigma_2(A_0)(p,\eta)$ determines a family of one dimensional Bessel operators given by
\begin{equation} \label{eq:boundarysymbol}
\widehat{P}_{(p,\eta)} = |\TD|^2 + \sigma_2(A_0)(p,\eta).
\end{equation}
The operator $\widehat{P}_{(p,\eta)}$ is called the boundary symbol operator of $P$. Let $\mathcal{M}_+(p,\eta)$ denote the space of solutions to the equation
\[
\widehat{P}_{(p,\eta)}u = 0
\]
which are bounded as $x\rightarrow \infty$. Ellipticity at $p\in \partial X$ implies that $\dim \mathcal{M}_+(p,\eta) = 1$ for each $\eta \in T_p^* \partial X \setminus 0$. Indeed, the space of solutions to $\widehat{P}_{(p,\eta)}u = 0$ is spanned by the modified Bessel functions 
\[
\left\lbrace x^{1/2}K_\nu(i\xi(p,\eta)x),\, x^{1/2}I_\nu(i\xi(p,\eta)x)\right\rbrace.
\]
Since $\Re i\xi(p,\eta) > 0$, it follows that 
\[
x^{1/2}K_\nu(i\xi(p,\eta)x) = \mathcal{O}\left(e^{-x/C}\right), \quad x \rightarrow \infty, 
\]
while the second solution grows exponentially \cite[Chapter 7.8]{olver:2014}. Thus only the first solution can possibly lie in $\mathcal{M}_+(p,\eta)$.

\subsection{Parameter-dependent Bessel operators} \label{subsect:semiclassicalbessel}

\begin{defi} \label{def:semiclassicalbesseloperator}
Let $\OL{X}$ denote a compact manifold with boundary as in Section \ref{subsect:manifoldwithboundary}. A differential operator $P(\lambda) \in \mathrm{Diff}^2_{(\lambda)}(X)$ is called a parameter-dependent Bessel operator of order $\nu > 0$ if there exist
\[
A(\lambda) = A(x,y,D_y;\lambda) \in \mathrm{Diff}^2_{(\lambda)}(\partial X), \quad B(\lambda) = B(x,y,D_y;\lambda)\in \mathrm{Diff}^1_{(\lambda)}(\partial X)
\]
depending smoothly on $x \in [0,\varepsilon)$, such that $B(0,y,D_y;\lambda) = 0$ and 
\begin{equation} \label{eq:parameterbesseloperator}
\phi^* P(\lambda) = |\TD|^2 + B(x,y,D_y;\lambda)\TD + A(x,y,D_y;\lambda).
\end{equation}
The set of such operators is denoted by $\mathrm{Bess}^{(\lambda)}_\nu(X)$.
\end{defi}

Ellipticity with parameter is defined by replacing the standard principal symbol of $A$ with its parameter-dependent version. Begin by fixing an angular sector $\Lambda \subseteq \mathbb{C}$.

\begin{defi} \label{def:semiclassicalelliptic}
A parameter-dependent Bessel operator $P(\lambda)$ is said to be (properly) parameter-elliptic with respect to $\Lambda$ at $p \in \partial X$ if for each $(\eta,\lambda) \in T_p^*\partial X \times \Lambda \setminus 0$, the polynomial
\begin{equation} \label{eq:besselsymbolsemiclassical}
\xi \mapsto \xi^2 + \sigma_2^{(\lambda)}(A_0)(p,\eta;\lambda)
\end{equation}
has no real roots.
\end{defi}

Similarly, for $(p,\eta,\lambda) \in T^* \partial X \times \Lambda \setminus 0$, define
\[
\widehat{P}_{(p,\eta;\lambda)} = |\TD|^2 + \sigma^{(\lambda)}_2 (A_0)(p,\eta;\lambda),
\]
and then let $\mathcal{M}_+(p,\eta;\lambda)$ denote the space of solutions to $\widehat{P}_{(p,\eta;\lambda)}u=0$ which are bounded as $x\rightarrow \infty$. As in Section \ref{subsect:ellipticity}, this space is one-dimensional.

\section{Motivation: asymptotically anti-de Sitter manifolds} \label{sect:motivation}

This section recalls the notion of an asymptotically anti-de Sitter (aAdS) metric. Then, a convenient expression for the Klein--Gordon equation is given with respect to a certain product decomposition near the conformal boundary. By means of a Fourier transform, the initial boundary value problem for the Klein--Gordon equation is reduced to the study of the boundary value problem for a stationary partial differential equation depending polynomially on the spectral parameter. The corresponding operator is a Bessel operator whose order $\nu$ depends on the Klein--Gordon parameter; the condition $\nu > 0$ translates into the well-known Breitenlohner--Freedman bound.

\subsection{aAdS metrics}

Let $\OL{X}$ denote an $n$-dimensional compact manifold with boundary as in Section \ref{subsect:manifoldwithboundary}, and set $\OL M = \mathbb{R} \times \OL X$. Here $t \in \mathbb{R}$ will denote a global time coordinate. A boundary defining function $\rho$ on $M$ satisfying $\partial_t \rho = 0$ is said to be stationary. There is an obvious one-to-one correspondence between stationary boundary defining functions on $\OL M$ and boundary defining functions $x$ for $\OL X$.

\begin{defi}
A smooth Lorentzian metric $g$ on $M$ is said to asymptotically simple if there exists a boundary defining function $\rho \in C^\infty(\OL M)$ with the following properties.
\begin{enumerate} \itemsep6pt
\item the Lorentzian metric $\bar{g} = \rho^2 g$ admits a smooth extension to $\OL M$,
\item the restriction $\bar{g}|_{T\partial  M}$ is again Lorentzian.
\end{enumerate}
\end{defi}
The map $g \mapsto \bar{g}|_{\partial  M}$ depends on $g$ and a choice of boundary defining function. However, the conformal class $[\bar{g}|_{T \partial  M}]$ depends only on $g$, since any two boundary defining functions differ by a positive multiple. Also note that if $g$ is asymptotically simple, then $d\rho$ is spacelike for $\bar{g}|_{\partial  M}$.

\begin{defi}
An asymptotically simple manifold $(\OL M,g)$ is said to be aAdS if there exists a boundary defining function $\rho$ such that 
$|d\rho|_{\bar g}^2 = -1$ on $\partial M$.
\end{defi}

The aAdS property does not depend on the choice of boundary defining function. In addition to being aAdS, suppose that $g$ is stationary in the sense that $\partial_t$ is a Killing vector field for $g$. For the remainder of this section, all metrics are assumed to be stationary.

The following is a well-known observation of Graham--Lee \cite{graham:1991}, adapted to the Lorentzian setting.
\begin{lem} \label{lem:geodesicbdf}
Suppose that $(\OL M, g)$ is an aAdS spacetime. If $g$ and $\gamma_0 \in [\bar{g}|_{T\partial  M}]$ are stationary, then there exists a unique stationary boundary defining function $x$ with the following properties.
\begin{enumerate} \itemsep6pt
\item $x^2 g|_{T\partial  M} = \gamma_0$,
\item $|dx|^2_{x^2 g} = -1$ in a collar neighborhood of $\partial M$.
\end{enumerate}
\end{lem}
\begin{proof}
The proof of \cite[Lemma 5.2]{graham:1991} goes through unchanged. 
\end{proof} 
If $x$ satisfies the condition described in Lemma \ref{lem:geodesicbdf}, then $x$ is said to be a geodesic boundary defining function. Note the integral curves of $\nabla_{x^2 g} x$ are geodesics of $x^2 g$ near $\partial M$. Thus the Gauss lemma implies that
\begin{equation} \label{eq:modelmetric}
\phi^*(g) = \frac{-ds^2 + \gamma(s)}{s^2}
\end{equation}
on $[0,\varepsilon) \times \partial M$, where $\phi : [0,\varepsilon) \times \partial M \rightarrow \OL{\mathcal{C}}$ is the collar diffeomorphism obtained by the flow-out of $-\nabla_{x^2 g}x$. Here $s\mapsto \gamma(s)$ is a smooth family of stationary Lorentzian metrics on $\partial M$ such that $\gamma(0) = \gamma_0$. Furthermore, $\phi^*x = s$, so by an abuse of notation it is convenient to write $g = x^{-2}(-dx^2 + \gamma(x))$ on the collar neighborhood $\OL{\mathcal{C}}$.

\begin{defi} \label{defi:modulox3} The metric $g$ as in Lemma \ref{lem:geodesicbdf} is said to even modulo $\mathcal{O}(x^3)$ (in the sense of Guillarmou \cite{guillarmou:2005}) if there exists a two-tensor $\gamma_1$ on $\partial M$ such that
\[
\gamma(s) = \gamma_0 + s^2 \gamma_1 + \mathcal{O}(s^3).
\]
\end{defi}

\noindent As in \cite[Proposition 2.1]{guillarmou:2005}, this evenness property is instrinsic to the conformal class $[\bar{g}|_{T\partial  M}]$. The fundamental class of aAdS metrics which are even modulo $\mathcal{O}(x^3)$ are the Einstein aAdS metrics; see \cite[Section 2]{anderson2004structure} for example. This includes all of the physically motivated aAdS spaces. The Einstein condition also enforces additional conditions on the expansion of $\gamma(s)$ which are not exploited here (in the asymptotically hyperbolic setting, see Mazzeo--Pacard \cite[Section 2]{mazzeo2011constant}).

Let $\phi : [0,\varepsilon) \times \partial M \rightarrow \OL{\mathcal{C}}$ denote the collar diffeomorphism associated to a geodesic boundary defining function $x$ as above, and let $t : \OL{M} \rightarrow \mathbb{R}$ be the time function associated with the identification $\OL{M} = \mathbb{R} \times \OL{X}$. If $t_0$ denotes the restriction of $t$ to $\partial M$, then $t_0$ gives a time coordinate on $[0,\varepsilon)\times \partial M$. In general it is not true that $\phi^*(t) = t_0$ unless $x^{-2}g^{-1}(dx,dt)$ vanishes identically. On the other hand, by stationarity there is a $t_0$-invariant function $h$ such that $\phi^*(t) = t_0 + h$, so the map 
\begin{equation} \label{eq:phi0}
\phi_0(x,y,t_0) = \phi(x,y,t_0-h)
\end{equation}
does satisfy $\phi_0^*(t) = t_0$, and $\phi_0 : [0,\varepsilon)\times \partial M \rightarrow \OL{\mathcal{C}}$ is still compatible with $x$.  Furthermore, if $\{t=0\}$ meets $\partial M$ orthogonally with respect to $x^2 g$, then $dh|_{\partial M} = 0$ as well.

\subsection{The Klein--Gordon equation} \label{subsect:kleingordon}
Fix a stationary aAdS spacetime $(\OL M, g)$ and a geodesic boundary defining function $x$ as in Lemma \ref{lem:geodesicbdf}. Furthermore, suppose that $g$ is even modulo $\mathcal{O}(x^3)$ in the sense of Definition \ref{defi:modulox3}. In light of the product decomposition $\eqref{eq:modelmetric}$ it follows that near $\partial M$,
\[
\Box_g = x^2 D_x^2 + i(n - 1 + e(x))xD_x + x^2\Box_{\gamma(x)},
\]
where $x \mapsto e(x)$ is a smooth family of functions on $\partial M$ such that 
\[
e(x) = x^2 e_0 + \mathcal{O}(x^3)
\] 
for some $e_0 \in C^\infty(\partial M)$ and $\partial_{t} e(x) = 0$. Indeed, $e(x) = (1/2) x\partial_x \log(\det \gamma(x))$, and $\det \gamma(x) = \det \gamma_0 + \mathcal{O}(x^2)$. Given $\nu > 0$, let
\[
P_g = x^{-(n+1)/2}( \Box_g + \nu^2 - n^2/4) x^{(n-1)/2},
\]
which corresponds to conjugating the Klein--Gordon operator with mass $\nu^2 - n^2/4$ by $x^{(n-1)/2}$ and then dividing by $x^2$. Explicitly,
\begin{equation} \label{eq:Pg}
P_g = D_x^2 +(\nu^2 - 1/4)x^{-2} + ix^{-1}e(x)D_x + \left( \tfrac{n-1}{2} \right) x^{-2}e(x) + \Box_{\gamma(x)}.
\end{equation}
In what follows, $\OL{X}$ will be identified with the time slice $\{t=0\}$ within $\OL{M}$, and functions on $\OL{X}$ are  identified with $t$-invariant functions on $\OL{M}$. Using that $g$ is stationary, it is possible to define a differential operator on $X$, depending on $\lambda \in \mathbb{C}$, by
\begin{equation} \label{eq:stationarykleingordon}
P(\lambda)u = e^{i\lambda t} P_g (e^{-i\lambda t}u),
\end{equation}
where $u \in C^\infty(X)$.

\begin{lem} \label{lem:stationarykleingordon}
Let $(\OL M, g)$ denote a stationary aAdS spacetime. Suppose that $x$ is a geodesic boundary defining function, and $g$ is even modulo $\mathcal{O}(x^3)$. Let $\phi_0$ be given by \eqref{eq:phi0}. If $\OL{X}$ meets $\partial M$ orthogonally with respect to $x^2 g$ and $P(\lambda)$ is given by \eqref{eq:stationarykleingordon}, then 
\[
P(\lambda) \in \mathrm{Bess}^{(\lambda)}_\nu(X;x,\phi_0).
\]
\end{lem}
\begin{proof}
Note that $\phi_0 = \phi \circ H$, where $H(x,y,t_0) = (x,y,t_0-h(x,y))$ in the notation of \eqref{eq:phi0}. As noted following \eqref{eq:phi0}, the differential of $h$ vanishes along $\partial M$. It remains to combine this observation with the expression \eqref{eq:Pg} for $P_g$ with respect to the product structure induced by $\phi$.
\end{proof}

\begin{lem} \label{lem:adsiselliptic} The following ellipticity properties hold.
	\begin{enumerate} \itemsep6pt
		\item If $\partial_t$ is timelike for $\gamma_0$, then $P(\lambda)$  is elliptic on $\partial X$ in the sense of Section \ref{subsect:ellipticity}.
		
		\item If $\partial_t$ and $dt$ are both timelike for $\gamma_0$, then $P(\lambda)$ is parameter-elliptic on $\partial X$ in the sense of Section \ref{subsect:semiclassicalbessel} with respect to any angular sector $\Lambda\subseteq \mathbb{C}$ disjoint from $\mathbb{R}\setminus 0$.
	\end{enumerate}
\end{lem}
\begin{proof}
	See \cite[Section 3.2]{vasy:2013} and \cite[Lemma 2.3]{gannot:2014:kerr}.
	\end{proof}

\section{Function spaces and mapping properties} \label{sect:functionmapping}

The purpose of this section is to define Sobolev-type spaces $\Sob^{s}$ based on the elementary derivatives $\TD$ and $|\TD|^2$, both on $\RNP$ and on a manifold with boundary. Finally, it is shown that Bessel operators act continuously between these spaces. 

The exposition is closest to that of \cite{warnick:2013:cmp}, where these ``twisted'' Sobolev spaces were first introduced in the context of aAdS geometry. The relationship between $\Sob^1$ and certain weighted Sobolev spaces was exploited both in \cite{warnick:2013:cmp} and also in the closely related study of asymptotically hyperbolic spaces in \cite{gonzalez:2010}. 

Throughout this section the spaces $L^2(\RNP)$ and $L^2(\RN)$ are defined with respect to ordinary Lebesgue measure, and $H^m(\RN)$ denotes the standard Sobolev space of order $m$ on $\RN$. Here $\RN$ will always denote the boundary $\partial \RNP$. The notation $\Sob^0(\RNP) := L^2(\RNP)$ is also frequently used.

\subsection{The weighted space $H^1_\mu(\RNP)$} \label{subsect:weighted} Given $\mu \in \mathbb{R}$, let 
\[
H_\mu^1(\RNP) = \{ u \in {\mathscr{D}}'(\RNP): x^{\frac{\mu}{2}} D^\alpha u \in L^2(\RNP) \text{ for } |\alpha| \leq 1 \},
\]
which is a Hilbert space under the norm 
\[
\| u \|^2_{H^1_\mu(\RNP)} = \sum_{|\alpha| \leq 1} \| x^{\frac{\mu}{2}} D^\alpha u \|^2_{L^2(\RNP)}.
\]
Furthermore, let $\dot{H}_\mu^1(\RNP)$ denote the closure of $C_c^\infty(\RNP)$ in $H_\mu^1$. These spaces are well studied; see Lions \cite{lions:1961} or Grisvard \cite{grisvard:1963} for example.
\begin{lem} \label{lem:weightedensity} The following hold for $\mu \in \mathbb{R}$.
\begin{enumerate} \itemsep6pt
\item If $|\mu| < 1$, then $C_c^\infty(\OL{\RNP})$ is dense in $H^1_\mu(\RNP)$.
\item If $|\mu| \geq 1$, then $H^1_\mu(\RNP)= \dot{H}^1_\mu(\RNP)$.
\end{enumerate}
\end{lem}
\begin{proof}
Proofs of these facts may be found in \cite{grisvard:1963,lions:1961}.
\end{proof} 

Given a Hilbert space $E$, let $H_\mu^1(\RP;E)$ denote the Hilbert space of $E$-valued distributions $u \in {\mathscr{D}}'(\RP;E)$ such that
\[
x^{\frac{\mu}{2}} u \in L^2(\RP;E), \quad x^{\frac{\mu}{2}} u' \in L^2(\RP;E),
\]
equipped with obvious norm. The Sobolev embedding theorem in this setting, \cite[Proposition 1.1']{grisvard:1963}, is
\[
H^1_\mu(\RP;E) \hookrightarrow C^0(\OL{\RP};E), \quad \mu < 1,
\]
so $u \mapsto u(0)$ is continuous $H^1_\mu(\RP;E) \rightarrow E$. Taking $E = L^2(\RN)$, it follows that when $\mu < 1$, any $u \in H^1_\mu(\RNP)$ admits a trace
\begin{equation} \label{eq:L2trace}
u \mapsto u|_{\RN} \in L^2(\RN).
\end{equation}
Furthermore, the kernel of $u \mapsto u|_{\RN}$ is $\dot{H}^1_\mu(\RNP)$ (see \cite[Proposition 1.2]{grisvard:1963}). The next lemma improves upon the regularity of this restriction.

\begin{lem} \label{lem:trace} If $|\mu| <1$, then the restriction 
\[
u \mapsto u|_{\RN}, \quad u \in C_c^\infty(\OL{\RNP})
\]
extends uniquely to continuous map $\gamma: H_\mu^1(\RNP) \rightarrow H^{(1-\mu)/2}(\RN)$. Furthermore, $\gamma$ admits a continuous right inverse.
\end{lem}
\begin{proof}
By the Sobolev embedding, any $\varphi \in C_c^\infty(\OL{\RP}) \subseteq H^1_\mu(\RP)$ admits an estimate of the form
\[
| \varphi(0) |^2 \leq C \int_{\RP} x^{\mu} \left( | \varphi |^2 + | \varphi' |^2 \right) dx.
\]
Apply this inequality to the function $\varphi(sx)$, and then choose $s$ (depending on $\varphi$) satisfying
\[
\int_{\mathbb{R}} x^{\mu} |\varphi|^2 \, dx = s^2 \int_{\mathbb{R}} x^{\mu} |\varphi'|^2 \, dx.
\]
This yields the estimate
\begin{equation} \label{eq:interpinequality}
| \varphi(0) |^2 \leq 2 C \left( \int_{\RP} x^{\mu} | \varphi|^2 dx \right)^{(1-\mu)/2}\left( \int_{\RP} x^{\mu} | \varphi' |^2 dx \right)^{(1+\mu)/2}.
\end{equation}
Now consider $u \in H^1_\mu(\RNP)$ and let $\hat{u}(q)$ denote its Fourier coefficients, where $q \in \mathbb{Z}^{n-1}$. It suffices to apply the inequality \eqref{eq:interpinequality} to $\hat{u}(q)$, which lies in $H^1_\mu(\RP)$ for each $q \in \mathbb{Z}^{n-1}$. Multiplying \eqref{eq:interpinequality} by $\left< q \right>^{1-\mu}$ and summing over all $q$, it follows that
\[
\| \gamma u \|_{H^{(1-\mu)/2}} \leq C \| u \|_{H^1_\mu(\RP)}.
\]
When $|\mu| < 1$, the unique continuation of $\gamma$ follows from the density of $C_c^\infty(\OL{\RNP})$ in $H^1_\mu(\RNP)$. That $\gamma$ admits a right inverse is also straightforward; see Lemma \ref{lem:tracelift} for a closely related result.
\end{proof}
The trace $u \mapsto \gamma u$ defined in Lemma \ref{lem:trace} agrees with the restriction given by \eqref{eq:L2trace} since they both agree on the dense set $C_c^\infty(\OL{\RNP})$.

\subsection{The space $\Sob^1(\RNP)$} \label{subsect:sob1} Given $\nu \in \mathbb{R}$, define
\[
\Sob^1(\RNP) = \{ u \in {\mathscr{D}}'(\RNP): \TD^j D_y^\alpha u \in L^2(\RNP) \text{ for $j + |\alpha| \leq 1$} \},
\]
where $\TD^j D_y^\alpha u$ is taken in the sense of distributions on $\RNP$; then $\Sob^1(\RNP)$ is a Hilbert space when equipped with the norm
\[
\| u \|^2_{\Sob^1(\RNP)} = \sum_{j + |\alpha| \leq 1} \| \TD^j D_y^\alpha u \|^2_{L^2(\RNP)}.
\]
The space $\Sob^1_*(\RNP)$ is defined analogously by replacing $\TD$ with its formal adjoint $\TD^*$. Let $\dot{\Sob}^1(\RNP)$ denote the closure of $C_c^\infty(\RNP)$ in $\Sob^1(\RNP)$, and similarly for $\dot{\Sob}^1_*(\RNP)$.
\begin{lem} \label{lem:usualsobolev}
	If $\nu \neq 0$, then $\dot{\Sob}^1(\RNP) = \dot{H}^1(\RNP)$ with an equivalence of norms.
\end{lem}
\begin{proof}
This can be deduced from Hardy's inequality 
\[
(1/4)\| x^{-1} u \|^2_{L^2(\RNP)} \leq	\| D_x u \|^2_{L^2(\RNP)},
\] 
valid for $u \in C_c^\infty(\RNP)$, and the density of $C_c^\infty(\RNP)$ in both spaces.
	\end{proof}

The basic observation concerning $\Sob^1(\RNP)$ is that the map ${\mathscr{D}}'(\RNP) \rightarrow {\mathscr{D}}'(\RNP)$ given by $u \mapsto x^{\nu-1/2}u$ restricts to an isometric isomorphism
\[
\Sob^1(\RNP) \rightarrow H^1_{1-2\nu}(\RNP).
\]
It follows from Lemma \ref{lem:weightedensity} that $x^{1/2-\nu}C^\infty(\OL{\RNP})$ is dense in $\Sob^1(\RNP)$ if $0 < \nu < 1$, and $\Sob^1(\RNP) = \dot{\Sob}^1(\RNP)$ if $\nu \geq 1$. Using Lemma \ref{lem:trace}, it is also possible to define weighted traces of $\Sob^1(\RNP)$ functions, as will be explained in Section \ref{subsect:traces}.

\subsection{The space $\Sob^2(\RNP)$} Given $\nu > 0$, define
\[
\Sob^2(\RNP) = \{ u \in \Sob^1(\RNP): \TD u \in \mathcal{H}^1_*(\RNP), \text{ and } D_y^\alpha u \in \Sob^1(\RNP) \text{ for } |\alpha| \leq 1 \}.
\] 
Then $\Sob^2(\RNP)$ becomes a Hilbert space when equipped with the norm
\begin{equation} \label{eq:H2nor}
\| u \|^2_{\Sob^2(\RNP)} = \| |\TD|^2 u \|^2_{L^2(\RNP)} + \sum_{|\alpha| \leq 1} \| D_y^\alpha u \|^2_{\Sob^1(\RNP)}.
\end{equation}
Although $x^{1/2-\nu}C_c^\infty(\RNP)$ is dense in $\Sob^1(\RNP)$ when $0 < \nu < 1$, this is not the case for $\Sob^2(\RNP)$. In fact, $x^{1/2-\nu}C_c^\infty(\OL{\RNP})$ is not contained in $\Sob^2(\RNP)$ unless $\nu = 1/2$. An appropriate dense space of smooth functions is defined in Section \ref{subsect:traces}. 

\subsection{Weighted traces} \label{subsect:traces}

It follows from Lemma \ref{lem:trace} that the weighted restriction
\[
u \mapsto x^{\nu- 1/2}u|_{\RN}, \quad u \in x^{1/2-\nu}C_c^\infty(\OL{\RNP})
\]
extends uniquely to a continuous map $\gamma_- : \Sob^1(\RNP) \rightarrow H^\nu(\RN)$, and furthermore, if $0 < \nu < 1$, then $\gamma_-$ admits a continuous right inverse. Similarly, there exists a weighted restricted 
\[
\gamma_-^* : \mathcal{H}^1_*(\RNP) \rightarrow H^{1-\nu}(\RN),
\]
given by $\gamma_-^* u = x^{1/2 -\nu}u|_{\RN}$. However, note that $\gamma_-^*$ is now defined for $\nu < 1$. Indeed, $\Sob^1_*(\RNP)$ is isomorphic to $H^1_{2\nu -1}(\RNP)$, and the trace on the latter space is only defined for $2\nu -1 < 1$. Since $u \in \Sob^2(\RNP)$ implies $\TP u \in \Sob^1_*(\RNP)$, there exists a second trace
\[
\gamma_+ : \mathcal{H}^{2}(\RNP) \rightarrow H^{1-\nu}(\RN)
\]
given by the composition $\gamma_+ = \gamma_-^* \circ \TP$. The trace $\gamma_+$ therefore exists for $0 < \nu < 1$.

\begin{defi} \label{defi:cnu}
Given $\nu > 0$, let $\Cnu$ denote the following spaces of functions.
\begin{enumerate} \itemsep6pt
\item If $0 < \nu < 1$, then $\Cnu$ consists of $u \in C^\infty(\RNP)$ of the form
\begin{equation} \label{eq:Cnu}
u(x,y) = x^{1/2-\nu}u_-(x^2,y) + x^{1/2+\nu}u_+(x^2,y),
\end{equation}
where $u_\pm \in C_c^\infty(\OL{\RNP})$.
\item If $\nu \geq 1$, then $\Cnu = C_c^\infty(\RNP)$.
\end{enumerate}
\end{defi}

Note that $\Cnu$ is contained in $\Sob^s(\RNP)$ for each $s=0,1,2$, but $\Cnu$ is not contained in $x^{1/2-\nu}C_c^\infty(\OL{\RNP})$ unless $\nu = 1/2$. On the other hand, traces of $u\in \Cnu$ are still easily computed from the definitions:

\begin{lem} \label{lem:directtrace}
If $0 < \nu < 1$ and $u \in \Cnu$ satisfies \eqref{eq:Cnu}, then
\begin{equation} \label{eq:tracevalues}
\gamma_- u = u_-(0,\cdot), \quad \gamma_+ u = 2\nu u_+(0,\cdot).
\end{equation}
\end{lem}
\begin{proof}
If $u \in \Cnu$ satisfies \eqref{eq:Cnu}, then $x^{\nu - 1/2}u(x,y) = u_-(x^2,y) + x^{2\nu}u_+(x^2,y)$. Since this function is continuous on $\RP$ with values in $C^\infty(\RN)$, it follows that $\gamma_- u = u_-(0,\cdot)$ (see the remark after Lemma \ref{lem:trace}). A similar argument shows that $\gamma_+ u = 2\nu u_+(0,\cdot)$.
\end{proof}

\begin{lem} \label{lem:traceproperty}
If $\nu > 0$, then $\Cnu$ is dense in $\Sob^s(\RNP)$ for $s = 0,1,2$.
\end{lem}
\begin{proof}
A proof is provided in Appendix \ref{appendix1}.
\end{proof}

\begin{prop}
If $0 < \nu < 1$, then there exist unique continuous maps
\[
\gamma_\mp : \Sob^{s}(\RNP) \rightarrow H^{s-1\pm\nu}(\RN)
\]
such that if $u \in \Cnu$ satisfies \eqref{eq:Cnu}, then $\gamma_- u = u_-(0,\cdot)$ and $\gamma_+ u = 2\nu u_+(0,\cdot)$. Here $\gamma_-$ is defined for $s=1,2$, while $\gamma_+$ is only defined for $s=2$.
\end{prop}
\begin{proof}
Combining Lemma \ref{lem:traceproperty} with Lemma \ref{lem:directtrace} shows that the map 
\[
u \mapsto u_-(0,\cdot), \quad u \in \Cnu
\] 
admits a unique extension $\Sob^{s}(\RNP) \rightarrow H^{\nu}(\RN)$ for $s = 1,2$. The additional regularity $\gamma_- u \in H^{1+\nu}(\RN)$ for $u \in \Sob^2(\RNP)$ follows from the equality $\gamma_\pm \partial_y^\alpha u = \partial_y^\alpha \gamma_\pm u$ for $u\in \Cnu$ and each multiindex $\alpha$. Similarly, the map 
\[
u \mapsto 2\nu u_+(0,\cdot), \quad u \in \Cnu
\] 
admits a unique extension $\Sob^2(\RNP) \rightarrow H^{1-\nu}(\RN)$.
\end{proof}

\subsection{Dual spaces}
Throughout, $\Sob^0(\RNP) = L^2(\RNP)$ is identified with its own antidual $\Sob^0(\RNP)'$ via the Riesz representation. Given $s = 1,2$, let 
\[
\Sob^{-s}(\RNP) = \Sob^{s}(\RNP)'
\]
denote the corresponding antiduals. Since the inclusion $\iota : \Sob^{s}(\RNP) \hookrightarrow \Sob^0(\RNP)$ is dense, $\Sob^0(\RNP)$ is identified with a dense subspace of $\Sob^{-s}(\RNP)$ via the map $\iota^* : \Sob^{0}(\RNP) \hookrightarrow \Sob^{-s}(\RNP)$. Thus if $s \geq 0$ and $u, v \in \Sob^{s}(\RNP)$, then the image $\iota^* u$ in $\Sob^{-s}(\RNP)$ acts on $v$ via the $\Sob^0(\RNP)$ pairing
\[
\iota^*u(v) = \langle u, v\rangle_{\RNP}.
\]
Because $\Sob^{s}(\RNP)$ is dense in $\Sob^{-s}(\RNP)$, there is no ambiguity in using the notation     
\[
\langle f, v \rangle_{\RNP}:= f(v), \quad f \in \Sob^{-s}(\RNP),\;v \in \Sob^{s}(\RNP)
\]
in general.

\subsection{A Fourier characterization} \label{subsect:fourier} Given $s = 0,1,2$, any $u \in \Sob^{s}(\RNP)$ has well defined Fourier coefficients 
\[
\hat{u}(q) = (2\pi)^{-(n-1)/2} \int_{[-\pi,\pi]^{n-1}} e^{-i\left<q,y\right>} u(\cdot,y) \, dy, \quad q \in \mathbb{Z}^{n-1}.
\]
It is easily seen $\hat{u}(q) \in \Sob^{s}(\RP)$ for each fixed $q \in \mathbb{Z}^{n-1}$.

This may be extended uniquely by duality: given $f \in \mathcal{H}^{-s}(\RNP)$, let $\hat{f}(q) \in \Sob^{-s}(\RP)$ denote the functional
\begin{equation} \label{eq:dualfourier}
\langle \hat{f}(q) , v \rangle_{\RNP} = (2\pi)^{-(n-1)/2} \langle f, e^{i\left<q,y\right>} v \rangle_{\RNP},
\end{equation}
where $v \in \Sob^s(\RP)$. Given $\tau > 0$ and $u \in \Cnu$, let
\begin{equation} \label{eq:dilationdef}
(S_\tau u)(x,y) = u(\tau x,y)
\end{equation}
denote the action of dilation in the normal variable. This clearly extends to a bounded map $S_{\tau} : \Sob^{s}(\RNP) \rightarrow \Sob^{s}(\RNP)$ for $s = 0,1,2$. Furthermore, $S_{\tau}$ may be extended uniquely to $\Sob^{-s}(\RNP)$ by duality: given $f \in \mathcal{H}^{-s}(\RNP)$, define
\[
\left< S_\tau f ,v \right>_{\RNP} = \tau^{-1} \left<f , S_{\tau^{-1}} v \right>_{\RNP}
\]  
for $v \in \Sob^{s}(\RNP)$.
\begin{lem} \label{lem:fouriernorm}
Given $s=0,\pm 1,\pm 2$,
\[
\| u \|^2_{\Sob^{s}(\RNP)} = \sum_{q\in \mathbb{Z}^{n-1}} \left<q\right>^{2s-1} \| S_{\left<q\right>^{-1}} \hat{u}(q) \|^2_{\Sob^{s}(\RP)}.
\]
for each $u \in \Cnu$.
\end{lem}
\begin{proof}
When $s\geq 0$ this follows from Parseval and Fubini's theorems. When $s < 0$, the proof is a simple modification of the argument in \cite[Lemma 2.3.1]{kozlov:1997}.
\end{proof}

\subsection{The space $\TSob^{s}(\RNP)$} \label{subsect:tildespace} If $\UL t = (t_1,\ldots, t_k)$, define
\[
H^{\UL t}(\RN) := \prod_{j = 1}^k H^{t_k}(\RN).
\]
Keeping this notation in mind, let $\UL\nu = (1 - \nu, 1+\nu)$ and then set
\begin{equation} \label{eq:gamma}
\UL\gamma = \colvec{\gamma_-}{\gamma_+}.
\end{equation}
Following \cite{kozlov:1997,roitberg:1996} in the smooth setting, define the following spaces for $0 < \nu < 1$. Given $s = 0,\,\pm 1,\,\pm 2$, let $\TSob^{s}(\RNP)$ denote the set of all
\[
(u, \phi_-, \phi_+ ) \in \Sob^{s}(\RNP) \times H^{s - \UL\nu}(\RN)
\]
such that
\begin{enumerate} \itemsep6pt
\item $\phi_- = \gamma_- u$ and  $\phi_+ = \gamma_+ u$ if $s =2$,
\item $\phi_- = \gamma_- u$ and $\phi_+$ is arbitrary if $s = 1$,
\item $\phi_\pm$ are arbitrary if $s \leq 0$.
\end{enumerate}
A typical element of $\TSob^{s}(\RNP)$ will be denoted $(u, \underline{\phi})$, where $\underline{\phi} = (\phi_-, \phi_+)$. The norm of $(u,\UL{\phi})$ is given by 
\[
\| (u,\UL{\phi}) \|^2_{\TSob^{s}(\RNP)} = \| u \|^2_{\Sob^s(\RNP)} + \| \UL{\phi} \|^2_{H^{s-\UL{\nu}}(\RN)}.
\]
If $s = 2$, then $u \mapsto (u, \UL{\gamma} u)$ provides an isomorphism 
\[
\Sob^{2}(\RNP) \rightarrow \TSob^{2}(\RNP).
\]
On the other hand, if $s \leq 1$, then the two spaces $\Sob^{s}(\RNP),\, \TSob^{s}(\RNP)$ are fundamentally different.

\begin{lem} \label{lem:tildedensity} Let $0 < \nu < 1$. For each $s=0,\pm 1, \pm 2$, the set  
\[
\{ (u, \UL{\gamma}u) : u \in \Cnu\}
\]
is dense in $\TSob^{s}(\RNP)$. 
\end{lem}
\begin{proof}
It suffices to prove this for $s\geq 0$, since $\TSob^0(\RNP)$ is dense in $\TSob^{s}(\RNP)$ if $s < 0$. Given $(u,\UL{\phi})\in \TSob^{s}(\RNP)$, choose $u_n \in \Cnu$ and $\phi_{\pm,n} \in C^\infty(\mathbb{R}^n)$ such that 
\[
u_n \rightarrow u \text{ in }\Sob^s(\RNP), \quad {\phi_{\pm,n}} \rightarrow \phi_{\pm} \text{ in } H^{s-\UL{\nu}}(\RN).
\] 
Let $\chi \in C_c^\infty(\OL{\RP})$ satisfy $\chi = 1$ near $x=0$, and define
\[
u_{n,\varepsilon} = u_n - \left( x^{1/2-\nu}(\gamma_- u_n - \phi_{-,n}) + (2\nu)^{-1} x^{1/2+\nu}(\gamma_+ u_n - \phi_{+,n}) \right)\chi(\varepsilon^{-1}x).
\]
Clearly $u_{n,\varepsilon} \in \Cnu$ and $\gamma_\pm u_{n,\varepsilon} = \phi_{\pm,n}$. Furthermore, since $s\geq 0$, it is easy to check that $u_{n,\varepsilon} \rightarrow u_n$ in $\Sob^{s}(\RNP)$ for $n$ fixed and $\varepsilon \rightarrow 0$. Thus it is possible find a sequence $\varepsilon_n \rightarrow 0$ such that $u_{n,\varepsilon_n} \rightarrow u$ in $\Sob^s(\RNP)$ as $n \rightarrow \infty$, and hence
\[
(u_{n,\varepsilon_n},\UL{\gamma} u_{n,\varepsilon_n}) \rightarrow (u,\UL{\phi})
\]
in $\TSob^{s}(\RNP)$.
\end{proof}

Recall from Section \ref{subsect:fourier} the dilation $S_\tau$ given by \eqref{eq:dilationdef}. Note that
\[
(\gamma_- \circ S_\tau) u = \tau^{1/2-\nu} \gamma_- u, \quad (\gamma_+ \circ S_\tau) u = \tau^{1/2+\nu} \gamma_+ u
\]
for each $\tau > 0$ and $u\in \Cnu$. Thus $S_\tau$ may be extended uniquely to $\TSob^{s}(\RNP)$ by defining
\[
S_\tau (u,\UL{\phi}) = (S_\tau u , \tau^{1/2-\nu} \phi_-, \tau^{1/2+\nu} \phi_+ ).
\]
It follows from Lemma \ref{lem:fouriernorm} and the usual Fourier characterization of $H^m(\RN)$ that
\[
\| (u,\UL{\phi}) \|^2_{\TSob^{s}(\RNP)} = \sum_{q\in \mathbb{Z}^{n-1}} \left<q\right>^{2s-1} \| S_{\left<q\right>^{-1}} (\hat{u}(q),\UL{\hat{\phi}}(q)) \|^2_{\TSob^{s}(\RP)}
\]
for each $(u,\UL{\phi}) \in \TSob^s(\RNP)$.

\subsection{Parameter-dependent norms} \label{subsect:semiclassicalnorms}

When considering the action of parameter dependent Bessel operators, one must consider modified norms on the spaces defined so far. Given $s = 0,1,2$ and $u \in \Sob^{s}(\RNP)$, let
\[
\VERT u \VERT^2_{\Sob^{s}(\RNP)} = \sum_{j=0}^s |\lambda|^{2(s-j)} \| u \|^2_{\Sob^{j}(\RNP)}.
\]
Furthermore, if $f \in \Sob^{-s}(\RNP)$, let
\[
\VERT f \VERT_{\Sob^{-s}(\RNP)} = \sup \{ |\left< f, u \right>_{\RNP}|: \VERT v \VERT_{\Sob^s(\RNP)} = 1\}.
\]
Recall the standard parameter-dependent norms $\VERT v \VERT_{\Sob^m(\RN)}$ on $H^m(\RN)$ as in Section \ref{subsect:conventions}. Given $(u,\UL{\phi}) \in \TSob^{s}(\RNP)$, set
\[
\VERT (u,\UL{\phi}) \VERT^2_{\TSob^{s}(\RNP)} = \VERT u \VERT^2_{\Sob^s(\RNP)} + \VERT \UL{\phi} \VERT^2_{H^{s-\UL{\nu}}(\RN)}. 
\]
These parameter-dependent norms have the property that there exists $C>0$ independent of $\lambda$ such that
\[
\VERT u \VERT_{\Sob^{s-1}(\RNP)} \leq C |\lambda|^{-1} \VERT u \VERT_{\Sob^s(\RNP)}, \quad \VERT (u,\UL{\phi}) \VERT_{\TSob^{s-1}(\RNP)} \leq C |\lambda|^{-1} \VERT (u,\UL{\phi}) \VERT_{\TSob^s(\RNP)}
\]
for $u \in \Sob^s(\RNP)$ and $(u,\UL{\phi}) \in \TSob^s(\RNP)$, respectively.

\subsection{Mapping properties} \label{subsect:mappingpropertiesonRNP}

In this section, mapping properties of Bessel operators on $\RNP$ are examined. The analogues of Green's formulas are established, which allow the extension of $P$ to spaces with low regularity. Recall from Section \ref{subsect:bessel} that $P \in \mathrm{Bess}_\nu(\RNP)$ means that
\begin{equation} \label{eq:besselmappingrnp}
P = |\TD|^2 + B(x,y,D_y)\TD + A(x,y,D_y),
\end{equation}
where $B \in \mathrm{Diff}^1(\RN)$ and  $A \in \mathrm{Diff}^2(\RN)$ depend smoothly on $x \in \OL{\RP}$ and $B(0,y,D_y)= 0$. Throughout this section, assume that the coefficients of $A,B$ are constant outside a compact subset of $\OL{\RNP}$. The boundedness of each term in $P$ will be examined individually. 

Before proceeding, it is necessary to consider certain multipliers of $\Sob^s(\RNP)$ when $s\geq 0$. The commutation relations
\begin{equation} \label{eq:commutatorformulas}
[ \TP, \varphi ] = \partial_x \varphi = [ \TP^*, \varphi ], \quad [|\TD|^2, \varphi] = -\partial_x^2 \varphi - 2(\partial_x \varphi)\partial_x
\end{equation}
will be used throughout the following lemma.

\begin{lem} \label{lem:leibniz}
	Suppose that $\varphi \in C^\infty(\overline{\RNP})$ is bounded along with all of its derivatives, and consider multiplication by $\varphi$ as a continuous map ${\mathscr{D}}'(\RNP) \rightarrow {\mathscr{D}}'(\RNP)$. 
	\begin{enumerate} \itemsep6pt
		\item For $s=0,1$, multiplication by $\varphi$ restricts to a continuous map 
		\[
		\Sob^{s}(\RNP) \rightarrow \Sob^{s}(\RNP).
		\]
		\item If $\partial_x \varphi|_{\RN}= 0$, then multiplication by $\varphi$ restricts to a continuous map
		\[
		\Sob^{2}(\RNP) \rightarrow \Sob^{2}(\RNP).
		\]
	\end{enumerate}
	In either of these two cases,
	\begin{equation} \label{eq:multiplierestimate}
	\| \varphi u \|_{\Sob^s(\RNP)} \leq \| \varphi \|_{C^0(\OL{\RNP})} \| u \|_{\mathcal{H}^{s}} + C_s\| u \|_{\Sob^{s-1}(\RNP)},
	\end{equation}
	where $C_s \geq 0 $ depends on the first $s$ derivatives of $\varphi$, and $C_0 = 0$.
\end{lem}
\begin{proof}
	The continuity statement is obvious for $s=0$. For $s = 1$ it follows from the first commutator formula \eqref{eq:commutatorformulas}. When $s=2$, the additional condition $\partial_x \phi |_{\RN} = 0$ is needed to ensure that
	\[
	u \mapsto (\partial_x \varphi)\partial_x u
	\]
	is bounded $\Sob^{1}(\RNP) \rightarrow \Sob^0(\RNP)$: the vanishing of $\partial_x \phi$ at the boundary implies $(\partial_x \varphi)\partial_x =  (\partial_x \varphi)\TP$ modulo multiplication by a smooth function, which acts continuously by the first part. The estimate \eqref{eq:multiplierestimate} follows as well from \eqref{eq:commutatorformulas}.
\end{proof}

\begin{rem}
	Lemma \ref{lem:leibniz} result may also be extended to $\TSob^s(\RNP)$ by defining 
	\[
	\varphi(u,\UL{\phi}) := (\varphi u, \varphi|_{\RN}\UL{\phi}),
	\]
	and using that standard Sobolev spaces on $\RN$ are closed under multiplication by smooth functions.
\end{rem}
\begin{rem}
	The hypotheses of Lemma \ref{lem:leibniz} can not be improved when $s=2$, in the sense that $\Sob^2(\RNP)$ is not closed under multiplication by arbitrary $C_c^\infty(\OL{\RNP})$ functions. On the other hand, as a special case of Lemma \ref{lem:leibniz}, if $\varphi \in C_c^\infty(\OL{\RNP})$ is constant in a neighborhood of $\RN$, then $\Sob^s(\RNP)$ is closed under multiplication by $\varphi$ for each $s=0,1,2$. 
\end{rem}

Now consider the term $|\TD|^2$ in \eqref{eq:besselmappingrnp} which is clearly bounded
\[
|\TD|^2 : \Sob^2(\RNP) \rightarrow \Sob^0(\RNP).
\]
The distinction between $0< \nu < 1$ and $\nu \geq 1$ plays an important role when extending this action.  Suppose that $0 < \nu < 1$, and let $J$ denote the usual symplectic matrix,
\[ 
J = \left( \begin{array}{cc}
0 & 1 \\
-1 & 0 \\
\end{array} \right).
\]
Then the following formulae are valid for each $u,v \in \mathcal{F}_\nu$:
\begin{align}
\left< |\TD|^2 u,v \right>_{\RNP} &= \left< u, |\TD|^2 v \right>_{\RNP} + \left< \UL{\gamma}u,J\UL{\gamma}v \right>_{\RN}, \label{eq:TD1} \\
\left< |\TD|^2 u,v \right>_{\RNP} &= \left< \TD u,\TD v \right>_{\RNP} - \left< \gamma_+ u, \gamma_- v \right>_{\RN} \label{eq:TD2}.
\end{align}
Since $\Cnu$ is dense, \eqref{eq:TD1} is valid for $v \in \Sob^2(\RNP)$, and \eqref{eq:TD2} is valid for $v \in \Sob^1(\RNP)$.
\begin{lem} \label{lem:TDgreen1}
	Let $0 < \nu < 1$ and $s=0,1,2$. Then there exists $C>0$ such that
	\[
	\| |\TD|^2u\|_{\Sob^{s-2}(\RNP)} \leq C \| (u,\UL{\gamma}u) \|_{\TSob^s(\RNP)}
	\]
	for each $u \in \Cnu$.
\end{lem}
\begin{proof}
	For $s=2$ this follows since the norms $\| u \|_{\Sob^2(\RNP)}$ and $\| (u,\UL{\gamma}u)\|_{\TSob^2(\RNP)}$ are equivalent for each $u \in \Cnu$. The case $s=1$ follows from \eqref{eq:TD2}, and the case $s=0$ follows from \eqref{eq:TD1}.
\end{proof}
As a consequence of Lemma \ref{lem:TDgreen1}, the map $(u,\UL\gamma u) \mapsto |\TD|^2 u$ with $u \in \Cnu$ admits a unique extension as a bounded operator
\[
|\TD|^2 : \TSob^s(\RNP) \rightarrow \Sob^{s-2}(\RNP)
\]
for $s=0,1,2$ and $0 < \nu < 1$.
The situation is simpler when $\nu \geq 1$, since in that case $\Cnu = C_c^\infty(\RNP)$ is dense in $\Sob^s(\RNP)$. The analogues of \eqref{eq:TD1}, \eqref{eq:TD2} are given by
\begin{align}
\left< |\TD|^2 u,v \right>_{\RNP} &= \left< u, |\TD|^2 v \right>_{\RNP}, \label{eq:TD3} \\
\left< |\TD|^2 u, v \right>_{\RNP} &= \left< \TD u, \TD v \right>_{\RNP} \label{eq:TD4},
\end{align}
valid for each $u, v\in \Cnu$. The analogue of Lemma \ref{lem:TDgreen1} is the following.
\begin{lem} \label{lem:TDgreen2}
	Let $\nu \geq 1$ and $s=0,1,2$. Then there exists $C>0$ such that
	\[
	\| |\TD|^2 u \|_{\Sob^{s-2}(\RNP)} \leq C \| u \|_{\Sob^s(\RNP)}
	\]
	for each $u \in \Cnu$.
\end{lem}
From Lemma \ref{lem:TDgreen2}, it follows that the map $u \mapsto |\TD|^2u$ with $u \in \Cnu$ admits a unique continuous extension as a bounded operator $|\TD|^2 : \Sob^{s}(\RNP) \rightarrow \Sob^{s-2}(\RNP)$ for $s=0,1,2$ and $\nu \geq 1$.

Next consider a typical term in $B\TD$. Such a term may be written as $b(x,y) D_y^\beta \TD$, where $b \in x C^\infty(\OL{\RNP})$ is constant for $x$ large and $|\beta| \leq 1$. The following result holds for all $\nu >0$, since there are no boundary terms when integrating by parts.

\begin{lem} \label{lem:Bbounded}
	Suppose that $b \in x C^\infty(\OL{\RNP})$ is constant for $x$ large and $|\beta| \leq 1$. Then
	\[
	b D_y^\beta \TD : \Sob^{s}(\RNP) \rightarrow \Sob^{s-|\beta|-1}(\RNP) 
	\]
	is bounded for each $s =0,1,2$. Furthermore, there exists $c >0$ depending on $s,\beta$ and $C \geq 0$ depending on $b,s,\beta,r$ such that
	\begin{equation} \label{eq:Bbounded}
	\| b D_y^{\beta} \TD u \|_{\Sob^{s-|\beta|-1}(\RNP)} \leq c r\|b\|_{C^1(\OL{\RNP})}\| u \|_{\Sob^{s}(\RNP)} + C\| u \|_{\Sob^{s-1}(\RNP)}.
	\end{equation}
	for each $u\in \Sob^{s}(\RNP)$ such that $\supp u \subseteq \{ 0 \leq x \leq r \}$.	
\end{lem}
\begin{proof}
	The boundedness result is clear for $s=2$. For $s=0,1$, it follows from the same considerations as in Lemma \ref{lem:closedunderadjoint}: define $B = bD_y^\beta$, and note that $B = B_1 x = xB_1$ where $B_1$ is smooth up to $x=0$. Thus
	\[
	B\TD = \TD^* B + i(1-2\nu)B_1 + [B,D_x]. 
	\]
	Then for each $u,v \in \Cnu$,
	\[
	\left< B\TD u, v \right>_{\RN} = \left< \TD u, B^* v \right>_{\RN} = \left< u, B^* \TD v - i(1-2\nu)B_1^* v + [B,D_x]^* v  \right>_{\RNP}.
	\]
	The first equality implies boundedness for $s=1$, while the second implies boundedness for $s=0$.
	
	Similarly, \eqref{eq:Bbounded} clearly holds for $s=2$. To prove the other cases, begin by writing $b = xb_1$, where $b_1$ is smooth up to $x=0$. Also define $q = [D_y^\beta,b]$ and $q = xq_1$, so that $q_1$ is smooth up to $x=0$ (and vanishes if $|\beta|=0$). To avoid a distracting proliferation of complex conjugates, assume that $q$ is real-valued.

	\begin{inparaenum}
		\item If $s=1$, then for $u, v \in \Cnu$,
	\[
	\left< bD_y^\beta D_\nu u, v \right>_{\RNP} = \left< b \TD u, D_y^\beta v \right>_{\RNP} - \left< u,  \TD ({q}v) + i(2\nu-1) q_1v \right>_{\RNP}.
	\]
	Thus
	\[
	\| bD_y^\beta \TD u \|_{\Sob^{s-|\beta|-1}(\RNP)} \leq \| b D_\nu u \|_{\Sob^0(\RNP)} + C \| u \|_{\Sob^0(\RNP)},
	\]
	whence the result follows by Lemma \ref{lem:leibniz}.
	
	\item Similarly for $s=0$, if $u, v \in \mathcal{F}_\nu$, then
	\begin{align*}
	\left< bD_y^\beta \TD  u, v \right>_{\RNP} = \left< b u, \TD D_y^\beta v \right>_{\RNP} &+ \left< [b,\TD]u -i(2\nu-1) b_1 u, D_y^\beta v \right> \\ &- \left< u, \TD(qv) + i(2\nu-1)q_1v \right>_{\RNP}.
	\end{align*}
	The first term gives
	\[
	|\left< bu , \TD D_y^\beta v \right>_{\RNP}| \leq \| bu\|_{\Sob^0(\RNP)}\| v \|_{\Sob^{1+|\beta|}(\RNP)},
	\]
	as desired. Now $[b,\TD] = i(\partial_x b)$, and the second term can be written as
	\[
	\left< u, i(\partial_x b - (2\nu-1)b_1 ) D_y^\beta v \right>,
	\]	
	which is bounded in absolute value by a constant times $\|u\|_{\Sob^{-1}(\RNP)}\| v \|_{\Sob^{1+|\beta|}(\RNP)}$ according to Lemma \ref{lem:leibniz}. Similarly, 
	\[
	|\left< u, i(2\nu-1)q_1v \right>_{\RNP}| \leq C \|u\|_{\Sob^{-1}(\RNP)}\| v \|_{\Sob^{1}(\RNP)}.
	\]
	Now $|\left< u, \TD(qv) \right>| \leq \| u\|_{\Sob^{-1}(\RNP)}\| \TD(qv) \|_{\Sob^1(\RNP)}$, so it remains to bound the term $\| \TD(qv) \|_{\Sob^1(\RNP)}$. For this, write
	\[
	\TD \TD q = (\TD x - i)\TD q_1 = (\TD^* x -2i\nu)\TD q_1  = x |\TD|^2q_1 - i(2\nu+1)\TD q_1.
	\]
	Using \eqref{eq:commutatorformulas}, 
	\[
	x|\TD|^2 {q_1} = xq_1|\TD|^2 - 2i x(\partial_x q_1)\TD- x(\partial_x^2 q_1)  + (2\nu-1)(\partial_x q_1),
	\]
	which is bounded $\Sob^2(\RNP) \rightarrow \Sob^0(\RNP)$. Terms of the form $D_{y}^\alpha \TD(qv)$ with $|\alpha| \leq 1$ are bounded similarly. Since $q$ vanishes for $|\beta| = 0$, this shows that
	\[
	\| \TD (q v) \|_{\Sob^{1}(\RNP)} \leq C\| v\|_{\Sob^{1+|\beta|}(\RNP)},
	\]
	thus completing the proof.
		\end{inparaenum}
	\end{proof}

\begin{rem} Lemma \ref{lem:Bbounded} implies that $bD_y^\beta \TD$ is also bounded 
	\[
	\TSob^s(\RNP) \rightarrow \Sob^{s-|\beta|-1}(\RNP)
	\]
	 since the projection $\TSob^s(\RNP) \rightarrow \Sob^s(\RNP)$ onto the first factor is continuous. 
\end{rem}

Finally, a typical term in the operator $A$ can be written as $a(x,y)D_y^\alpha$, where $|\alpha| \leq 2$ and $a \in C^\infty(\OL{\RNP})$ is constant outside a compact subset of $\OL{\RNP}$.
\begin{lem} \label{lem:Abounded}
	Suppose that $a \in C^\infty(\OL{\RNP})$ is constant for $x$ large. 
	\begin{enumerate} \itemsep6pt
		\item 	If $s=0,1$ and $|\alpha| \leq 2$, then $a D_y^\alpha : \Sob^{s}(\RNP) \rightarrow \Sob^{s-|\alpha|}(\RNP)$ is bounded. 
		\item If $s=0,1,2$ and $0 < |\alpha| \leq 2$, then $a D_y^\alpha : \Sob^{s}(\RNP) \rightarrow \Sob^{s-|\alpha|}(\RNP)$ is bounded.
	\end{enumerate}
	Furthermore, suppose that $a(0,p) = 0$ for $p\in\RN$. Then there exists $c>0$ depending on $s,\alpha$ and $C \geq 0$ depending on $a,s,\alpha,r$ such that in each of the above cases,
	\begin{equation} \label{eq:Abounded}
	\| a D_y^{\alpha} u \|_{\Sob^{s-|\alpha|}(\RNP)} \leq cr\|a\|_{C^1(\OL{\RNP})}\| u \|_{\Sob^{s}(\RNP)} + C \| u \|_{\Sob^{s-1}(\RNP)}
	\end{equation}
	for each $u\in \Sob^{s}(\RNP)$ such that $\supp u \subseteq \{ (x,y) \in \OL{\RNP}: |x| + |y-p| < r   \}$.	
\end{lem}
\begin{proof}
	\begin{inparaenum}
\item  First suppose that $s=0,1$. The boundedness result is clear if $s=1$ and $|\alpha| \leq 1$ or $s=0$ and $|\alpha| =0$. Otherwise, suppose that $s=1$ and $|\alpha| = 2$. Write $aD_y^\alpha = \sum_{|\gamma| = 1} D_y^\gamma A_{\gamma}$ for smooth tangential operators $A_{\gamma}(x,y,D_y)$ of order at most one. Then for each $u,v \in \Cnu$ and $|\gamma| = 1$,
	\[
	|\left< D_y^\gamma A_\gamma u, v \right>_{\RNP}| = | \left< A_{\gamma}u, D_y^\gamma v \right>_{\RNP} | \leq C\| u \|_{\Sob^1(\RNP)} \| v \|_{\Sob^1(\RNP)}.
	\]
	On the other hand, suppose that $s=0$. Then
	\[
	|\left< aD_y^\alpha u, v \right>_{\RNP} | = | \left< u, D_y^\alpha av \right>_{\RNP}| \leq C\| u \|_{\Sob^0(\RNP)} \| v \|_{\Sob^{|\alpha|}(\RNP)}
	\]
	for $1 \leq |\alpha| \leq 2$.
	
\item  The only case not handled above is $s=2$, in which case it follows from Lemma \ref{lem:leibniz} that $aD_y^\alpha$ is bounded $\Sob^2(\RNP)\rightarrow \Sob^{2-|\alpha|}(\RNP)$ provided $|\alpha| \neq 0$.
	
\item  The estimate \eqref{eq:Abounded} follows from the same arguments as in the first two parts of the proof.
\end{inparaenum}
	\end{proof}

To summarize the above discussion, write
$A = \sum_{|\alpha| \leq 2} a_{\alpha} D_y^\alpha$ (non uniquely) in the form
\[
A = \sum_{|\alpha| \leq 1} D_y^\alpha A_\alpha
\]
for some $A_\alpha \in \mathrm{Diff}^1(\RN)$ which depends smoothly on $x \in \RP$. Recall that $P^*$ is also a Bessel operator, according to Lemma \ref{lem:closedunderadjoint}. Then there are the two Green's formulas
\begin{equation}
\langle Pu , v \rangle_{\RNP} = \langle u, P^* v \rangle_{\RNP} + \langle \UL{\gamma} u, J\UL{\gamma}v \rangle_{\RN} \label{eq:green2} 
\end{equation} 
and
\begin{multline}
\langle Pu , v \rangle_{\RNP} =  \left< \TD u, \TD v \right>_{\RNP} + \left<  \TD u, B^*v \right>_{\RNP} \\ + \sum_{|\alpha| \leq 1} \left< A_\alpha u, D_y^\alpha v \right>_{\RNP}  - \left< \gamma_+ u, \gamma_- v \right>_{\RN} \label{eq:green1},
\end{multline}
valid for each $u, v \in \Cnu$.
\begin{lem} \label{lem:besselextension1} Let $0 < \nu < 1$ and $s=0,1,2$. Then there exists $C>0$ depending on $s$ such that
\[
\| Pu \|_{\Sob^{s-2}(\RNP)} \leq C \| (u,\UL{\gamma}u) \|_{\TSob^{s}(\RNP)}
\]
for each $u \in \Cnu$. Thus the map $(u,\UL{\gamma}u) \mapsto Pu$ with $u \in \Cnu$ admits a unique extension as a bounded operator
\[
P : \TSob^{s}(\RNP) \rightarrow \Sob^{s-2}(\RNP)
\]
for $s=0,1,2$ and $0 < \nu < 1$. When $s=0,1$, this extension is determined by \eqref{eq:green2}, \eqref{eq:green1}.
\end{lem}
\begin{proof}This is a direct consequence of Lemmas \ref{lem:TDgreen1}, \ref{lem:Bbounded}, and \ref{lem:Abounded}.
\end{proof}

The situation is simpler when $\nu \geq 1$: the analogues of \eqref{eq:green2}, \eqref{eq:green1} are given by
\begin{align}
&\langle Pu , v \rangle_{\RNP} = \langle u, P^* v \rangle_{\RNP}  \label{eq:green4} \\
&\langle Pu , v \rangle_{\RNP} =  \left< \TD u, \TD v \right>_{\RNP} + \left< \TD, B^*v \right>_{\RNP} + \sum_{|\alpha| \leq 1} \left< A_\alpha u, D_y^\alpha v \right>_{\RNP}\label{eq:green3},
\end{align}
valid for each $u, v \in \Cnu$. As before, \eqref{eq:green4} is in fact valid for $v \in \Sob^2(\RNP)$, while $\eqref{eq:green3}$ is valid for $v \in \Sob^1(\RNP)$.

\begin{lem} \label{lem:besselextension2}
Let $\nu \geq 1$ and $s = 0,1,2$. Then there exists $C>0$ such that
\[
\| Pu \|_{\Sob^{s-2}(\RNP)} \leq C \| u \|_{\Sob^{s}(\RNP)}
\]
for each $u \in \Cnu$. Thus the map $u \mapsto Pu$ with $u \in \Cnu$ admits a unique extension as a bounded operator
\[
P : \Sob^{s}(\RNP) \rightarrow \Sob^{s-2}(\RNP)
\]
for $s=0,1,2$ and $\nu \geq 1$. When $s=0,1$ this extension is determined by \eqref{eq:green4}, \eqref{eq:green3}. The action of $P$ on $\Sob^s(\RNP)$ is simply the restriction of $P : {\mathscr{D}}'(\RNP) \rightarrow {\mathscr{D}}'(\RNP)$ to $\Sob^s(\RNP)$. 
\end{lem}
\begin{proof}
	This is a direct consequence of Lemmas \ref{lem:TDgreen2}, \ref{lem:Abounded}, and \ref{lem:Bbounded}.
	\end{proof}
	
Suppose that $0< \nu < 1$. If $s=0,1$, then an element $f \in \Sob^{s-2}(\RNP)$ is not uniquely determined by a distribution in $\mathscr{D}'(\RNP)$. On the other hand, $f$ may certainly be restricted to a functional on $\dot{\Sob}^{s}(\RNP)$, which is determined uniquely by a distribution since $C_c^\infty(\RNP)$ is dense in this space by definition. Given $s=0,1,2$ and $u \in \Sob^s(\RNP), \,f \in \Sob^{s-2}(\RNP)$, the equation $Pu = f$ can be interpreted in this weak sense, namely
	\[
	\left< u, P^*v \right>_{X} = \left<f, v \right>_X
	\]
	for all $v \in C_c^\infty(\RNP) \subseteq \dot{\Sob}^{2-s}(\RNP)$. For $s=2$ this is just the statement that $Pu = f$ in distributions. Furthermore, if $(u,\UL{\phi}) \in \TSob^s(\RNP)$ and $P(u,\UL{\phi}) = f$, then $Pu = f$ weakly.

Now suppose that $P \in \mathrm{Bess}_{\nu}^{(\lambda)}(\RNP)$ is a parameter-dependent Bessel operator. Recalling the definition of the parameter-dependent norms as in Section \ref{subsect:semiclassicalnorms}, it is straightforward to show that the following hold:
\begin{enumerate} \itemsep6pt
\item If $0 < \nu < 1$ and $s=0,1,2$, then there exists $C>0$ such that 
\[
\VERT P(\lambda)(u,\UL{\phi}) \VERT_{\Sob^{s-2}(\RNP)} \leq C \VERT (u,\UL{\phi}) \VERT_{\TSob^{s}(\RNP)}
\]
for each $(u,\UL{\phi}) \in \TSob^s(\RNP)$.
\item If $\nu \geq 1$ and $s=0,1,2$, then there exists $C>0$ such that 
\[
\VERT P(\lambda)u \VERT_{\Sob^{s-2}(\RNP)} \leq C \VERT u \VERT_{\Sob^{s}(\RNP)}
\]
for each $u \in \Sob^s(\RNP)$.
\end{enumerate}
There are also straightforward extensions of Lemmas \ref{lem:TDgreen1}, \ref{lem:TDgreen2}, \ref{lem:Bbounded}, and \ref{lem:Abounded} for parameter-dependent norms.

\subsection{Function spaces on a manifold}

 \label{subsect:spaceonmanifold}
Consider a compact manifold with boundary $\OL{X}$, equipped with a distinguished boundary defining function $x$ and collar diffeomorphism $\phi$ as in Section \ref{subsect:manifoldwithboundary}.

\begin{defi} \label{defi:cnuonmanifold}
Given $\nu > 0$, let $\mathcal{F}_\nu(X)$ denote the following spaces of functions.
\begin{enumerate} \itemsep6pt
\item If $0 < \nu < 1$, then $\mathcal{F}_\nu(X)$ consists of $u \in C^\infty(X)$ such that
\begin{equation} \label{eq:cnuonmanifold}
(u \circ \phi)(x,y) = x^{1/2-\nu}u_-(x^2,y) + x^{1/2+\nu}u_+(x^2,y)
\end{equation}
for some $u_\pm \in C^\infty([0,\sqrt{\varepsilon}) \times \partial X)$.
\item If $\nu \geq 1$, then $\mathcal{F}_\nu = C_c^\infty(X)$.
\end{enumerate}
\end{defi}

Thus $\Cnu = \mathcal{F}_\nu(\RNP)$.
Fix a finite open cover $\OL X = \bigcup_i U_i$ by coordinate charts $(U_i, \psi_i)$, such that either
\[
U_i \cap \partial X = \emptyset, \quad \psi_i : U_i \rightarrow \psi_i(U_i) \subseteq \mathbb{T}^n_+,
\]
or if $U_i \cap \partial X \neq \emptyset$, then
\[
U_i = \phi([0,\varepsilon) \times Y_i), \quad  \psi_i = (1 \times \theta_i)\circ \phi^{-1}
\]
for a coordinate chart $(Y_i, \theta_i)$ on $\partial X$. This of course implies that $\partial X = \bigcup_i Y_i$, where the union is taken over all $i$ such that $U_i \cap \partial X \neq \emptyset$. Now take a partition of unity of the form
\[
\sum_{i} \chi_i^2 = 1, \quad \chi_i \in C_c^\infty(U_i),
\]
with the additional property that if $U_i \cap \partial X \neq \emptyset$, then $\chi_i$ has the form
\[
\chi_i = (\alpha \beta_i) \circ \phi^{-1},
\]
for functions $\alpha \in C_c^\infty([0,\varepsilon)), \, \beta_i \in C_c^\infty(Y_i)$, where $\alpha = 1$ near $x=0$. Note that if $u \in \mathcal{F}_\nu(X)$ then $\chi_i u$ may be identified with an element of $\Cnu$ via the coordinate map $\psi_i$. Keeping this in mind, define
\[
\| u \|_{i,\Sob^{s}(X)} := \| (\chi_i u) \circ \psi_i^{-1} \|_{\Sob^{s}(\RNP)}
\]
for $s=0,\pm 1,\pm 2$ and $u \in \mathcal{F}_\nu(X)$.

\begin{defi} \label{defi:normonmanifold} Given $s=0,\pm 1,\pm 2$, let
\[
\| u \|^2_{\Sob^s(X)} = \sum_i \| u \|^2_{i,\Sob^{s}(X)}.
\]
Then define
\[
\Sob^s(X) = \text{closure of } \mathcal{F}_\nu(X) \text{ in the $\Sob^s(X)$ norm}.
\]
\end{defi}

To prove that $\Sob^s(X)$ is independent of the choice of covering $U_i$ and partition of unity $\chi_i$, the following two elementary results are needed; their proofs are left to the reader.

\begin{lem} \label{lem:invariance1}
	Let $Y,Y'$ be open subsets of $\RN$, and suppose that $\Phi : Y \rightarrow Y'$ is a diffeomorphism between them. Suppose that $K\subseteq \OL{\RP}\times Y$ is compact. Then for each $s=0,\pm1,\pm2$ there exists $C>0$ such that
	\[
	C^{-1} \| u \|_{\Sob^s(\RNP)} \leq \| u \circ (1\times \Phi) \|_{\Sob^s(\RNP)} \leq C\| u\|_{\Sob^s(\RNP)}
	\]
	for each $u \in \Cnu$ with $\supp u \subseteq K' := (1\times \Phi)(K)$,
\end{lem}

\begin{lem} \label{lem:invariance2}
	Let $K$ be a compact subset of $\RNP$. Then for each $s=0,\pm1,\pm2$ there exists $C>0$ such that
	\[
	C^{-1} \| u \|_{H^s(\RNP)} \leq \| u \|_{\Sob^s(\RNP)} \leq C\| u \|_{H^s(\RNP)}
	\]
	for each $u \in C_c^\infty(\RNP)$ such that $\supp u \subseteq K$.
\end{lem}

The combination of Lemmas \ref{lem:invariance1} and \ref{lem:invariance2} show that the spaces $\Sob^s(X)$ do not depend on any of the choices used to define them. 

\begin{lem}
Fix a density on $\OL X$ of product type near $\partial X$. Let $\left< \cdot, \cdot \right>_X$ denote the inner product on $L^2(X;\mu)$. For each $s=0,\pm 1,\pm2$,
\[
| \left< u , v \right>_X | \leq C \| u \|_{\Sob^{s}(X)} \| v \|_{\Sob^{-s}(X)},
\]
where $u,v \in \mathcal{F}_\nu(X)$. Furthermore, $\left< \cdot, \cdot \right>_X$ extends to a nondegenerate pairing $\Sob^s(X) \times \Sob^{-s}(X) \rightarrow \mathbb{C}$.
\end{lem}
\begin{proof}
	This can be reduced to the case on $\RNP$ via the coordinate charts $(U_i,\psi_i)$ and partition of unity $\chi_i$ used to define $\Sob^s(X)$.
	\end{proof}

Thus $\Sob^{-s}(X)$ is naturally identified with the antidual of $\Sob^s(X)$ via the inner product induced by $\mu$ on $\Sob^0(X)$. When $0 < \nu < 1$, it is also possible to show that the maps
\[
u \mapsto u_-(0,\cdot), \quad u \mapsto 2\nu u_+(0,\cdot)
\]
for $u \in \mathcal{F}_\nu(X)$ satisfying \eqref{eq:cnuonmanifold} admit continuous extensions $\gamma_\mp$ such that
\[
\gamma_\mp : \Sob^s(X) \rightarrow H^{s-1\pm\nu}(\partial X).
\]
It is understood that $\gamma_-$ exists for $s=1,2$, while $\gamma_+$ exists for $s=2$. 
The spaces $\TSob^{s}(X)$ are then defined exactly as in Section \ref{subsect:tildespace}.

\begin{lem} If $P \in \mathrm{Bess}_\nu(X)$, then the following hold.
\begin{enumerate} \itemsep6pt
\item If $0 < \nu <1$ and $s=0,1,2$, then there exists $C>0$ such that
\[
\| Pu \|_{\Sob^{s-2}(X)} \leq C \| (u,\UL{\gamma}u) \|_{\TSob^{s}(X)}.
\]
for each $u \in \mathcal{F}_\nu(X)$.
\item If $\nu \geq 1$ and $s=0,1,2$, then there exists $C>0$ such that 
\[
\| Pu \|_{\Sob^{s-2}(X)} \leq C \| u \|_{\Sob^{s}(X)}.
\]
for each $u \in \mathcal{F}_\nu(X)$.
\end{enumerate} 
\end{lem}

As in Section \ref{subsect:mappingpropertiesonRNP}, it follows that $(u,\UL{\gamma}u) \mapsto Pu$ admits a unique extension to $\TSob^{s}(X)$ for $0 < \nu <1$, and $u \mapsto Pu$ has a unique continuous extension to $\Sob^s(X)$ for $\nu \geq 1$.

The parameter-dependent norms on $\Sob^s(X)$ are defined by replacing $\| \cdot \|_{\Sob^s(X)}$ with $\VERT \cdot \VERT_{\Sob^s(X)}$ in Definition \ref{defi:normonmanifold}, and similarly for $\TSob^s(X)$. Then $P$ is uniformly bounded in $\lambda$ with respect to these norms.

A proof of the following can be found in \cite[Section 6]{holzegel:2012wt}. It is used in Section \ref{sect:fredholm} to prove the Fredholm property for certain boundary value problems.
\begin{lem}[{\cite[Section 6]{holzegel:2012wt}}] \label{lem:compact}  Let $\nu >0$ and $\mu$ be a density of product type near $\partial X$.
	\begin{enumerate} \itemsep6pt
		\item The inclusion $\Sob^1(X) \hookrightarrow \Sob^0(X)$ is compact.
		\item The injection $\Sob^0(X) \hookrightarrow \Sob^{-1}(X)$ induced by the $L^2(X;\mu)$ inner product is compact.
		\item If $0< \nu < 1$, then $\TSob^{1}(X)\hookrightarrow  \TSob^0(\RNP)$ and the injection $\TSob^0(X)\hookrightarrow \TSob^1(X)'$ induced by the $L^2(X;\mu)$ and $L^2(\partial X; \mu_{\partial X})$ inner products are compact.
		\end{enumerate}
\end{lem}
\begin{proof}
 For a proof of the first statement, see \cite[Section 6]{holzegel:2012wt}. The other two cases follow by duality.
\end{proof}

\subsection{Graph norms} \label{subsect:graphnorm}

Throughout this section, assume that $0 < \nu < 1$. Following \cite[Chapter 6.1]{roitberg:1996}, an alternative characterization of the spaces $\TSob^s(X)$ is given. Given $s=0,1,2$ and a Bessel operator $P$, define the norm
\[
\| u \|_{\Sob^s_P(X)} = \| u \|_{\Sob^s(X)} + \| Pu \|_{\Sob^{s-2}(X)}
\]
for $u \in \mathcal{F}_\nu(X)$.
\begin{lem} \label{lem:graphnorm}
	Give $s=0,1,2$ there exists $C>0$ such that
	\[
	C^{-1} \| u \|_{\Sob^s_P(X)} \leq \| (u, \UL{\gamma}u) \|_{\TSob^s(X)} \leq C\| u \|_{\Sob^s_P(X)}
	\]
	for each $u \in \mathcal{F}_\nu(X)$.
\end{lem}
\begin{proof}
	The first inequality above holds according to Lemma \ref{lem:besselextension1}. For the converse, it clearly suffices to bound 
	\[
	\| \UL{\gamma}u \|_{H^{s-\UL \nu}(\partial X)} \leq C\| u \|_{\Sob^s_P(X)}.
	\]
 Fix $u \in \mathcal{F}_\nu(X)$, and note that 
	\[
\| \UL{\gamma} u \|_{H^{s-\UL{\nu}}(\partial X)} = \| J\UL \gamma u \|_{H^{s-2+\UL{\nu}}(\partial X)}.
\]
On the other hand if $\psi \in H^{2-s-\UL{\nu}}(\partial X)$ has norm one, let $(v,\UL\psi) = \widetilde{\mathcal K}\UL\psi \in \TSob^{2-s}(X)$, where $\widetilde{\mathcal{K}} : H^{2-s-\UL{\nu}}(\partial X) \rightarrow \TSob^{2-s}(X)$ is the map defined in Lemma \ref{lem:tracelift}. Applying Green's formula,
	\begin{align*}
|\left< J \UL{\gamma} u, \UL{\psi} \right>_{\partial X}| &= |\left< Pu, v \right>_X - \left< u, P^*(v,\UL{\psi}) \right>_X| \\
&\leq  C_1 (\| Pu \|_{\Sob^{s-2}(X)} + \|u \|_{\Sob^s(X)}) \| \widetilde{\mathcal{K}}\UL{\psi} \|_{\TSob^{2-s}(X)} \leq C_2 \| u \|_{\Sob^s_P(X)},
\end{align*}
whence it follows that $\| (u, \UL{\gamma}u) \|_{\TSob^s(X)} \leq C\| u \|_{\Sob^s_P(X)}$ for some $C>0$.
\end{proof}

Let $\Sob^s_P(X)$ denote the closure of $\mathcal{F}_\nu(X)$ in the norm $\| \cdot \|_{\Sob^s_P(X)}$. Since $(u,\UL{\gamma }u),\, u \in \mathcal{F}_\nu(X)$ is dense in $\TSob^s(X)$, it follows from Lemma \ref{lem:graphnorm} that $\Sob^s_P(X)$ is naturally isomorphic to $\TSob^s(X)$ via the closure of the map $u \mapsto (u,\UL{\gamma}u)$. Moreover, any element of $\Sob^s_P(X)$ can be identified with a unique pair $(u,f)$, where $u \in \Sob^s(X), \, f \in \Sob^{s-2}(X)$, and $Pu = f$ in the weak sense (described at the end of Section \ref{subsect:mappingpropertiesonRNP}).

\section{Elliptic boundary value problems} \label{sect:ellipticBVP}

This section concerns boundary value problems for Bessel operators on a compact manifold with boundary $\OL{X}$ as in Section \ref{subsect:manifoldwithboundary}. When $0 < \nu < 1$, these are of thee form
\begin{equation}\label{eq:BVP}
\begin{cases} Pu = f & \text{ on } X,\\
Tu = g & \text{ on } \partial X.
\end{cases}
\end{equation}
Here $P \in \mathrm{Bess}_\nu(X)$ is Bessel operator which is elliptic in the sense of Section \ref{subsect:ellipticity} on $\partial X$, and 
\[
T = T^+ \gamma_+  + T^- \gamma_- 
\] 
for some differential operators $T^\pm$ on the boundary, to be specified in the next section. The boundary operator $T$ is only relevant when $0 < \nu < 1$. When $\nu \geq 1$, one considers the simpler equation
\[
Pu = f \text{ on } X.
\] 

To highlight the difference between the cases $0 < \nu < 1$ and $\nu \geq 1$, fix $p \in \partial X$ and consider the model equation on $\RP$ determined by the boundary symbol operator,
\begin{equation} \label{eq:ellipticitymodel}
\widehat{P}_{(p,\eta)}u = f,
\end{equation}
referring to Section \ref{subsect:ellipticity} for notation. Suppose that $P$ is elliptic at $p\in \partial X$. Any two solutions to the equation \eqref{eq:ellipticitymodel} differ by an element of the kernel of $\widehat{P}_{(p,\eta)}$. If $u \in \ker \widehat{P}_{(p,\eta)}$ satisfies $u \in L^2((1,\infty))$, then necessarily $u \in \mathcal{M}_+(p,\eta)$. On the other hand, if $\nu$ is not an integer, then 
\begin{equation} \label{eq:besselorigin}
K_\nu(s) = \frac{\pi}{2}\frac{I_{-\nu}(s) - I_{\nu}(s)}{\sin(\nu\pi)},
\end{equation}
where $I_\nu$ is the modified Bessel function of the first kind \cite[Chapter 7.8]{olver:2014} (if $\nu$ is an integer, equality holds in the sense of limits). In particular, if $0 < \nu < 1$, then $I_{\pm \nu}(s) = \mathcal{O}(s^{\pm \nu})$. Consequently $\ker \widehat{P}_{(p,\eta)} \cap L^2(\RP) = \mathcal{M}_+(p,\eta)$, and hence $\widehat{P}_{(p,\eta)}$ cannot be an isomorphism between any $L^2$ based spaces: in general, \eqref{eq:ellipticitymodel} must be augmented by boundary conditions so that the $L^2$ kernel is trivial. Of course, all of these observations are classical when $\nu = \frac{1}{2}$ (boundary value problems in the smooth setting).

This is in contrast to the situation when $\nu \geq 1$. In that case, $\sqrt{x}K_\nu(i\xi(p,\eta) x)$ is not square integrable near the origin, and so the $L^2$ kernel of $\widehat{P}_{(p,\eta)}$ is always trivial. Hence specifying $f$ on the right hand side of \eqref{eq:ellipticitymodel} (in an appropriate function space) will uniquely determine a solution $u$. Thus in the case $\nu \geq 1$, it is not necessary to impose any boundary conditions apart from the square integrability requirement.

In the self-adjoint setting, the heuristic above is the limit point/limit circle criterion of Weyl on self-adjoint extensions of symmetric ordinary differential operators with regular singular points; see \cite{zettl:2010} for an exhaustive modern treatment, and \cite{bachelot:2011:jmpa,ishibashi:2004:cqg} for discussions in the context of aAdS spacetimes

\subsection{Boundary conditions} \label{subsect:BC}

This section is only relevant in the case $0 < \nu < 1$. Choose differential operators
\[
T^- \in \mathrm{Diff}^1(\partial X), \quad T^+ \in \mathrm{Diff}^0(\partial  X),
\]
noting that $T_+$ is just multiplication by a smooth function on $\partial X$. Then set 
\[
T = T^- \gamma_- + T^+ \gamma_+.
\] 
A natural question is how to define the ``leading order'' term in $T$. Suppose that $\mu \in \{1-\nu,2-\nu,1+\nu\}$ and
\begin{equation} \label{eq:orderlessthanm}
\mathrm{ord}(T^-)   \leq \mu - 1 +\nu, \quad \mathrm{ord}(T^+) \leq \mu - 1 - \nu.
\end{equation}
Then $T$ is said to have $\nu$-order less than or equal to $\mu$, written as $\mathrm{ord}_\nu(T) \leq \mu$. Note that if $\mathrm{ord}_\nu(T) \leq \mu$, then $T: \mathcal{H}^{2}(X) \rightarrow H^{2 - \mu}(\partial  X)$ is continuous. If $\mathrm{ord}_\nu(T) \leq \mu$, define the family of operators 
\[
\widehat{T}_{(p,\eta)} = \sigma_{\lceil \mu-1+\nu \rceil}(T^-)(p,\eta) \gamma_-  +  \sigma_{\lceil\mu-1-\nu\rceil}(T^+)(p,\eta)\gamma_+,
\]
indexed by $(p,\eta) \in T^* \partial X$. Thus each $(p,\eta) \in T^*\partial X$ gives rise to a one-dimensional boundary operator $\widehat{T}_{(p,\eta)}$.

\subsection{The boundary value problem} \label{subsect:theBVP}

Although boundary value problems of the form $\eqref{eq:BVP}$ are ultimately of interest, for duality purposes it is convenient to consider a more general type of problem. Fix $J \in \mathbb{N}$, and choose
\begin{itemize} \itemsep6pt
\item $\mu_k \in \{1-\nu,2-\nu,1+\nu\}$ for $k \in \{1,\ldots, J+1\}$,
\item numbers $\tau_j \in \mathbb{R}$ for $j \in \{1,\ldots J\}$, not necessarily integers.
\end{itemize}
Let $T = (T_1,\ldots, T_{J+1})^\top$ denote a $(J+1) \times 1$ matrix of boundary operators, such that $\mathrm{ord}_\nu(T_k) \leq \mu_k$. Furthermore, for each $k \in \{1,\ldots,J+1\}$ and $j \in \{1,\ldots,J\}$, suppose $C_{k,j} \in \mathrm{Diff}^*(\partial X)$ is a differential operator on $\partial X$ such that
\[
\mathrm{ord}(C_{k,j}) \leq \tau_j + \mu_k.
\]
Let $C$ denote the $(J+1) \times J$ matrix with entries $C_{k,j}$. Given these prerequisites, consider the modified boundary value problem
\begin{equation} \label{eq:modifiedBVP}
\begin{cases}
P u  = f & \text{ on $X$} \\ 
T u + C\underline{u} = \underline{g} & \text{ on $\partial X$},
\end{cases}
\end{equation}
where $\underline{u} = (u_1,\ldots u_J), \, \underline{g} = (g_1,\ldots, g_{J+1})$ are collections of functions on $\partial X$. In order to associate an operator to this problem, note that $Tu$ may be written in the form
\[
Tu = G \UL{\gamma} u,
\]
where $G$ is the $(J+1) \times 2$ matrix
\[
G =  \left( \begin{array}{cc}
T^{-}_1 & T^+_1 \\
\vdots & \vdots \\
T^-_{J+1} & T^+_{J+1}
\end{array} \right).
\]
Throughout, it is always understood that $G$ is associated with $T$ in this way. Finally, set $\UL{\mu} = (\mu_1,\ldots,\mu_{J+1})$ and $\UL{\tau} = (\tau_1, \ldots, \tau_{J})$. Then let $\mathscr{P}$ denote the map
\[
\mathscr{P}(u, \underline{\phi}, \underline{u}) = (P(u,\underline{\phi}), G\underline{\phi} + C \underline{u}).
\]
This is also written as $\mathscr{P} = \{P,T,C\}$.

\begin{lem} \label{lem:BVPcontinuous}
The map $\mathscr{P} = \{P,T,C\}$ is bounded
\[
\mathscr{P} :
 \TSob^{s}(X) \times H^{s+\underline{\tau}}(\partial X) \rightarrow \Sob^{s-2}(X) \times H^{s-\underline{\mu}}(\partial X)
\]
for each $s=0,1,2$.
\end{lem}
\begin{proof}
The mapping properties follows from the results of Section \ref{subsect:spaceonmanifold}.
\end{proof}

\subsection{The adjoint boundary value problem} \label{subsect:adjointproblem}
Fix a density $\mu$ which is of product type near $\partial X$. Let $P^*$ denote the formal $L^2(X;\mu)$ adjoint of $P$; then $P^*$ is also a Bessel operator in light of Lemma \ref{lem:closedunderadjoint}. Let $C^*, \, G^*$ denote the formal $L^2(\partial X;\mu_{\partial X})$ adjoints of $C, \, G$. Define the problem
\begin{equation} \label{eq:modifiedBVPadjoint}
\begin{cases}
P^* v  = f & \text{ on $X$}, \\ 
J \UL{\gamma} v + G^* \underline{v} = \underline{g} & \text{ on $\partial X$}, \\
C^* \underline{v} = \underline{h} & \text{ on $\partial X$},
\end{cases}
\end{equation}
where $\underline{v} = (v_1,\ldots,v_{J+1}), \, (\underline{g},\underline{h}) = (g_1,g_2,h_1,\ldots h_J)$ are functions on $\partial X$. 

Although Green's formula \eqref{eq:green2} was previously only established for the formal adjoint of a Bessel operator on $\RNP$, it is clear that \eqref{eq:green2} also holds here when the appropriate $\mu$ and $\mu_{\partial X}$ inner products are substituted on $X$ and $\partial X$:
\[
\left<Pu, v \right>_X + \left< Tu + C\UL{u}, \UL{v} \right>_{(\partial X)^{J+1}} = \left< u, P^* v \right>_X + \left< \UL{\gamma}u, G^* \UL{v} + J\UL{\gamma} v \right>_{(\partial X)^2} + \left<\UL{u}, C^* \UL{v}\right>_{(\partial X)^J}.
\]
In light of this, the problem \eqref{eq:modifiedBVPadjoint} is said to be the formal adjoint of \eqref{eq:modifiedBVP}. Also notice that \eqref{eq:modifiedBVPadjoint} has the same form as \eqref{eq:modifiedBVP}. The corresponding operator is denoted by $\mathscr{P}^*$.

\subsection{The Lopatinski\v{\i} condition} \label{subsect:lopatinskii}

The standard Lopatinski\v{\i} condition for smooth elliptic boundary value problems (see \cite{lions:1968,roitberg:1996}) has a natural generalization to the situation here. Begin by choosing $c_{k,j} \in \mathbb{Z}$ (not necessarily nonnegative) such that 
\[
\mathrm{ord}(C_{k,j}) \leq c_{k,j} \leq \tau_j + \mu_k,
\]
and then define the matrix $\widehat{C}_{(y,\eta)}$ with entries 
\[
(\widehat{C}_{(p,\eta)})_{k,j} = \sigma_{c_{k,j}}(C_{k,j})(p,\eta).
\] 
Thus $(p,\eta) \mapsto \widehat{C}_{(p,\eta)}$ is a function on $T^*\partial X$ with values in matrices over $\mathbb{C}$. Furthermore, define $\widehat{G}_{(p,\eta)}$ by the equality
\[
\widehat{G}_{(p,\eta)}\UL{\gamma}u = \widehat{T}_{(p,\eta)}u.
\]

\begin{defi} \label{defi:lopatinskii} Suppose $P$ is elliptic on $\partial X$. The boundary operators $(T,C)$ are said to satisfy the Lopatinski\v{\i} condition with respect to $P$ if for each fixed $p \in \partial M$ and $\eta \in T^*_p\partial X \setminus 0$, the only element $(u,\underline{u}) \in \mathcal{M}_+(p,\eta) \times \mathbb{C}^J$ satisfying
\[ 
\widehat{T}_{(p,\eta)}u + \widehat{C}_{(p,\eta)}\underline{u} = 0 \\
\]
is the trivial solution $(u,\underline{u}) = 0$. The boundary value problem \eqref{eq:modifiedBVP}, or equivalently the operator $\mathscr{P} = \{P,T,C\}$, is said to be elliptic on $\partial X$ if $P$ is elliptic on $\partial X$ in the sense of Definition \ref{def:elliptic} and $(T,C)$ satisfy the Lopatinski\v{\i} condition on $\partial X$ with respect to $P$.
\end{defi}
It is easy to see that the generalized Dirichlet condition $T = \gamma_-$ and Neumann condition $T = \gamma_+$ satisfy the Lopatinski{\v\i} condition with respect to any elliptic Bessel operator. The same is therefore true for the Robin condition $T = \gamma_+ + T^- \gamma_-$, where $T^- \in \mathrm{Diff}^0(\partial X)$. On the other hand, when $T^-$ is allowed to be a first order operator, there are phenomena not present for smooth boundary value problems; the following two examples illustrate some possibilities.

	\begin{example} Consider a boundary condition $T = \gamma_+ + T^- \gamma_-$,
		where $T^-$ is a nonzero vector field on $\partial X$.
		\begin{enumerate} \itemsep6pt
			\item 	If $1/2 < \nu < 1$, then $\widehat{T}_{(p,\eta)} = \gamma_+$ for arbitrary $T^-$. Thus $T$ satisfies the Lopatinski{\v\i} conditions with respect to any elliptic Bessel operator.
			\item If $ \nu = 1/2$, then $T$ is a classical oblique boundary condition. The Lopatinski{\v\i} condition is satisfied if $T^-$ is a real vector field for example, but can otherwise fail. 
			
			\item If $0 < \nu < 1/2$, then
			\[
			\widehat{T}_{(p,\eta)} = \sigma_1(T^-)(p,\eta) \gamma_-.
			\]
			Since $\sigma_1(T^-)(p,\eta)$ is linear in $\eta$, it must have a nontrivial zero at each $p \in \partial X$ provided the dimension of the underlying manifold $X$ is at least three (or two if $T^-$ is real). In that case the Lopatinski{\v\i} condition necessarily fails at every point on the boundary.
		\end{enumerate} 
		
\end{example}

\begin{example}
			Consider the operator $\Delta_\nu = |\TD|^2 + D_y^2 + D_z^2$ acting on $(0,1) \times \mathbb{T}^2$, where $(y,z)$ are standard coordinates on $\mathbb{T}^2 = (R/2\pi\mathbb{Z})^2$. Clearly $\Delta_\nu$ is an elliptic Bessel operator. Consider the boundary value problem
			\[
			\begin{cases} \Delta_\nu u = f, \\ T u = g, \quad  u|_{x=1} = 0, \end{cases}
			\]
			where  $T = (\partial_y - \partial_z)\gamma_-$. 
This is not a Fredholm problem, since there is an infinite dimensional kernel: for each $n \geq 0$, consider the function
			\[
			u_n(x,y,z) =\left(\sqrt{x}K_\nu(nx) - \frac{K_\nu(n)}{K_\nu(-n)}\sqrt{x}K_\nu(-nx) \right)e^{in(y+z)}.
			\]
			The family $\{u_n\}$ is linearly independent and each $u_n$ solves the boundary value problem. If $0 < \nu < 1/2$, then $T' = \gamma_+ + T$ is a compact perturbation of the original problem; thus the problem with $T'$ replacing $T$ is not Fredholm either. If $1/2 \leq \nu < 1$ then the problem with $T'$ satisfies the Lopatinski{\v\i} condition, so is indeed Fredholm by the arguments in Section \ref{sect:fredholm}.
	\end{example}

Before proceeding with the next lemma, suppose that $P \in \mathrm{Bess}_\nu(X)$ and $\mu$ is a density of product type near $\partial X$. If $P$ is elliptic at $p \in \partial X$, then so is $P^*$, since the function \eqref{eq:properlyellipticpolynomial} is simply replaced by its complex conjugate.

\begin{lem} \label{lem:adjointiselliptic}
Suppose that $\mathscr{P} = \{P,T,C\}$ is elliptic. If $\mu$ is a density of product type near $\partial X$ and $\mathscr{P}^*$ is the corresponding adjoint boundary value problem, then $\mathscr{P}^*$ is also elliptic.
\end{lem}
\begin{proof}
Since ellipticity only depends on various ``principal symbols'', it is easy to see that
\[
\widehat{P^*}_{(p,\eta)} = \widehat{P}^*_{(p,\eta)},
\]
where the latter adjoint is calculated with respect to the standard $L^2(\RP)$ inner product. Similarly
\[
\widehat{T^*}_{(p,\eta)} = \widehat{T}^*_{(p,\eta)}, \quad \widehat{C^*}_{(p,\eta)} = \widehat{C}^*_{(p,\eta)},
\]
where the latter adjoints are taken in the sense of matrices over $\mathbb{C}$. 
 
Suppressing the dependence on $(p,\eta)$, Green's formula \eqref{eq:green2} implies that
\begin{multline*}
\langle \widehat{P}u, v \rangle_{\RP} + \langle \widehat{T} u + \widehat{C} \UL{u}, \UL{v} \rangle_{\mathbb{C}^{J+1}} = \langle u,\widehat{P}^* v \rangle_{X} + \langle \UL{\gamma}u, \widehat{G}^* \UL{v} + J\UL{\gamma} v \rangle_{\mathbb{C}^2} + \langle \UL{u}, \widehat{C}^* \UL{v} \rangle_{\mathbb{C}^J}.
\end{multline*}
The goal is to prove that if $v \in L^2(\RP)$ and the right hand side vanishes, then $(v,\underline{v}) = 0$. The proof relies on Lemma \ref{lem:1daprioris=2} below (whose proof is of course independent of the present lemma). As in the proof of Lemma \ref{lem:1daprioris=2}, the Lopatinski\v{\i} condition implies that
\[
(u,\underline{u}) \mapsto \widehat{T}u + \widehat{C}\UL{u}
\]
is an isomorphism between the spaces $\mathcal{M}_+ \times \mathbb{C}^J \rightarrow \mathbb{C}^{J+1}$. So choose $(u,\underline{u}) \in \mathcal{M}_+ \times \mathbb{C}^J$ such that
\[
\widehat{T}u + \widehat{C}\UL{u} = \UL{v}.
\]
Since $\widehat{P}u = 0$, it follows from Green's formula that $\underline{v} = 0$. On the other hand, from Lemma \ref{lem:1daprioris=2} it is always possible to solve the inhomogeneous equation
\[
\begin{cases} \widehat{P} u = v, \\
\widehat{T} u + \widehat{C} \UL{u} = 0,
\end{cases}
\]
whence Green's formula implies that $v = 0$ as well. 
\end{proof}

\subsection{The Dirichlet Laplacian} \label{subsect:dirichletlaplacian} The results of this section are applied in one dimension to Section \ref{subsect:ccestimate}. Define the Bessel operator $\Delta_\nu \in \mathrm{Bess}_\nu(\RNP)$ by
\[
\Delta_\nu = |\TD|^2 + \Delta_{\RN}, 
\]
where $\Delta_{\RN} = \sum_{i=1}^{n-1} D_{y^{i}}^2$ is the non-negative Laplacian on $\RN$. Consider the continuous, non-negative Hermitian form
\begin{equation} \label{eq:sesquilinearform}
\ell(u,v) := \langle \TD u, \TD v\rangle_{\RNP} + \sum_{i=1}^{n-1} \left<D_{y^i} u, D_{y^i}u \right>_{\RNP}
\end{equation}
on $\dot{\Sob}^1(\RNP)$. Associated to this form is the unbounded self-adjoint operator $L$ on $L^2(\RNP)$ with domain
\begin{equation} \label{eq:D(L)}
D(L) = \dot{\Sob}^{1}(\RNP) \cap \{ u \in L^2(\RNP): \Delta_\nu u \in L^2(\RNP) \},
\end{equation}
and $Lu = \Delta_\nu u$ in the sense of distributions for each $u \in D(L)$. The domain $D(L)$ is equipped with the graph norm. 
\begin{rem}
	In one dimension it is obvious that $D(L) = \Sob^2(\RP)\cap \dot{\Sob}^1(\RP)$, with an equivalence of norms via the open mapping theorem. This is also true in higher dimensions, but is not immediate from the definition.
\end{rem}

 The next lemma follows from the Lax--Milgram theorem.
\begin{lem} \label{lem:laxmilgram}
Let $\nu >0$. For each $a \in \mathbb{C} \setminus (-\infty,0]$ the inverse $(L + a)^{-1}$ exists, and maps 
\[
(L + a)^{-1}: \begin{cases} \dot{\Sob}^1(\RNP)' \rightarrow \dot{\Sob}^1(\RNP),\\
L^2(\RNP) \rightarrow D(L). \end{cases}
\]
\end{lem}
\begin{proof}
Since $a \notin \ (-\infty,0]$ the form $\ell_a(u,v) = \ell(u,v) + a\left<u,v\right>_{\RNP}$ is coercive on $\dot{\Sob}^1(\RNP)$, so $\ell_a(u,v)$ defines an inner product on $\dot{\Sob}^1$ equivalent to the usual one. The Lax--Milgram theorem guarantees that for each $f \in \dot{\Sob}^1(\RNP)'$ there exists a unique $u \in \dot{\Sob}^1(\RNP)$ such that $\ell_a(u,v) = \left<f,v\right>$, and the mapping $u \mapsto f$ is continuous $\dot{\Sob}^1(\RNP)' \rightarrow \dot{\Sob}^1(\RNP)$. 

Furthermore, the unbounded operator associated to $\ell_a$ is clearly $L + a$ (acting in the distributional sense) so $L + a : D(L) \rightarrow L^2(\RNP)$ is bijective. Since this map is continuous when $D(L)$ is equipped with the graph norm, it is an isomorphism by the open mapping theorem.
\end{proof}

\subsection{Elliptic Bessel operators on $\RP$} \label{subsect:ellipticbesselonRP}

In this section, fix an operator $P$ on $\RP$ of the form
\begin{equation} \label{eq:1dbessel}
P = |\TD|^2 + a, \quad a \in \mathbb{C}.
\end{equation}
Thus $\xi \mapsto \xi^2 + a$ has no real roots precisely when $a \notin (-\infty,0]$. In that case, $P$ is said to be \emph{regular}. This is distinguished from ellipticity of $P$ since the principal symbol of multiplication by $a$ as a second order operator is zero (in other words, the boundary symbol operator is $|\TD|^2$ and not $|\TD|^2 + a$). Furthermore, if $0 < \nu < 1$, fix boundary conditions $(T,C)$. Thus $T$ is just a column vector of $J$ boundary operators $T_k =  \sum_{\pm} T_{k}^\pm \gamma_\pm$ with $T^\pm_k \in \mathbb{C}$, and $C$ is a $(J+1) \times J$ matrix with $\mathbb{C}$-valued entries.

Regularity of the operator $\mathscr{P} = \{P,T,C\}$ is defined as just the Lopatinski{\v\i} condition: let $\mathcal{M}_+$ denote the space of bounded solutions to the equation $Pu = 0$. Then $\mathscr{P}$ is regular if the only element $(u, \UL{u}) \in \mathcal{M}_+ \times \mathbb{C}^J$ satisfying $Tu + C\UL{u}=0$ is the trivial solution.

\begin{prop} \label{prop:1dapriori}
Suppose that $P$ given by \eqref{eq:1dbessel} is regular, and that $\mathscr{P} = \{P,T,C\}$ is regular if $0 < \nu < 1$.

\begin{enumerate} \itemsep6pt
\item If $0 < \nu < 1$, then $\mathscr{P}$ is an isomorphism
\[
\TSob^{s}(\RP) \times \mathbb{C}^J \rightarrow \Sob^{s-2}(\RP) \times \mathbb{C}^{1+J}
\]
for each $s = 0,1,2$. The operator norm of $\mathscr{P}^{-1}$ depends continuously on $a$ and the coefficients of $G$ and $C$

\item If $\nu \geq 1$, then $P$ is an isomorphism
\[
\Sob^s(\RP) \rightarrow \Sob^{s-2}(\RP)
\]
for each $s=0,1,2$. The operator norm of $P^{-1}$ depends continuously on $a$.
\end{enumerate}
\end{prop}

\noindent The proof of this proposition is split up across several Lemmas. 
\begin{lem} \label{lem:1daprioris=2}
Proposition \ref{prop:1dapriori} holds when $0 < \nu < 1$ and $s=2$.
\end{lem}
\begin{proof}
Since $\TSob^{2}(\RP)$ is isomorphic to $\Sob^{2}(\RP)$ via the map $v \mapsto (v,\UL{\gamma}v)$, it is sufficient to prove the lemma with $\Sob^{2}(\RP)$ replacing $\TSob^{2}(\RP)$. By the regularity condition, $P$ is injective. Indeed any solution in $\Sob^{2}(\RP)$ to the equation $P u = 0$ must lie in $\mathcal{M}_+$, and the Lopatinski\v{\i} condition implies that such a solution is unique. It  remains to show surjectivity.

Fix $(f,g) \in \Sob^0(\RP) \times \mathbb{C}^{J+1}$. From Lemma \ref{lem:laxmilgram}, it follows that that the equation
\[
P u = f
\]
has a solution $u_1 \in \mathcal{H}^2(\RP)\cap \dot{\Sob}^1(\RP)$. It then suffices to let $(u_2, \UL{u}) \in \mathcal{M}_+ \times \mathbb{C}^J$ solve
\[
\begin{cases}
P u_2 = 0,\\
T u_2 + C\UL{u} = g - T u_1. \\
\end{cases}
\] 
This is possible since
\[
(u, \underline{u} ) \mapsto T u + C \underline{u}
\]
as a map between the finite dimensional vector spaces $\mathcal{M}_+ \times \mathbb{C}^J \rightarrow \mathbb{C}^{J+1}$ is injective, hence an isomorphism. Setting $u = u_1 + u_2$ shows that $\mathscr{P}(u,\UL{u}) = (f,g)$. It is also easy to see that the operator norm of $\mathscr{P}$ depends continuously on $a$ and the coefficients of $G$ and $C$, which implies the same for the operator norm of $\mathscr{P}^{-1}$ via the resolvent identity.
\end{proof}

\begin{lem} \label{lem:1daprioris=0}
Proposition \ref{prop:1dapriori} holds when $0 < \nu < 1$ and $s=0$.
\end{lem}
\begin{proof}
Since the formal adjoint operator $\mathscr{P}^*$ is also regular according to Lemma \ref{lem:adjointiselliptic}, the map
\[
\Sob^{2}(\RP) \times \mathbb{C}^{1+J}  \rightarrow \Sob^{0}(\RP) \times \mathbb{C}^2 \times \mathbb{C}^{J} 
\]
given by
\[
(v,\underline{v}) \mapsto (P^* v, \, J\UL{\gamma} v + G^*\underline{v}, \, C^*\underline{v})
\]
is an isomorphism according to Lemma \ref{lem:1daprioris=2}. But in that case, a direct calculation shows that $\mathscr{P}^*$ agrees with the Hilbert space adjoint $\mathscr{P}'$ of
\[
\mathscr{P} : \TSob^{0}(\RP) \times \mathbb{C}^J \rightarrow \mathcal{H}^{-2}(\RP) \times \mathbb{C}^{1+J}.
\]
Since $\mathscr{P}'$ is an isomorphism, $\mathscr{P}$ is an isomorphism on the stated spaces as well.
\end{proof}

\noindent To prove Proposition \ref{prop:1dapriori} for $s=1$, the following regularity result is needed.

\begin{lem} \label{lem:1dregularity} Let $0 < \nu < 1$. Suppose that $(u,\underline{\phi}) \in \TSob^0(\RP)$ satisfies $P(u,\underline{\phi}) \in \Sob^{-1}(\RP)$. Then $(u,\underline{\phi}) \in \TSob^{1}(\RP)$.
\end{lem}
\begin{proof}
Let $f \in \dot{\Sob}^1(\RP)'$ denote the restriction of the functional $P(u,\underline{\phi})$ to $\dot{\Sob}^1(\RP)$. This implies that $f = Pu$ in the sense of distributions. By Lemma \ref{lem:laxmilgram}, there exists a unique $\tilde u \in \dot{\Sob}^1(\RP)$ such that $P \tilde u = f$ in the distributional sense. Thus in sense of distributions on $\RP$,
\[
P(u - \tilde u) = 0.
\]
Since $u$ and $\tilde{u}$ are square integrable, it follows that $u - \tilde{u} \in \mathcal{M}_+$. Thus it is certainly true that
\[
u = (u - \tilde{u}) + \tilde{u} \in \Sob^1(\RP).
\]
It remains remains to prove that $\phi_- = \gamma_- u$. A priori $(u,\underline{\phi}) \in \TSob^0(\RP)$, so for each $v\in \Sob^2(\RP)$,
\[
\langle f , v \rangle_{\RP} = \langle u , P^* v \rangle_{\RP} - \phi_+ (\gamma_- v) + \phi_- (\gamma_+ v).
\]
Using that $u \in \Sob^1(\RP)$, this may be rewritten as
\[
\langle f , v \rangle_{\RP} - \langle \TD u, \TD v \rangle_{\RP} - a \langle  u , v \rangle_{\RP} + \phi_+(\gamma_-v) = (\phi_- - \gamma_- u)\gamma_+ v  
\]
for each $v\in \Sob^2(\RP)$. But the left hand side extends to a continuous functional on $\Sob^1(\RP)$, which is not true of the right hand side unless $\phi_- = \gamma_- u$, thus completing the proof. 
\end{proof}

\begin{lem} \label{lem:1daprioris=1}
Proposition \ref{prop:1dapriori} holds when $0 < \nu < 1$ and $s=1$.
\end{lem}
\begin{proof}
The regularity result of Lemma \ref{lem:1dregularity} combined with Lemma \ref{lem:1daprioris=0} shows that $\mathscr{P}$ defines a continuous bijection, hence an isomorphism
\[
\TSob^1(\RP) \times \mathbb{C}^J \rightarrow \Sob^{-1}(\RP) \times \mathbb{C}^{J+1}
\]
as stated.
\end{proof}

\begin{lem} \label{lem:1dapriorinugeq1}
Proposition \ref{prop:1dapriori} holds when $\nu \geq 1$.
\end{lem}
\begin{proof}
If $\nu \geq 1$, then $\Sob^{s}(\RP) = \dot{\Sob}^s(\RP)$ for $s\geq 0$. Thus it suffices to apply Lemma \ref{lem:laxmilgram} directly when $s=1,2$. The case $s=0$ is handled by duality, similar to \ref{lem:1daprioris=0}.
\end{proof}

\begin{proof}[Proof of Proposition \ref{prop:1dapriori}]
The combination of Lemmas \ref{lem:1daprioris=2}, \ref{lem:1daprioris=1}, \ref{lem:1daprioris=0}, and \ref{lem:1dapriorinugeq1} establishes Proposition \ref{prop:1dapriori}
\end{proof}

\subsection{Elliptic Bessel operators on $\RNP$ with constant coefficients} \label{subsect:ccestimate}

Throughout this section, $P$ denotes a constant coefficient Bessel operator on $\RNP$,
\begin{equation} \label{eq:ccbessel}
P(\TD,D_y) = |\TD|^2 + A(D_y).
\end{equation}
If $0 < \nu < 1$, then $P$ is also augmented by boundary conditions $(T,C)$ with constant coefficients: thus each boundary operator is of the form $T_k(D_y) = \sum_{\pm} T_k^\pm(D_y) \gamma_\pm$, and each entry of $C(D_y)$ has constant coefficients.

The principal part of $P$ is the operator
\[
P^\circ(\TD,D_y) = |\TD|^2 + A^\circ(D_y),
\]
where $A^\circ(D_y)$ is the usual principal part of $A$. The principal parts of $(T(D_y),C(D_y))$ are defined to be the unique boundary operators $(T^\circ(D_y),C^\circ(D_y))$ satisfying
\[
T^\circ(\eta) = \widehat{T}_\eta, \quad C^\circ(\eta) = \widehat{C}_\eta
\]
for each $\eta \in \mathbb{R}^{n-1}$. Finally, define $\mathscr{P}^\circ(\TD,D_y) = \{ P^\circ(\TD,D_y),T^\circ(D_y),C^\circ(D_y)\}$. Ellipticity of either $P$ or $\mathscr{P}$ depends only on these principal parts.

\begin{lem} \label{lem:higherdimapriori}
Assume that $P$ and $\mathscr{P}$ are elliptic. Furthermore, assume that the one dimensional operators $P(\TD,q)$ (if $\nu \geq 1$) and $\mathscr{P}(\TD,q)$ (if $0< \nu < 1$) are regular for each $q\in\mathbb{Z}^{n-1}$.
\begin{enumerate} \itemsep6pt
	\item If $0 < \nu < 1$, then
	\[
	\mathscr{P}: \Sob^2(\RNP) \times H^{2+\UL{\tau}}(\RN) \rightarrow \Sob^0(\RNP) \times \Sob^{2-\UL{\mu}}(\RN)
	\]
	is an isomorphism.
	\item If $\nu \geq 1$, then
	\[
	P: \Sob^2(\RNP) \rightarrow \Sob^0(\RNP)
	\]
	is an isomorphism.
	\end{enumerate}
\end{lem}
\begin{proof}
\begin{inparaenum}
	\item  Let $0 < \nu <1$. By ellipticity,  
	\[
	\mathscr{P}^\circ(\TD,\left<q\right>^{-1}q): \Sob^2(\RP) \times \mathbb{C}^J \rightarrow \mathbb{C}^{1+J}
	\]
	is an isomorphism for each $q \in \mathbb{Z}^{n-1} \setminus 0$, as in the proof of Lemma \ref{lem:1daprioris=2}. Since $\left<q\right>^{-1}q$ ranges over a compact subset of $\mathbb{R}^{n-1}$, the operator norm of $\mathscr{P}^\circ(\TD,\left<q\right>^{-1}q)^{-1}$ is bounded uniformly with respect to $q \in \mathbb{Z}^{n-1}\setminus 0$. On the other hand, the homogeneity of $P^\circ$ implies
	\[
	\tau^{-2}S_{-\tau}P^\circ(\TD,D_y)S_\tau = P^\circ(\TD,\tau^{-1}D_y), \quad \tau^{-\UL{\mu}+1/2} T^\circ(D_y)S_\tau = T^\circ(\tau^{-1} D_y).
	\]
	Using $\tau = \left<q\right>$, this implies that the operator norm corresponding to the problem
	\[
	\begin{cases}
	\left<q\right>^{-2}S_{\left<q\right>^{-1}}{P}(\TD,q)S_{\left<q\right>}v = \phi, \\
	\left<q\right>^{-{\mu}_k + 1/2}T(q) S_{\left<q\right>}v + \sum_{i=1}^J \left<q\right>^{-\tau_j - \mu_k}C_{k,j}(q)\UL{v} = \UL{\psi}
	\end{cases}
	\]
	tends to that of $\mathscr{P}^\circ(\TD,\left<q\right>^{-1} q)$ as $|q| \rightarrow \infty$. Combined with the regularity assumption, the former problem is invertible for all $q \in \mathbb{Z}^{n-1}$ with an inverse whose operator norm is uniformly bounded in $q$. Apply this invertibility result to the functions
	\begin{equation*} 
	v = S_{\left<q\right>^{-1}} \hat{u}(q), \quad \underline{v} = (\left<q\right>^{\tau_1 + 1/2} \hat{\UL u}_1(q), \ldots, \left<q\right>^{\tau_J + 1/2} \hat{\UL u}_J(q)).
	\end{equation*}
	This implies that
	\begin{multline}
	\| S_{\left<q\right>^{-1}} \hat{u}(q) \|^2_{\Sob^2(\RNP)} +  \left< q \right>^{1+2\UL\tau} \| \hat{\UL{u}}(q) \|^2_{\mathbb{C}^J} \\ \leq C ( \left< q \right>^{-4} \| S_{\left<q\right>^{-1}} P(\TD,q)\hat{u}(q) \|_{\Sob^0(\RNP)}^2 \\ + \left<q\right>^{1-2\UL{\mu}} \|T(q)\hat{u}(q) + C_{k,j}(q) \hat{u}(q) \|^2_{\mathbb{C}^{J+1}} ). \label{eq:fourierestimate}
	\end{multline}
	From \eqref{eq:fourierestimate} it follows that $\mathscr{P}$ is injective. Now multiply this equation by $\left<q\right>^{2s-1} = \left<q\right>^{3}$ and sum over $q\in \mathbb{Z}^{n-1}$. Then Lemma \ref{lem:fouriernorm} shows that the Fourier series for $(u,\UL{u})$ converges in $\Sob^2(\RNP)\times H^{2+\UL{\tau}}(\RN)$. Combined with the fact that $\mathscr{P}(\TD,q)$ is invertible for each $q\in\mathbb{Z}^{n-1}$, this shows that $\mathscr{P}$ is surjective.
	
	\item The proof when $\nu \geq 1$ follows as above, disregarding the boundary operators.
\end{inparaenum}
	\end{proof}
	
\begin{cor}
	Assume that $P$ and $\mathscr{P}$ are elliptic. Furthermore, assume that the one dimensional operators $P(\TD,q)$ (if $\nu \geq 1$) and $\mathscr{P}(\TD,q)$ (if $0< \nu < 1$) are regular for each $q\in\mathbb{Z}^{n-1}$.
	\begin{enumerate} \itemsep6pt
		\item If $0 < \nu < 1$, then
		\[
		\mathscr{P}: \TSob^s(\RNP) \times H^{s+\UL{\tau}}(\RN) \rightarrow \Sob^{s-2}(\RNP) \times \Sob^{s-\UL{\mu}}(\RN)
		\]
		is an isomorphism for $s=0,1,2$.
		\item If $\nu \geq 1$, then
		\[
		P: \Sob^s(\RNP) \rightarrow \Sob^{s-2}(\RNP)
		\]
		is an isomorphism for $s=0,1,2$.
	\end{enumerate}
\end{cor} \label{cor:higherdimapriori}
\begin{proof}
	$(1)$ It remains to handle the cases $s=0,1$. First consider $s=0$. As in the proof of Lemma \ref{lem:1daprioris=0}, the formal adjoint
	\[
	\mathscr{P}^*: \Sob^2(\RNP) \times H^{\UL{\mu}}(\RN) \rightarrow \Sob^0(\RNP) \times H^{\UL{\nu}}(\RN) \times H^{-\UL{\tau}}(\RN)
	\] 
	agrees with the adjoint of 
	\[
	\mathscr{P}: \TSob^0(\RNP) \times H^{\UL{\tau}}(\RN) \rightarrow \Sob^{-2}(\RNP) \times H^{-\UL{\mu}}(\RN).
	\]
	Now $\mathscr{P}^*$ satisfies the same hypotheses as $\mathscr{P}$ in regards to the application of Lemma \ref{lem:higherdimapriori}, so is an isomorphism. This implies that $\mathscr{P}'$ is an isomorphism, hence so is $\mathscr{P}$ on the stated spaces.
	
	The case $s=1$ follows from \eqref{eq:fourierestimate} combined with Lemma \ref{lem:1daprioris=1}: indeed, multiplying the analogue of \eqref{eq:fourierestimate} by $\left<q\right>^{2s-1} = \left<q\right>$ and using the invertibility result from Lemma \ref{lem:1daprioris=1} shows that $\mathscr{P}$ is surjective on $\TSob^1(\RNP)\times H^{1+\UL{\tau}}(\RNP)$ (as well as injective by the $s=0$ case).
	
	$(2)$ As usual, when $\nu \geq 1$ the proof follows by dropping the boundary terms.
	\end{proof}
	
\begin{rem}
If $P(\TD,D_y)$ is elliptic, then $P^\circ(\TD,D_y + \tfrac{1}{2})$ satisfies the hypotheses of Lemma \ref{lem:higherdimapriori}. Similarly, if $0< \nu < 1$ and $\mathscr{P}(\TD,D_y)$ is elliptic, then $\mathscr{P}^\circ(\TD,D_y+\tfrac{1}{2})$ also satisfies the hypotheses of Lemma \ref{lem:higherdimapriori}.
\end{rem}

\subsection{Elliptic Bessel operators on $\RNP$ with variable coefficients}

In this section, let $P$ be a Bessel operator on $\RNP$ of the form
\[
P(x,y,\TD,D_y) = |\TD|^2 + B(x,y,D_y)\TD + A(x,y,D_y),
\]
where the coefficients of $A, B$ are constant outside a compact subset of $\OL{\RNP}$. If $0 < \nu < 1$, then $P$ is also augmented by boundary conditions $(T(y,D_y),C(y,D_y))$. Introduce the notation 
\begin{gather*}
P^{(0)}(\TD,D_y) := P^\circ(0,0,\TD,D_y+\tfrac{1}{2}),\\ T^{(0)}(D_y) = T^\circ(0,D_y+\tfrac{1}{2}), \quad C^{(0)}(D_y) = C^\circ(0,D_y+ \tfrac{1}{2}).
\end{gather*}
According to Lemma \ref{lem:higherdimapriori}, if $P$ and $\mathscr{P}$ are elliptic, then $P^{(0)}$ (if $\nu \geq 1$) and $\mathscr{P}^{(0)}$ (if $0 < \nu < 1$) are isomorphisms on the appropriate spaces. 

Given $\rho > 0$, define the Fourier multiplier $K_\rho = {1}_{|x| \geq \rho}(D_y)$. This operator acts both on Sobolev spaces $H^m(\RN)$, as well as on $\Sob^s(\RNP)$ (or $\TSob^s(\RNP)$) via the results of Section \ref{subsect:fourier}. If $m > m'$ then clearly
\begin{equation} \label{eq:SrhoH}
\| K_\rho \phi \|_{H^{m'}(\RN)} \leq \left<\rho\right>^{m'-m} \| \phi \|_{H^m(\RN)}
\end{equation}
for $\phi \in H^m(\RN)$. Similarly,
\begin{equation} \label{eq:SrhoSob}
\| K_\rho u \|_{\Sob^s(\RNP)} \leq C\left<\rho \right>^{-|\alpha|} \| D_y^\alpha u \|_{\Sob^s(\RNP)} \leq C\left<\rho \right>^{-|\alpha|} \| u \|_{\Sob^{s+|\alpha|}(\RNP)}
\end{equation}
for $u \in \Sob^s(\RNP)$, provided $s+|\alpha| \leq 2$. A similar statement holds for $(u,\UL{\phi}) \in \TSob^s(\RNP)$.

\begin{lem} \label{lem:localizedvariableRNP}
Assume that $P$ and $\mathscr{P}$ are elliptic at $\RN$. Then there exists $\delta >0$ such that the following hold.

	\begin{enumerate} \itemsep6pt
		\item Let $0 < \nu < 1$ and $s=0,1,2$. Suppose that $(u,\UL{\phi},\UL{u}) \in \TSob^{s}(\RNP) \times H^{s+\UL{\tau}}(\RN)$ satisfies
			\begin{gather*}
			\supp u \subseteq \{(x,y)\in\OL{\RNP}: |x| + |y| < \delta  \}, \\ \supp \UL{\phi} \subseteq \{y \in \RN: |y| < \delta \}, \quad \supp \UL{u} \subseteq \{ y \in \RN: |y| < \delta\}.
			\end{gather*}
			Then 
			\begin{align} \label{eq:variablecoefficientestimateRNP}
			\| (u, \UL{\phi}, \UL{u}) \|_{\TSob^{s}(\RNP) \times H^{s+\UL{\tau}}(\RN)} &\leq C( \| \mathscr{P}(u,\UL{\phi},\UL{u}) \|_{\Ran{s}{\RNP}{\RN}} \notag \\ &+ \| (u, \UL{\phi}, \UL{u}) \|_{\TSob^{s-1}(X) \times H^{s-1+\UL{\tau}}(\RN)} ),
			\end{align}
			where $C>0$ does not depend on $(u,\UL{\phi},\UL{u})$. In addition, if $s=0,1$ and 
			\[\mathscr{P}(u,\UL{\phi},\UL{u}) \in \Sob^{s-1}(\RNP) \times \Sob^{s-\UL{\mu}+1}(\RN),
			\]
			then $(u,\UL{\phi},\UL{u}) \in \TSob^{s+1}(\RNP) \times H^{s+\UL{\tau}+1}(\RNP)$.
			
	\item Let $\nu \geq 1$ and $s=0,1,2$. Suppose that $u \in \Sob^2(\RNP)$ satisfies
	\[
	\supp u \subseteq \{(x,y)\in\OL{\RNP}: |x| + |y| < \delta  \}
	\]
	Then 
	\[
	\| u \|_{\Sob^{s}(\RNP)} \leq C ( \| Pu \|_{\Sob^{s-2}(\RNP)} + \| u \|_{\Sob^{s-1}(\RNP)}),
	\]
		where $C>0$ does not depend on $u$. In addition, if $s=0,1$ and $Pu \in \Sob^{s-1}(\RNP)$, then $u \in \Sob^{s+1}(\RNP)$.
	\end{enumerate}	
	\end{lem}
	\begin{proof}
\begin{inparaenum}
\item For concreteness, assume that $s=1$ and $\mathscr{P}(u,\UL{\phi},\UL{u}) \in \Sob^0(\RNP) \times H^{2-\UL{\mu}}(\RN)$. If $(f,\UL{g}) = \mathscr{P}(u,\UL{\phi},\UL{u})$, consider the identity
		\begin{multline} \label{eq:perturbidentity}
		\mathscr{P}^{(0)}(u,\UL{\phi},\UL{u}) + (\mathscr{P}-\mathscr{P}^{(0)})(K_\rho (u,\UL{\phi}) ,K_\rho \UL{u}) \\  = (f,\UL{g}) - (\mathscr{P}-\mathscr{P}^{(0)})((1-K_\rho)(u,\UL{\phi}), (1-K_\rho)\UL{u}).
		\end{multline}
		Noting that the term $(P-P^{(0)})(u,\UL{\phi})$ depends only on $u$ (and not on $\UL{\phi})$, it follows from Lemmas \ref{lem:Bbounded}, \ref{lem:Abounded} and \eqref{eq:SrhoSob} that
		\begin{align*}
		\|(P-P^{(0)})K_\rho u\|_{\Sob^0(\RNP)} &\leq C_1 \delta \| u\|_{\Sob^2(\RNP)} + C_2 \| K_\rho u \|_{\Sob^1(\RNP)}\\ &\leq (C_1\delta + C_2 \left< \rho \right>^{-1}) \| (u,\UL{\phi}) \|_{\TSob^2(\RNP)}
		\end{align*}
		for positive constants $C_1, C_2$ independent of $\rho$. By standard interpolation inequalities on $H^m(\RN)$,
		\begin{align*}
		\| (T_k - T^{(0)}_k) K_\rho \UL{\phi} \|_{H^{2-\mu_k}(\RN)} &\leq C_3\delta \| \UL{\phi} \|_{H^{2-\UL{\nu}}(\RN)} + C_4 \| K_\rho \UL{\phi} \|_{H^{1-\UL{\nu}}(\RN)} \\ &\leq (C_3 \delta + C_4 \left<\rho \right>^{-1}) C_5 \| (u,\UL{\phi}) \|_{\TSob^2(\RNP)}.
		\end{align*}
		For this, one should consider the cases $0< \nu < 1/2, \, \nu =1/2$, and $1/2 < \nu < 1$ separately, but they all yield the same type of the estimate. Similarly,
		\[
		\| (C - C^{(0)})K_\rho \UL{u} \|_{H^{2-\UL{\mu}}(\RN)} \leq ( C_6 \delta + C_7\left< \rho \right>^{-1}) \| \UL{u} \|_{H^{s+\UL{\tau}}(\RN)}.
		\]
		These inequalities imply that the operator norm of 
		\[
		(u,\UL{\phi},\UL{u}) \mapsto (\mathscr{P} - \mathscr{P}^{(0)})(K_\rho(u,\UL\phi),K_\rho \UL{u})
		\]
		can be made arbitrarily small by choosing $\delta >0$ small and $\rho > 0$ large. Since $\mathscr{P}^{(0)}$ is invertible with domain $\TSob^2(\RNP) \times H^{2+\UL{\tau}}(\RN)$, it follows that the operator on the left hand side of \eqref{eq:perturbidentity} is invertible for $\delta$ small and $\rho$ large. 
		
		On the other hand, the map
		\[
		(u,\UL{\phi},\UL{u}) \mapsto (\mathscr{P}-\mathscr{P}^{(0)})((1-K_\rho)(u,\UL{\phi}), (1-K_\rho)\UL{u})
		\] 
		is bounded $\TSob^1(\RNP) \times H^{1+\UL{\tau}}(\RN) \rightarrow \Sob^0(\RNP) \times H^{2-\UL{\mu}}(\RN)$. In particular, $(u,\UL{\phi},\UL{u}) \in \TSob^2(\RNP) \times H^{2+\UL{\tau}}(\RN)$, and the estimate \eqref{eq:variablecoefficientestimateRNP} holds. Of course this also implies that \eqref{eq:variablecoefficientestimateRNP} holds for arbitrary $(u,\UL{\phi},\UL{u}) \in \TSob^2(\RNP)\times H^{2+\tau}(\RN)$ as well. The exact same argument establishes the regularity result for $s=0$, as well as \eqref{eq:variablecoefficientestimateRNP} for $s=0,1$.
		
\item As usual, the case $\nu \geq 1$ can be handled by a simpler argument not involving the boundary operators.
	\end{inparaenum}
				\end{proof}

Lemma \ref{lem:localizedvariableRNP} can be semi-globalized via a partition of unity argument.

\begin{cor} \label{cor:localizedvariableRNP}
	Assume that $P$ and $\mathscr{P}$ are elliptic at $\RN$. There exists $\delta > 0$ such that if $\varphi, \chi \in C_c^\infty([0,\delta))$ satisfy $\varphi = 1$ near $x=0$ and $\chi = 1$ near $\supp \varphi$, then the following hold.
	
		\begin{enumerate} \itemsep6pt
			\item Let $0 < \nu < 1$ and $s=0,1,2$. Then 
			\begin{align} \label{eq:nuleq1RNPsemiglobal}
			\| \varphi(u, \UL{\phi}, \UL{u}) \|_{\TSob^{s}(\RNP) \times H^{s+\UL{\tau}}(\RN)} &\leq C( \| \varphi \mathscr{P}(u,\UL{\phi},\UL{u}) \|_{\Ran{s}{\RNP}{\RN}} \notag \\ &+ \|\chi (u, \UL{\phi}, \UL{u}) \|_{\TSob^{s-1}(\RNP) \times H^{s-1+\UL{\tau}}(\RN)} )
			\end{align}
			 for each $(u,\UL{\phi},\UL{u})\in \TSob^{s}(\RNP) \times \Sob^{s+\UL{\tau}}(\RN)$. In addition, if $s=0,1$ and 
			\[
			\varphi \mathscr{P}(u,\UL{\phi},\UL{u}) \in \Sob^{s-1}(\RNP) \times \Sob^{s-\UL{\mu}+1}(\RN),
			\]
			then $\varphi (u,\UL{\phi},\UL{u}) \in \TSob^{s+1}(\RNP) \times H^{s+\UL{\tau}+1}(\RN)$.
			
			\item Let $\nu \geq 1$ and $s=0,1,2$. Then 
			\begin{equation} \label{eq:nugeq1RNPsemiglobal}
			\| \varphi u \|_{\Sob^{s}(\RNP)} \leq C ( \| \varphi Pu \|_{\Sob^{s-2}(\RNP)} + \| \chi u \|_{\Sob^{s-1}(\RNP)})
			\end{equation}
			for each $u \in \Sob^s(\RNP)$. In addition, if $s=0,1$ and $\varphi Pu \in \Sob^{s-1}(\RNP)$, then $\varphi u \in \Sob^{s+1}(\RNP)$.
		\end{enumerate}	
	\end{cor}
	\begin{proof} [Sketch of proof for $0< \nu <1$]
		By compactness of $\RN$ it is possible to choose $\delta$ and a finite cover $\RN = \bigcup_i U_i$ such that Lemma \ref{lem:localizedvariableRNP} is valid for $(u,\UL{\phi},\UL{u})$ supported in $[0,\delta) \times U_i$. Fix a partition of unity $\beta_i$ subordinate to $U_i$, and choose $\gamma_i$ supported in $U_i$ that $\gamma_i = 1$ on $\supp \beta_i$. For $\varphi, \chi$ as in the statement of the corollary,
		\begin{multline*}
		\| \varphi(u, \UL{\phi}, \UL{u}) \|_{\TSob^{s}(\RNP) \times H^{s+\UL{\tau}}(\RN)} \leq  C_1\| \varphi \mathscr{P}(u,\UL{\phi},\UL{u}) \|_{\Sob^s(\RNP)\times H^{s-\UL{\mu}(\RN)}}\\
		 + \sum_i \| [\mathscr{P},\beta_i \varphi] \gamma_i \chi (u,\UL{\phi},\UL{\mu}) \|_{\Sob^s(\RNP)\times H^{s-\UL{\mu}(\RN)}} \\
		 + C_2 \| \varphi (u,\UL{\phi},\UL{u}) \|_{\Sob^{s-1}(\RNP) \times H^{s-1+\UL{\tau}}(\RN)}.
		\end{multline*}
Writing $\varphi_i = \beta_i \varphi$, the commutator $[\mathscr{P},\beta_i\varphi]$ is given by
\[
(u,\UL{\phi},\UL{u}) \mapsto (P(\varphi_i u, \beta_i \UL{\phi}) - \varphi_i P(u,\UL{\phi}), [G,\beta_i] \UL{\phi} + [C,\beta_i]\UL{u} ).
\]
It is then straightforward to check that this operator has the requisite mapping properties. The regularity statement is established in the same way.
	\end{proof}

\begin{rem} As usual, the norms of the lower order terms on the right hand sides of \eqref{eq:nuleq1RNPsemiglobal}, \eqref{eq:nugeq1RNPsemiglobal} can be taken in less regular Sobolev spaces by iterating Corollary \ref{cor:localizedvariableRNP}. Similarly, the regularity result can also be iterated.
	\end{rem}

\subsection{Elliptic Bessel operators on a compact manifold with boundary} \label{subsect:variable}

The main theorem in this section establishes elliptic estimates and elliptic regularity for elliptic Bessel operators on a compact manifold with boundary $\OL{X}$.

\begin{theo} \label{theo:bvptheo}
	Let $\OL{X}$ be a compact manifold with boundary as in Section \ref{subsect:manifoldwithboundary}. Assume that 
	\[
	P \in \mathrm{Bess}_\nu(X)
	\]
	is elliptic at $\partial X$ in the sense of Section \ref{subsect:ellipticity}. If $0 < \nu < 1$, then assume $P$ is augmented by boundary conditions $(T,C)$ such that 	$\mathscr{P} = \{P,T,C\}$
is elliptic at $\partial X$. There exists $\delta > 0$ such that if $\varphi, \chi \in C_c^{\infty}(\{ 0 \leq x < \delta\})$ satisfy $\varphi = 1$ near $\partial X$ and $\chi = 1$ near $\supp \varphi$, then the following hold.
	
		\begin{enumerate} \itemsep6pt
			\item Let $0 < \nu < 1$ and $s=0,1,2$. Then 
			\begin{align} \label{eq:nuleq1RNPgloba}
			\| \varphi(u, \UL{\phi}, \UL{u}) \|_{\TSob^{s}(X) \times H^{s+\UL{\tau}}(\partial X)} &\leq C( \| \varphi \mathscr{P}(u,\UL{\phi},\UL{u}) \|_{\Ran{s}{X}{\partial X}} \notag \\ &+ \|\chi (u, \UL{\phi}, \UL{u}) \|_{\TSob^{s-1}(X) \times H^{s-1+\UL{\tau}}(\partial X)} )
			\end{align}
			for each $(u,\UL{\phi},\UL{u})\in \TSob^{s}(X) \times \Sob^{s+\UL{\tau}}(\partial X)$. In addition, if $s=0,1$ and 
			\[
			\varphi \mathscr{P}(u,\UL{\phi},\UL{u}) \in \Sob^{s-1}(X) \times \Sob^{s-\UL{\mu}+1}(\partial X),
			\]
			then $\varphi (u,\UL{\phi},\UL{u}) \in \TSob^{s+1}(X) \times H^{s+\UL{\tau}+1}(\partial X)$.
			
			\item Let $\nu \geq 1$ and $s=0,1,2$. Then 
			\begin{equation} \label{eq:nugeq1RNPglobal}
			\| \varphi u \|_{\Sob^{s}(X)} \leq C ( \| \varphi Pu \|_{\Sob^{s-2}(X)} + \| \chi u \|_{\Sob^{s-1}(X)})
			\end{equation}
			for each $u \in \Sob^s(X)$. In addition, if $s=0,1$ and $\varphi Pu \in \Sob^{s-1}(X)$, then $\varphi u \in \Sob^{s+1}(X)$.
		\end{enumerate}	
\end{theo}
\begin{proof}
The global problem may be reduced to a local problem on $\RNP$ via coordinate charts and a partition of unity.
\end{proof}

As in the remark following Corollary \ref{cor:localizedvariableRNP}, the error terms in Theorem \ref{theo:bvptheo} can taken in weaker Sobolev spaces by iteration.

Recall the definition of $\Sob^s_P(X)$ in Section \ref{subsect:graphnorm}. Theorem \ref{theo:bvptheo} can be used to show that $\Sob^s_P(X)$ (or equivalently $\TSob^s(X)$) may be identified with the space of all pairs $(u,f) \in \Sob^s(X)\times\Sob^{s-2}(X)$ such that $Pu =f $ in the weak sense (cf. \cite[Chapter 6.1]{roitberg:1996}).

\begin{lem} \label{lem:graphnorm2}
	Let $0 < \nu < 1$, and suppose that $P$ is elliptic at $\partial X$. Then for $s=0,1,2$,
	\[
	\Sob^s_P(X) = \{ (u,f) \in \Sob^{s}(X) \times \Sob^{s-2}(X): \, Pu =f \text{ weakly} \},
	\]
	where the space on the right hand side is equipped with the $\Sob^s_P(X)$ norm.
	
\end{lem}
\begin{proof}
	As in the remark following Lemma \ref{lem:graphnorm}, $\Sob^s_P(X)$ is contained in the space on the right hand side. For the converse, suppose that $u \in \Sob^s(X)$ and $f = Pu \in \Sob^{s-2}(X)$ weakly. Consider the functional
	\[
	\ell(\UL{\psi}) = \left< u, P^*(v,\UL{\psi}) \right>_X - \left< f, v \right>_X, \quad \UL{\psi} \in H^{2-s-\UL{\nu}}(\partial X),
	\]
	where $v \in \Sob^{2-s}(X)$ is any element such that  $(v,\UL{\psi}) \in \TSob^{2-s}(X)$.  Since $Pu = f$ weakly, it follows from Lemma \ref{lem:mathring} that $\ell$ is well defined (namely it does not depend on the choice of $v$). In particular, one may take $(v,\UL{\psi}) = \widetilde{\mathcal{K}}\UL{\psi}$, where $\widetilde{\mathcal{K}}$ is a bounded right inverse as in Lemma \ref{lem:tracelift}. Thus
	\[
	\ell(\UL{\psi}) \leq C_1 \| u \|_{\Sob^s_P(X)} (\| \widetilde{ \mathcal{K}}\UL{\psi} \|_{\TSob^{2-s}(X)} + \| \UL{\psi} \|_{H^{2-s-\UL{\nu}}(\partial X)}) \leq C_2 \| u \|_{\Sob^s_P(X)} \| \UL{\psi}\|_{H^{2-s-\UL{\nu}}(\partial X)}.
	\]
	By the Riesz theorem, there exists a unique $\UL{\phi} \in H^{s-\UL{\nu}}(\partial X)$ such that
	\[
	\left< u, P^*(v,\UL{\psi}) \right>_X - \left< f, v \right>_X = \left< J\UL{\phi}, \UL{\psi} \right>_{\partial X}.
	\] 
	for each $(v,\UL{\psi}) \in \TSob^{2-s}(X)$. Consider the pair $(u, \UL{\phi})$; a priori this is an element of $\TSob^0(X)$. On the other hand, for each $v \in \mathcal{F}_\nu(X)$ (taking $\UL{\psi} = \UL{\gamma} v$),
	\begin{align*}
	\left< P(u,\UL{\phi}),v \right>_X &= \left< u, P^*v \right>_X + \left< \UL{\phi}, J\UL{\gamma}v \right>_{\partial X} \\
&= \left<f ,v\right>_X,
	\end{align*}
	so $P(u,\UL{\phi}) = f$. Since $f \in \Sob^{s-2}(X)$ and $\phi_- \in H^{s-1+\nu}(\partial X)$, Theorem \ref{theo:bvptheo} implies that $(u,\UL{\phi}) \in \TSob^s(X)$ since the the boundary value problem $\{P, \gamma_-\}$ is elliptic at $\partial X$. According to Lemma \ref{lem:graphnorm}, this means that the pair $(u,f)$ can be identified with an element of $\Sob^s_P(X)$.
	\end{proof}

	Suppose that $P$ is elliptic at $\partial X$ and let $s = 0,1$. If $u \in \Sob^s(X)$ and $Pu \in \Sob^0(X)$ in distributions, then there is a canonical $f \in \Sob^{s-2}(X)$ such that $Pu = f$ weakly, namely the element $Pu \in \Sob^0(X)\hookrightarrow \Sob^{s-2}(X)$ itself. According to Lemma \ref{lem:graphnorm2}, to this choice of $f$ there is a uniquely associated $\phi \in \Sob^{s-\UL{\nu}}(\partial X)$ such that 
	\[
	P(u,\UL{\phi}) = Pu,
	\] and the norm $\| (u,\UL{\phi}) \|_{\TSob^s(X)}$ is equivalent to $\| u \|_{\Sob^s(X)} + \| Pu \|_{\Sob^{2-s}(X)}$. Adding $\| Pu \|_{\Sob^0(X)}$ to both of these norms shows that the spaces
	\[
	\{u \in \Sob^s(X): Pu \in \Sob^0(X) \} \text{ and } \{ (u,\UL{\phi}) \in \TSob^s(X): P(u,\UL{\phi}) \in \Sob^0(X) \}
	\]
	coincide, with an equivalence between the natural graph norms. This will be exploited in Section \ref{subsect:globalass}.

\subsection{Parameter-elliptic boundary value problems} \label{subsect:semiclassicalvariable} This section concerns elliptic estimates for parameter-dependent Bessel operators. The exposition is deliberately brief, since most of the definitions and facts in this section are straightforward adaptations from the non-parameter-dependent setting. In particular the main theorem of this section, Theorem \ref{theo:bvptheosemiclassical}, is stated without proof. The interested reader is referred to \cite[Chapter 9]{roitberg:1996} for an indication of how the proofs should be modified in the parameter-dependent setting.

Fix a compact manifold with boundary $\OL{X}$ with the usual data of a boundary defining function and collar diffeomorphism. Let $P(\lambda) \in \mathrm{Bess}_\nu^{(\lambda)}(X)$ be a parameter-dependent Bessel operator; if $0 < \nu < 1$, then $P(\lambda)$ is augmented by boundary conditions as in Section \ref{subsect:theBVP}. The boundary conditions themselves may depend on the spectral parameter $\lambda$, namely one considers $(T(\lambda),C(\lambda))$ where
\[
T_k(\lambda) = \sum_\pm T^\pm_k(\lambda) \gamma_\pm, \quad T^-_k(\lambda) \in \mathrm{Diff}^1_{(\lambda)}(\partial X), \quad  T^+_k(\lambda) \in \mathrm{Diff}^0_{(\lambda)}(\partial X),
\]
and $C_{k,j}(\lambda) \in \mathrm{Diff}^*_{(\lambda)}(\partial X)$. It is necessary to formulate a parameter-dependent Lopatinski\v{\i} condition for $(T(\lambda),C(\lambda))$. Suppose that $\mu_k \in \{1-\nu,2-\nu,1+\nu\}$ and 
\[
\mathrm{ord}^{(\lambda)}_\nu(T_k(\lambda)) \leq \mu_k.
\] 
Here the order of $T$ with respect to $\nu$ is defined in the parameter-dependent sense, namely factors of $\lambda$ are given the same weight as a derivative tangent to $\partial X$. Define the family of operators 
\[
(\widehat{T}_{(p,\eta;\lambda)})_k = \sigma^{(\lambda)}_{\lceil \mu_k-1+\nu \rceil}(T^-(\lambda))(p,\eta;\lambda) \gamma_-  +  \sigma^{(\lambda)}_{\lceil\mu_k-1-\nu\rceil}(T^+(\lambda))(p,\eta;\lambda)\gamma_+,
\]
indexed by $(p,\eta,\lambda) \in T^* \partial X \times \mathbb{C}$. Thus each $(p,\eta,\lambda) \in T^*\partial X \times \mathbb{C}$ gives rise to a one-dimensional boundary operator $(\widehat{T}_{(p,\eta;\lambda)})_k$.

Next, choose $c_{k,j} \in \mathbb{Z}$ such that 
\[
\mathrm{ord}^{(\lambda)}(C_{k,j}(\lambda)) \leq c_{k,j} \leq \tau_j + \mu_k,
\]
and then define the matrix $\widehat{C}_{(y,\eta)}$ with entries 
\[
(\widehat{C}_{(p,\eta;\lambda)})_{k,j} = \sigma_{c_{k,j}}^{(\lambda)}(C_{k,j}(\lambda))(p,\eta;\lambda).
\] 
Again the order of $C_{k,j}(\lambda)$ is taken in the parameter-dependent sense.

\begin{defi} Suppose that $P(\lambda)$ is parameter-elliptic on $\partial X$ with respect to an angular sector $\Lambda$. The boundary operators $(T(\lambda),C(\lambda))$ are said to satisfy the parameter-dependent Lopatinski\v{\i} condition with respect to $P$ and $\Lambda$ if for each $p \in \partial X$ and $(\eta,\lambda) \in T^*_p\partial X \times \Lambda \setminus 0$, the only element $(u,\underline{u}) \in \mathcal{M}_+(p,\eta,\lambda) \times \mathbb{C}^J$ satisfying
\[ 
\widehat{T}_{(p,\eta,\lambda)}u + \widehat{C}_{(p,\eta,\lambda)}\underline{u} = 0 \\
\]
is the trivial solution $(u,\underline{u}) = 0$. The operator $\mathscr{P}(\lambda) = \{P(\lambda),T(\lambda),C(\lambda)\}$, is said to be parameter elliptic if $P(\lambda)$ is parameter-elliptic and $(T(\lambda),C(\lambda))$ satisfy the parameter-dependent Lopatinski\v{\i} condition on $\partial X$ with respect to $P(\lambda)$ and $\Lambda$.
\end{defi}

In the notation of Theorem \ref{theo:bvptheo}, the main theorem of this section is the following. As remarked in the introduction to this section, it is provided without proof.

\begin{theo} \label{theo:bvptheosemiclassical}
	Let $\OL{X}$ be a compact manifold with boundary as in Section \ref{subsect:manifoldwithboundary}. Assume that 
	\[
	P(\lambda) \in \mathrm{Bess}_\nu^{(\lambda)}(X)
	\] 
	is parameter-elliptic at $\partial X$ with respect to an angular sector $\Lambda$ in the sense of Section \ref{subsect:semiclassicalbessel}. If $0 < \nu < 1$, then assume $P(\lambda)$ is augmented by parameter-dependent boundary conditions $(T(\lambda),C(\lambda))$ such that $\mathscr{P}(\lambda) = \{P(\lambda),T(\lambda),C(\lambda)\}$
	 is elliptic at $\partial X$ with respect to $\Lambda$. There exists $\delta > 0$ such that if $\varphi, \chi \in C_c^{\infty}(\{ 0 \leq x < \delta\})$ satisfy $\varphi = 1$ near $\partial X$ and $\chi = 1$ near $\supp \varphi$, then the following hold.
	
	\begin{enumerate} \itemsep6pt
		\item Let $0 < \nu < 1$ and $s=0,1,2$. Then 
		\begin{align} \label{eq:variablecoefficientestimateRNPsemiglobalsemiclassical}
		\VERT \varphi(u, \UL{\phi}, \UL{u}) \VERT_{\TSob^{s}(X) \times H^{s+\UL{\tau}}(\partial X)} &\leq C( \VERT \varphi \mathscr{P}(\lambda)(u,\UL{\phi},\UL{u}) \VERT_{\Ran{s}{X}{\partial X}} \notag \\ &+ \VERT \chi (u, \UL{\phi}, \UL{u}) \VERT_{\TSob^{s-1}(X) \times H^{s-1+\UL{\tau}}(\partial X)} )
		\end{align}
		for each $(u,\UL{\phi},\UL{u})\in \TSob^{s}(X) \times \Sob^{s+\UL{\tau}}(\partial X)$ and $\lambda \in \Lambda$.
		
		\item Let $\nu \geq 1$ and $s=0,1,2$. Then 
		\begin{equation} \label{eq:nugeq1RNPglobalsemiclassical}
		\VERT \varphi u \VERT_{\Sob^{s}(X)} \leq C ( \VERT \varphi P(\lambda)u \VERT_{\Sob^{s-2}(X)} + \VERT \chi u \VERT_{\Sob^{s-1}(X)})
		\end{equation}
		for each $u \in \Sob^s(X)$ and $\lambda \in \Lambda$.
	\end{enumerate}	
\end{theo}

\subsection{Conormal regularity} \label{subsect:conormal}

So far only regularity at the $\Sob^2$ level has been discussed. Higher order regularity is defined in terms of a scale of conormal Sobolev spaces relative to $\Sob^s$. Let $\OL{X}$ be a compact manifold with boundary with a fixed boundary defining function $x$ and collar neighborhood $\OL{W}$. Then let $\OL{X}_\mathrm{even}$ denote the manifold $\OL{X}$ equipped with a new smooth structure: on the collar $\OL{W} \simeq [0,\varepsilon)_x\times \partial X$, functions are smooth if in the normal direction they depend on $x^2$ (rather than just $x$).

Define the Lie algebra $\mathcal{V}_b(\OL{X}_\mathrm{even})$ of smooth vector fields on $\OL{X}_\mathrm{even}$ which are tangent to $\partial X$. In local coordinates $x,y_1,\ldots,y_{n-1}$ on the collar, elements of $\mathcal{V}_b(\OL{X}_\mathrm{even})$ are $C^\infty(\OL{X}_\mathrm{even})$ linear combinations of $x\partial_x$ and $\partial_{y_i}$. 
\begin{lem} \label{lem:conormalcommutator}
Let $P \in \mathrm{Bess}_\nu(X)$. If $V \in \mathcal{V}_b(\OL{X}_\mathrm{even})$ and $x^{-1}Vx|_{\partial X}$ is nowhere vanishing, then there exists $f \in C^\infty(\OL{X})$ and $\widetilde{P} \in \mathrm{Bess}_\nu(X)$ such that
\[
[P,V] = f \widetilde P
\]
near $\partial X$.
\end{lem}
\begin{proof}
The hypothesis implies that in local coordinates
\[
V (x,y)=a(x^2,y)x\partial_x + \sum_{i=1}^{n-1} b^i(x^2,y) \partial_{y_i},
\] 
where $a(0,\cdot)$ is nowhere vanishing. Note that 
\[
[ |\TD|^2, x\partial_x ] = 2 |\TD|^2, \quad [\TD,x\partial_x ] = \TD.
\]
Also from \eqref{eq:commutatorformulas}, if $a \in C^\infty(\OL{X}_\mathrm{even})$, then
\[
[|\TD|^2,a] = \hat{a}x\TD + \tilde{a}
\]
for $\hat{a},\tilde{a} \in C^\infty(\OL{X}_\mathrm{even})$, as well as $[x\TD,a] \in  x^2 C^\infty(\OL{X}_\mathrm{even})$. The result follows immediately from these observations.
\end{proof}

Given $k \in \mathbb{N}$ and $s = 0,1,2$, the space $\Sob^{s,k}(X)$ is defined as
\[
\Sob^{s,k}(X) = \{ u \in \Sob^s(X): V_1 \cdots V_k u \in \Sob^s(X) \text{ for any } V_1,\ldots,V_k \in \mathcal{V}_b(\OL{X}_\mathrm{even}) \}.  
\]
Fixing a finite generating set $\mathcal{G}$ for $\mathcal{V}_b(\OL{X}_\mathrm{even})$, this space can be given the topology of a Hilbert space by inductively defining the norms
\[
\| u \|^2_{\Sob^{s,k}(X)} = \sum_{V\in \mathcal{G}} \| Vu \|^2_{\Sob^{s,k-1}(X)}.
\]
A different choice of generating set yields an equivalent norm. Note that over any compact $K\subseteq X$, there is an equivalence between functions in $\Sob^{s,k}(X)$ and $H^{s+k}(X)$ which are supported on $K$. In addition, all of the density results which hold for $\Sob^{s}(X)$ also hold for $\Sob^{s,k}(X)$.

Also observe that for $s=0,1$, no evenness assumptions are required for the vector fields tangent to $\partial X$ in the sense that there is equality
\[
\Sob^{s,k}(X) = \{ u \in \Sob^s(X): V_1 \cdots V_k u \in \Sob^s(X) \text{ for any } V_1,\ldots,V_k \in \mathcal{V}_b(\OL{X}) \}.  
\]
This is because $\Sob^s(X)$ is closed under multiplication by arbitrary $C^\infty(\OL{X})$ functions when $s=0,1$. Thus only $\Sob^{2,k}(X)$ necessitates the introduction of a new smooth structure on $\OL{X}$.

 A convenient generating set $\mathcal{G} = \{V_0,V_1,\ldots,V_N\}$ for $\mathcal{V}_b(\OL{W}_\mathrm{even})$ (at least near the boundary) is as follows: set $V_0 = x\partial_x$, and then choose a collection of vector fields $V_1,\ldots,V_N$ on $\partial X$ which span $T\partial X$. Then $V_0,\ldots ,V_N$ may be considered as vector fields on $[0,\varepsilon)_x\times\partial X$, hence on $\OL{W}$.

\begin{lem}
	Let $P \in \mathrm{Bess}_\nu(X)$ and $k \in \mathbb{N}$.
	\begin{enumerate}\itemsep6pt
		\item If $\nu >0$, then $P: \Sob^{2,k}(X)\rightarrow \Sob^{0,k}(X)$ is bounded.
		\item If $0 < \nu < 1$ and $T$ is a boundary operator such that $\mathrm{ord}_\nu(T) \leq \mu$, then $T : \Sob^{2,k}(X) \rightarrow H^{k+2-\mu}(\partial X)$ is bounded.
	\end{enumerate} 
\end{lem}
\begin{proof}
	\begin{inparaenum}
		
\item  This follows by Lemma \ref{lem:conormalcommutator} and induction on $k$, since multiplication by $f \in C^\infty(\OL{X})$ is bounded on each $\Sob^{0,k}(X)$.

\item Given a vector field $Z$ on $\partial X$, there exists $V \in \mathcal{V}_b(\OL{X}_\mathrm{even})$ such that $V|_{\partial X} = Z$. Then $Z(\gamma_\pm u) = \gamma_\pm(Vu)$ for each $u\in \Sob^{2,k}(X)$; this is certainly true on $\mathcal{F}_\nu(X)$ and extends by density.
\end{inparaenum}
\end{proof}

Consider the generating set $\mathcal{G}$ as above. Note that the flow of $V_0$ is given by $\exp(hV_0)(x,y) = (e^h x, y)$, where $(x,y) \in [0, \varepsilon)_x\times \partial X$. Given $V \in \mathcal{G}$, let 
\[
\rho_{V,h} u = (u \circ \exp(hV) - u)/h
\]
denote the associated difference quotient. 

Suppose that $u \in \Sob^{2,k}(X)$ is supported in $\{ 0 \leq x < \delta\}$ for $\delta>0$ sufficiently small. Observe that there exists $h_0 > 0$ depending on $\delta$ such that $\rho_{V_0,h}u$ is well defined for $h\in (0,h_0)$; the difference quotients corresponding to $V_1,\ldots,V_N$ are defined for all $h > 0$. The first step is to calculate the commutator of $P$ with $\rho_{V_0,h}$; this is illustrated for $[|\TD|^2,\rho_{V_0,h}]$. First note that
\[
[\partial_x, \rho_{V_0,h}]u = h^{-1}(e^h -1)(\partial_x u)\circ \exp(hV_0).
\]
A short calculation gives
\[
[|\TD|^2,\rho_{V_0,h}]u = h^{-1}(1-e^{2h})(|\TD|^2u)\circ\exp(hV_0),
\]
which shows that 
\[
\| [|\TD|^2,\rho_{V_0,h}] u \|_{\Sob^{0,k}(X)} \leq C \| u \|_{\Sob^{2,k}(X)}
\]
for $h \in (0,h_0)$, where $C>0$ does not depend on $u$ or $h$. Continuing this calculation shows that 
\[
\| [P, \rho_{V,h}] u \|_{\Sob^{0,k}(X)} \leq C \| u \|_{\Sob^{2,k}(X)}
\] 
for any $V \in \mathcal{G}$. As for the boundary operators, one has
\[
\gamma_- (u\circ \exp(hV_0)) = \gamma_- u,\quad \gamma_+ (u\circ \exp(hV_0)) = e^{(1/2+\nu)h}\gamma_+ u,
\]
so $\gamma_- \circ \rho_{V_0,h} = 0$ and $\gamma_+ \circ \rho_{V_0,h} = (e^{(1/2+\nu)h}-1)\gamma_+u$. Similarly, 
\[
\| [T,\rho_{V_i,h}] u \|_{H^{k+2-\mu}(\partial X)} \leq C \| u \|_{\Sob^{2,k}(X)}
\]
for $i = 1,\ldots, N$, uniformly in $h$.

\begin{theo} \label{theo:conormal}
		Let $\OL{X}$ be a compact manifold with boundary as in Section \ref{subsect:manifoldwithboundary}. Assume that 
		\[
		P \in \mathrm{Bess}_\nu(X)
		\]
		is elliptic at $\partial X$ in the sense of Section \ref{subsect:ellipticity}. If $0 < \nu < 1$, then assume $P$ is augmented by a boundary condition $T$ such that $\mathscr{P} = \{P,T\}$ is elliptic at $\partial X$. There exists $\delta>0$ such that if $\varphi, \chi \in C_c^{\infty}(\{ 0 \leq x < \delta\})$ satisfy $\varphi = 1$ near $\partial X$ and $\chi = 1$ near $\supp \varphi$, then the following hold.
		\begin{enumerate} \itemsep6pt
			\item Let $0 < \nu < 1$. If $\chi u \in \Sob^{2}(X)$ and $\chi Pu \in \Sob^{0,k}(X), \, Tu \in H^{k+2-\mu}(\partial X)$ for some $k \in \mathbb{N}$, then $\varphi u \in \Sob^{2,k}(X)$. Furthermore,
			\[
			\| \varphi u \|_{\Sob^{2,k}(X)} \leq C\left(\| \chi \mathscr{P}u \|_{\Sob^{0,k}(X)\times H^{k+2-\mu}(\partial X)} + \| \chi u \|_{\Sob^{0}(X)} \right),
			\]
			where $C>0$ does not depend on $u$.
			\item  Let $\nu \geq 1$. If $\chi u \in \Sob^{2}(X)$  and $\chi Pu \in \Sob^{0,k}(X)$ for some $k \in \mathbb{N}$, then $\varphi u \in \Sob^{2,k}(X)$. Furthermore,
			\[
			\| \varphi u \|_{\Sob^{2,k}(X)} \leq C\left(\| \chi Pu \|_{\Sob^{0,k}(X)} + \| \chi u \|_{\Sob^{0}(X)} \right),
			\]
			where $C>0$ does not depend on $u$.
		\end{enumerate}
	\end{theo}
	\begin{proof}
		The proof is by induction; the case $k=0$ is Theorem \ref{theo:bvptheo}. Suppose that the result holds for $k \in \mathbb{N}$; combined with the calculations preceding the theorem, this gives that $\rho_{V,h}(\varphi u) \in \Sob^{2,k}(X)$ is well defined and uniformly bounded for each $V \in \mathcal{G}$ and  $h>0$ sufficiently small. Standard functional analysis (extracting a weakly convergent subsequence, etc.) proves that $V \varphi u \in \Sob^{2,k}(X)$ for every $V \in \mathcal{G}$, with a corresponding estimate. This allows one to conclude the result for $k+1$.
		\end{proof}

\subsection{Asymptotic expansions}

Using Mellin transform techniques and the conormal regularity guaranteed by Theorem \ref{theo:conormal}, it is straightforward to give asymptotic expansions for solutions of certain Bessel equations. This section is a special case of far more general expansions; see \cite[Section 7]{mazzeo:1991} and \cite[Proposition 8.10]{vasy:2012:apde} for example.

\begin{prop}
Suppose that $P$ and $\{P,T\}$ are elliptic at $\partial X$. If $u \in \Sob^0(X)$ and $Pu \in \dot{C}^\infty(X)$, then the following hold.
\begin{enumerate} \itemsep6pt
	\item Let $0 < \nu < 1$.  If $Tu \in C^\infty(\partial X)$, then there exist $u_\pm \in C^\infty(\OL{X})$ such that
	\[
	u = x^{1/2+\nu}u_+ + x^{1/2-\nu}u_-.
	\]
	In addition $u_\pm - g_\pm \in x^2 C^\infty(\OL{X})$, where $g_- = \gamma_- u$ and $2\nu g_+ = \gamma_+ u$.  
	\item Let $\nu \geq 1$. Then there exists $u_+ \in C^\infty(\OL{X})$ such that 
	\[
	u = x^{1/2+\nu}u_+. 
	\]
\end{enumerate}
\end{prop}

\section{The Fredholm alternative and unique solvability} \label{sect:fredholm}

\subsection{Global assumptions} \label{subsect:globalass} Let $\OL{X}$ denote a compact manifold with boundary as in Section \ref{subsect:manifoldwithboundary}. Consider 
\[
P \in \mathrm{Bess}_\nu(X).
\]
Assume that $P$ is elliptic at $\partial X$ in the sense of Section \ref{subsect:ellipticity}. Furthermore, if $0 < \nu < 1$, fix a scalar boundary condition $T$ with $\mathrm{ord}_\nu(T) \leq \mu$; this is just for simplicity, whereas matrix boundary conditions necessarily arise in the adjoint problem. Assume that $\mathscr{P} = \{P,T\}$ is elliptic at $\partial X$ as well. 

Without any assumptions on the behavior of $P$ away from $\partial X$, there is no reason to expect that $P$ or $\mathscr{P}$ are Fredholm. This section outlines some additional global assumptions which guarantee a Fredholm problem. The simplest of these assumptions is that $P$ is everywhere elliptic (in the standard sense) on $X$, but in view of applications to general relativity, this is overly restrictive. Indeed, operators which arise in the study of quasinormal modes on black holes spacetimes have the property that their ellipticity degenerates at the event horizon. Moreover, rotating Kerr--AdS black holes contain an ergoregion, so that the corresponding operator is not everywhere elliptic even in the black hole exterior.

The global assumptions on $P$ presented next are motivated by recent work of Vasy \cite{vasy:2013}, which applies to the setting of rotating black holes. More generally, these assumptions are typical for situations where coercive estimates are proved via propagation results. 
Given $\nu > 0$, define the space  
\[
\mathcal{Y} = 
\begin{cases}
u \in \Sob^1(X): Pu \in \Sob^0(X), \, Tu \in H^{2-\mu}(\partial X) & \text{ if $0< \nu < 1$},\\
 u \in \Sob^1(X): Pu \in \Sob^0(X) & \text{ if $\nu \geq 1$}
\end{cases}
\]
where $Pu$ is taken as a distribution on $X$. That $Tu$ is well defined follows from Lemma \ref{lem:graphnorm2}. Equip $\mathcal{Y}$ with the norm 
\[
\| u \|_{\mathcal{Y}} = 
\begin{cases} 
\| u \|_{\Sob^1(X)} + \| Pu \|_{\Sob^0(X)} + \| Tu \|_{H^{2-\mu}(\partial X)} & \text{ if $0 < \nu < 1$} \\
\| u \|_{\Sob^1(X)} + \| Pu \|_{\Sob^0(X)} & \text{ if $\nu \geq 1$}.
\end{cases}
\]
According to the discussion following Lemma \ref{lem:graphnorm2}, the space $\mathcal{Y}$ is equivalent to 
\begin{equation} \label{eq:alternativespace}
\{(u,\UL{\phi})\in\TSob^1(X): \mathscr{P}(u,\UL{\phi}) \in \Sob^0(X) \times H^{2-\mu}(\partial X)\}
\end{equation}
for $0 < \nu <1$ when the latter space is equipped with the norm $\| (u,\UL{\phi}) \|_{\TSob^1(X)} + \|\mathscr{P}(u,\UL{\phi}) \|_{\Sob^0(X)\times H^{2-\mu}(\partial X)}$. The proof of the following result is left to the reader.

\begin{lem} \label{lem:Yk}
	The space $\mathcal{Y}$ has the following properties.
	\begin{enumerate} \itemsep6pt
		\item $\mathcal{Y}$ is complete, and $\mathcal{F}_\nu(X)$ is dense in $\mathcal{Y}$.
		\item If $0 < \nu < 1$, then $\mathscr{P}:\mathcal{Y} \rightarrow \Sob^0(X) \times H^{2-\mu}(\partial X)$ is bounded.
		\item If $\nu \geq 1$, then $P : \mathcal{Y} \rightarrow \Sob^{0}(X)$ is bounded.
		\item The inclusion $\mathcal{Y} \hookrightarrow \Sob^{0}(X)$ is compact.
	\end{enumerate} 
\end{lem}

The Fredholm properties of $P$ or $\mathscr{P}$ are examined when the following a priori estimates are satisfied. If $0 < \nu < 1$, suppose that the a priori estimate
\begin{equation} \label{eq:AP0} \tag{AP0} 
\| u \|_{\Sob^1(X)} \leq C \left( \| \mathscr{P}u \|_{\Sob^0(X)\times H^{2-\mu}(X)} + \| u \|_{\Sob^0(X)}  \right)  
\end{equation}
holds for each $u \in \mathcal{Y}$. If $\nu \geq 1$, then the a priori estimate is 
\begin{equation} \label{eq:AP1} \tag{AP1} 
\| u \|_{\Sob^1(X)} \leq C \left( \| Pu \|_{\Sob^0(X)} + \| u \|_{\Sob^0(X)} \right) 
\end{equation}
for each $u \in \mathcal{Y}$. In view of the compact embedding statement in Lemma \ref{lem:Yk} (cf. Lemma \ref{lem:compact}), it is standard that \eqref{eq:AP0} or \eqref{eq:AP1} imply $\mathscr{P}:\mathcal{Y} \rightarrow \Sob^0(X) \times H^{2-\mu}(\partial X)$ or $P : \mathcal{Y} \rightarrow \Sob^0(X)$ have finite dimensional kernels (see Lemma \ref{lem:fdkernel} below).

Suppose that $0 < \nu <1$. In order to prove that $\mathscr{P}$ has finite dimensional cokernel, it is necessary to introduce spaces associated with the formal adjoint $\mathscr{P}^*$ and Hilbert space adjoint $\mathscr{P}'$. Fix a density $\mu$ on $\OL X$ of product type near $\partial X$. A priori, $\mathscr{P}^*$ is bounded
\[
\mathscr{P}^*: \TSob^{0}(X) \times H^{\mu-2}(\partial X) \rightarrow \Sob^{-2}(X) \times H^{\UL{\nu}-2}(\partial X).
\]
Recall that if $(f,\UL{g}) = \mathscr{P}^*(v,\UL{\psi},\UL{v})$, then for $u \in \Sob^2(X)$ and $\UL{w} \in H^{2-\UL{\nu}}(\partial X)$,
\begin{equation} \label{eq:adjointvsformal1}
\left< u, f \right>_{X} + \left< \UL{w}, \UL{g} \right>_{\partial X} = \left< Pu, v \right>_X + \left< \UL{w} - \UL{\gamma}u, J\UL{\psi} \right>_{\partial X} +  \left< G\UL{w}, \UL{v} \right>_{\partial X},
\end{equation}
where the dualities on $X$ and $\partial X$ are induced by $\mu$ and $\mu_{\partial X}$. Now define the space
\[
\mathcal{X} = \begin{cases} \!\begin{aligned}
&  (v,\UL{\psi},\UL{v}) \in \TSob^{0}(X) \times H^{\mu -2}(\partial X):  \\
&  \mathscr{P}^*(v,\UL{\psi},\UL{v}) \in \Sob^{-1}(X) \times H^{\UL{\nu}-1}(\partial X), 
\end{aligned}     & \text{ if } 0 < \nu < 1,  
\\  u \in \Sob^0(X): Pu \in \Sob^{-1}(X) & \text{ if } \nu \geq 1.
\end{cases}
\]
The space $\mathcal{X}$ has properties similar to those in Lemma \ref{lem:Yk}. In particular, if $0 < \nu < 1$, then the set of all $(v,\UL{\gamma}v,\UL{v})$ such that $v \in \mathcal{F}_\nu(X)$ and $\UL{v} \in C^\infty(\partial X)$ is dense in $\mathcal{X}$. Similarly, $\mathcal{F}_\nu(X)$ is dense in $\mathcal{X}$ for $\nu \geq 1$.

The analogues of \eqref{eq:AP0} and \eqref{eq:AP1} are formulated next for the adjoint problems. First suppose that $0 < \nu < 1$. The relevant a priori estimate is
\begin{align} 
\| (v,\UL{\psi},\UL{v})  \|_{\TSob^0(X) \times H^{\mu-2}(\partial X)} &\leq C ( \| \mathscr{P}^*(v,\UL\psi,\UL{v}) \|_{\Sob^{-1}(X) \times H^{\UL{\nu}-1}(\partial X)} \notag \\ &+ \| (v, \UL{\psi}, \UL{v}) \|_{\TSob^{-1}(X) \times H^{\mu - 3}(\partial X)} )   \tag{AP0*} \label{eq:AP0*}
\end{align}
for each $(v,\UL{\psi},\UL{v}) \in \mathcal{X}$. When $\nu \geq 1$ the estimate is
\begin{equation} \tag{AP1*} \label{eq:AP1*}
\| v \|_{\Sob^0(X)} \leq C( \| P^*v \|_{\Sob^{-1}(X)} + \| v \|_{\Sob^{-1}(X)})
\end{equation}
for each $v \in \mathcal{X}$.

When $0 < \nu <1$, the formally adjoint operator $\mathscr{P}^*$ should be compared with the Hilbert space adjoint 
\[
\mathscr{P}' : \Sob^{0}(X) \times H^{\mu-2}(\partial X) \rightarrow \TSob^2(X)'
\]
defined by
\[
\left< (u,\UL{\phi}), \mathscr{P}'(v,\UL{v}) \right>_{X} = \left< Pu, v \right>_X + \left< Tu, \UL{v} \right>_{\partial X}.
\] 
 Recall that the inclusion of $\TSob^2(X) \hookrightarrow \TSob^1(X)$ is dense. Consequently $\TSob^{1}(X)'$ may be identified with a dense subspace of $ \TSob^{2}(X)'$, where this identification is induced by the $\mu$--inner product. In order to describe $\TSob^1(X)'$, note that that there is an isomorphism
\[
\Phi: \TSob^1(X) \rightarrow \Sob^1(X) \times H^{-\nu}(\partial X)
\]
given by $\Phi(u,\UL{\phi}) = (u,\phi_+)$; the inverse of $\Phi$ is $\Phi^{-1}(u,\phi_+) = (u,\gamma_-u,\phi_+)$. Thus for each $\alpha \in \TSob^1(X)'$ there exist unique $f \in \Sob^{-1}(X), \, g_+ \in H^{\nu}(\partial X)$ such that
\[
\alpha(u,\UL{\phi}) = \left<f, u \right>_X + \left< g_+, \phi_+ \right>_{\partial X}.
\]
Furthermore, note that if $g_- \in H^{-\nu}(\partial X)$, then the functional given by $u \mapsto \left< g_-, \gamma_- u \right>_{\partial X}$ is an element of $\Sob^1(X)'$. Thus it may be represented in the form $u \mapsto \left<f_-, u \right>_X$ for a unique $f_- \in \Sob^{-1}(X)$. The next lemma summarizes this discussion.

\begin{lem} \label{lem:alpharep}
Each $\alpha \in \TSob^{1}(X)'$ admits a representation
\begin{equation} \label{eq:alpharep}
\alpha(u,\UL{\phi}) = \left< f,u \right>_X + \left< \UL{g}, \UL{\phi} \right>_{\partial X},
\end{equation}
where $f \in \Sob^{-1}(X)$ and $\UL{g} \in H^{\UL{\nu}-1}(\partial X)$. Furthermore, $\| \alpha \|_{\TSob^{1}(X)'}$ is equivalent to the norm
\[
\inf \{ \| f \|_{\Sob^{-1}(X)} + \|\UL{g} \|_{H^{\UL{\nu}-1}(\partial X)} \},
\]
where the infimum is taken over all $f,\UL{g}$ such that \eqref{eq:alpharep} holds.
\end{lem}
\noindent Still assuming $0<\nu<1$, define the auxiliary space
\[
\widetilde{\mathcal{X}} = \{ (v,\UL{v}) \in \Sob^0(X) \times H^{\mu-2}(\partial X): \mathscr{P}'(v,\UL{v}) \in \TSob^1(X)' \}.
\]
\begin{lem} \label{lem:propertyBadjoint} Suppose that \eqref{eq:AP0*} holds. Then
\begin{align} 
\| (v,\UL{v}) \|_{\Sob^0(X) \times H^{\mu-2}(\partial X)} &\leq C ( \| \mathscr{P}'(v,\UL{v}) \|_{\TSob^1(X)'} \notag \\ &+ \| (v,\UL{v}) \|_{\Sob^{-1}(X) \times H^{\mu - 3}(\partial X)}) \tag{AP0'}  \label{eq:fdcokernelapriori2}
\end{align}
for each $(v,\UL{v}) \in \widetilde{\mathcal{X}}$.
\end{lem}
\begin{proof} Since $\mathscr{P}'(v,\UL{v}) \in \TSob^1(X)'$, there exists $f \in \Sob^{-1}(X)$ and $\UL{g} \in H^{\UL{\nu}-1}(\partial X)$ such that the action of $\mathscr{P}'(v,\UL{v})$ on $(u,\UL{\phi}) \in \TSob^1(X)$ is given by
\begin{equation} \label{eq:alpharep2}
(u, \UL{\phi}) \mapsto \left< f , u \right>_X + \left< \UL g, \UL{\phi} \right>_{\partial X}.
\end{equation}
Now let $\UL{\psi} = JG^*\UL{v} - J\UL{g}$, so that $J\UL{\psi} + G^* \UL{v} = \UL{g}$. Furthermore, note that $\UL{\psi} \in H^{-\UL{\nu}}(\partial X)$, so $(v,\UL{\psi})$ may be considered as an element of $\TSob^0(X)$. Referring back to \eqref{eq:adjointvsformal1}, it follows that $\mathscr{P}^*(v,\UL{\psi},\UL{v}) = (f,g)$. This shows that $(v,\UL{\psi},\UL{v}) \in {\mathcal{X}}$, so
\begin{align*}
\| (v,\UL{v}) \|_{\Sob^0(X) \times H^{\mu-2}(\partial X)} &\leq C ( \| f \|_{\Sob^{-1}(X)} + \| g \|_{H^{\UL{\nu}-1}(\partial X)} \\ &+ \| (v, \UL{v}) \|_{\Sob^{-1}(X) \times H^{\mu - 3}(\partial X)} )
\end{align*}
by \eqref{eq:AP0*}. In the last line, this used the fact that
\[
\| \UL \psi \|_{H^{-1-\UL\nu}(\partial X)} \leq C ( \| \UL v \|_{H^{\mu-3}(\partial X)} + \| g \|_{H^{\UL{\nu}-1}(\partial X)}).
\]
It now suffices to take the infimum over all $f,\UL{g}$ satisfying \eqref{eq:alpharep2}, and then appeal to Lemma \ref{lem:alpharep}. 
\end{proof}

\subsection{The Fredholm property} \label{subsect:thefredholmproperty}

In this section, the Fredholm property is established whenever \eqref{eq:AP0}, \eqref{eq:AP1}, \eqref{eq:AP1}, \eqref{eq:AP1*} hold. The proof is sketched in the more complicated case $0 < \nu < 1$.

\begin{lem} \label{lem:fdkernel} Let $0 < \nu < 1$. 
\begin{enumerate} \itemsep6pt
	\item If \eqref{eq:AP0} holds, then the operator
	\[
	\mathscr{P} : \mathcal{Y} \rightarrow \Sob^{0}(X) \times H^{2-\mu}(\partial X) 
	\]
	has a finite dimensional kernel.
	\item If \eqref{eq:AP0*} holds, then the operator 
	\[
	\mathscr{P}': \Sob^{0}(X) \times H^{\mu-2}(\partial X) \rightarrow  \Sob^{-2}(X) \times H^{\UL{\nu} - 2}(\partial X) 
	\]
	has a finite dimensional kernel
\end{enumerate}	
\end{lem}

\begin{proof}
\begin{inparaenum}
\item This is immediate from the compactness of the inclusion $\mathcal{Y} \hookrightarrow \Sob^0(X)$, combined with \eqref{eq:AP0}.

\item	Clearly the kernel of $\mathscr{P}'$ restricted to $\Sob^{0}(X) \times H^{\mu-2}(\partial X)$ is equal to the kernel of $\mathscr{P}'$ restricted to $\widetilde{\mathcal{X}}$. The result follows from the same type of compactness considerations as in $(1)$, using \eqref{eq:fdcokernelapriori2}.
\end{inparaenum}
\end{proof}

\noindent In light of Lemma \ref{lem:fdkernel}, let $\mathcal{K}$ denote the finite dimensional kernel of $\mathscr{P}'|_{\widetilde{\mathcal{X}}}$. Standard functional analytic arguments (cf. \cite[Section 4.3]{vasy:hyperbolic} in a similar setting) give the following solvability result.

\begin{lem} \label{lem:fdcokerneladjoint} Let $0 < \nu < 1$, and assume that \eqref{eq:fdcokernelapriori2} holds. If
	\[
	(h, \UL{k}) \in \Sob^{0}(X) \times H^{2 -\mu}(\partial X)
	\]
	 lies in the annihilator of $\mathcal{K}$ via the duality between $\Sob^{0}(X) \times H^{\mu-2}(\partial X)$ and $\Sob^{0}(X) \times H^{2 - \mu}(\partial X)$, then there exists $(u,\UL{\phi}) \in \TSob^1(X)$ such that $\mathscr{P}(u,\UL{\phi}) = (h,\UL{k})$.
\end{lem}

The combination of Lemmas \ref{lem:fdkernel} and \ref{lem:fdcokerneladjoint} imply the Fredholm property:

\begin{theo} \label{theo:fredholm}
Let $\nu > 0$ and $P$ as in Section \ref{subsect:globalass} be elliptic at $\partial X$. If $0 < \nu < 1$, then let $T$ denote a scalar boundary operator satisfying $\mathrm{ord}_\nu(T) \leq \mu$, such that $\mathscr{P} = \{P,T\}$ is elliptic at $\partial X$.

\begin{enumerate} \itemsep6pt
\item Suppose that $0 < \nu < 1$. If $\mathscr{P}$ satisfies \eqref{eq:AP0} and \eqref{eq:AP0*}, then
\[
\mathscr{P}: \mathcal{Y} \rightarrow \Sob^0(X) \times H^{2- \mu}(\partial X)
\]
is Fredholm.
\item Suppose that $\nu \geq 1$. If $P$ satisfies \eqref{eq:AP1} and \eqref{eq:AP1*}, then
\[
P: \mathcal{Y} \rightarrow \Sob^0(X)
\]
is Fredholm.
\end{enumerate}
\end{theo} 
\begin{proof}
\begin{inparaenum}
	\item  Lemma \ref{lem:fdkernel} shows the kernel is finite dimensional. On the other hand, Lemma \ref{lem:fdcokerneladjoint} shows that the equation $\mathscr{P}(u,\UL{\phi}) = (h,\UL{k}) $ has a solution $(u,\UL{\phi}) \in \TSob^1(X)$ for $(h,\UL{k})$ in a space of finite codimension in $\Sob^0(X)\times H^{2-\mu}(X)$; clearly this $(u,\UL{\phi})$ can be identified with a unique element of $\mathcal{Y}$, namely $u$.   

\item  When $\nu \geq 1$, there is a natural analogue of Lemma \ref{lem:fdcokerneladjoint}. Since the arguments are simpler when there is no boundary operator, the proofs are omitted.
\end{inparaenum}
\end{proof}

\subsection{Unique solvability}

In this section, again let $\OL{X}$ denote a compact manifold with boundary as in Section \ref{subsect:manifoldwithboundary}. This time, consider 
\[
P(\lambda) \in \mathrm{Bess}^{(\lambda)}_\nu(X).
\]
Assume that $P(\lambda)$ is parameter-elliptic at $\partial X$ with respect to an angular sector $\Lambda$ in the sense of Section \ref{subsect:semiclassicalbessel}. If $0 < \nu < 1$, fix a scalar boundary condition $T(\lambda)$ with $\mathrm{ord}^{(\lambda)}_\nu(T(\lambda)) \leq \mu$, and assume that $\mathscr{P}(\lambda) = \{P(\lambda),T(\lambda)\}$ is parameter-elliptic at $\partial X$ with respect to $\Lambda$. It is also assumed that the ``principal parts'' of $P(\lambda), T(\lambda)$ do not depend on $\lambda$, so the spaces $\mathcal{Y}$ defined in the previous section are independent of $\lambda$.

The parameter-dependent versions of \eqref{eq:AP0}, \eqref{eq:AP1}, \eqref{eq:AP0*}, \eqref{eq:AP1*} are obtained by replacing the norms $\| \cdot \|$ with their uniform counterparts $\VERT \cdot \VERT$, and insisting that the estimates hold for all $\lambda \in \Lambda$.

\begin{theo} \label{theo:Lambdafredholm}
Let $\nu > 0$ and $P(\lambda), \mathscr{P}(\lambda), \Lambda$ be as above. Suppose that the parameter-dependent versions of \eqref{eq:AP0}, \eqref{eq:AP1}, \eqref{eq:AP0*}, \eqref{eq:AP1*} hold. 

\begin{enumerate} \itemsep6pt
\item Let $0 < \nu < 1$. There exists $R> 0$ such that
\[
\mathscr{P}(\lambda): \mathcal{Y} \rightarrow \Sob^0(X) \times H^{2- \mu}(\partial X)
\]
is an isomorphism for all $\lambda \in \Lambda$ satisfying $|\lambda| > R$.
\item Let $\nu \geq 1$. Then there exists $R> 0$ such that
\[
P(\lambda): \mathcal{Y} \rightarrow \Sob^0(X)
\]
is an isomorphism for all $\lambda \in \Lambda$ satisfying $|\lambda| > R$.
\end{enumerate}
\end{theo} 
\begin{proof}
The parameter-dependent versions of \eqref{eq:AP0}, \eqref{eq:AP0*} show that $\mathscr{P}(\lambda)$ and $\mathscr{P}(\lambda)'$ are injective on the appropriate spaces (for $\lambda \in \Lambda$ with $|\lambda|$ sufficiently large). This implies that $\mathscr{P}(\lambda)$ is an isomorphism for $|\lambda|$ sufficiently large. Similar remarks hold for $P$ when $\nu \geq 1$.
\end{proof}

\section{Completeness of generalized eigenfunctions} \label{sect:completeness}

In this section, sufficient conditions are given which guarantee that an elliptic parameter-dependent Bessel operator has a complete set of generalized eigenvectors. Completeness of eigenvectors for non-self adjoint boundary value problems has a long history, going back to classical works of Keldysh \cite{keldyvs1951characteristic}, Browder \cite{browder1953eigenfunctions}, Schechter \cite{schechter1959remarks}, Agmon \cite{agmon1962eigenfunctions}, among many others. The results of this section apply to large classes of Bessel operator pencils with a spectral parameter in the boundary condition, and two-fold completeness is established (a condition stronger than completeness, described below).

One application of this section is to describe a class of boundary conditions for which linearized scalar perturbations of global anti-de Sitter space have complete sets of normal modes. Recent numerical and perturbative studies have hinted at a relationship between the linear spectra of such perturbations and possible nonlinear instability mechanisms \cite{balasubramanian2014holographic,bizon:2014:grg,bizon:2011:prl,bizon2015resonant,buchel2015conserved,dias:2012:cqg:b,dias:2012:cqg}. These normal modes have been studied by separation of variables techniques, but there has not appeared a general criterion guaranteeing completeness of normal modes (nor even the discreteness of normal frquencies) for general boundary conditions. The results of this section also apply to more general stationary aAdS spacetimes with compact time slices where $\partial_t$ is Killing but the spacetime is not necessarily static.

\subsection{Two-fold completeness}

The main reference for this section is \cite[Chapter II]{markus2012introduction}.
Let $\OL{X}$ be a manifold with boundary, and let $P(\lambda) \in \mathrm{Bess}^{(\lambda)}_\nu(X)$ be a  parameter-dependent Bessel operator. If $0< \nu < 1$, let $T(\lambda)$ be a scalar parameter-dependent boundary operator, written in the form $T(\lambda) = T_1 + \lambda T_0$.

If $0 < \nu < 1$, a complex number $\lambda_0 \in \mathbb{C}$ is said to be an eigenvalue of $\mathscr{P}(\lambda)$ if there exists $u_0 \in \Sob^2(X)$ such $\mathscr{P}(\lambda_0)u_0 = 0$. Corresponding to an eigenvalue $\lambda_0$, a sequence $(u_0,\ldots,u_k)$ with $u_0 \neq 0$ is said to be a chain of generalized eigenvectors if
\[\begin{cases}
P(\lambda_0)u_p + \partial_\lambda P(\lambda_0) u_{p-1} + \frac{1}{2} \partial_\lambda^2 P(\lambda_0) u_{p-2} = 0, \\
T(\lambda_0)u_p + \partial_\lambda T(\lambda_0)u_{p-1} = 0
\end{cases}
\]
for $p = 0,\ldots, k$. Thus $(u_0,\ldots,u_k)$ is a chain of generalized eigenvectors with eigenvalue $\lambda_0$ if and only if the function
\[
u(t) = e^{\lambda_0 t} \sum_{j=0}^k \frac{t^j}{j!} u_{k-j}
\]
solves the (time-dependent) equation $\mathscr{P}(\partial_t)u(t) = 0$. Such a solution $u(t)$ is called elementary. To each elementary solution is associated the Cauchy data $(u(0),\partial_t u(0))$. The set of generalized eigenvectors (for all possible eigenvalues) is said to be two-fold complete in a Hilbert space $H$ continuously embedded in $\Sob^0(X)\times \Sob^0(X)$ if the span of all Cauchy data $(u(0),\partial_t u(0))$ corresponding to elementary solutions (for all eigenvalues) is dense in $H$. The same definition holds for $\nu \geq 1$, this time replacing $\mathscr{P}(\lambda)$ with $P(\lambda)$.

A general criterion concerning two-fold completeness is given by \cite[Theorem 3.4]{yakubov1993completeness}, which is a refinement of the standard reference \cite[Corollary XI.9.31]{dunford1963linear}.

\begin{theo} \label{theo:completeness}
Let $P(\lambda), T(\lambda)$ be defined as above. Fix rays $\Gamma_1,\ldots, \Gamma_s$ through the origin of the complex plane such the angle between any two adjacent rays is less than or equal to $\pi/n$, where $\dim X = n$. 
\begin{enumerate} \itemsep6pt
	\item Let $0 < \nu < 1$. If $\mathscr{P}(\lambda)$ is elliptic with respect to $\Gamma_1,\ldots, \Gamma_s$, then the eigenvalues of $\mathscr{P}(\lambda)$ are discrete and the set of generalized eigenvectors is two-fold complete in the space $\{(v_1,v_2)\in \Sob^2(X)\times \Sob^1(X): T_0 v_2 + T_1 v_1 = 0 \}$.
	
	\item Let $\nu \geq 1$.  If $P(\lambda)$ is elliptic with respect to $\Gamma_1,\ldots, \Gamma_s$, then the eigenvalues of $P(\lambda)$ are discrete and the set of generalized eigenvectors is two-fold complete in the space $\Sob^2(X)\times \Sob^1(X)$.
\end{enumerate}
\end{theo}
\begin{proof} 
	\begin{inparaenum}
	\item First suppose that $0 < \nu < 1$. To apply \cite[Theorem 3.4]{yakubov1993completeness}, it must be verified that the singular values of the embeddings $J_k : \Sob^{k}(X)\hookrightarrow \Sob^{k-1}(X)$ satisfy $s_j(J_k) \leq Cj^{-1/n}$ for $k = 1,2$, and that the space $\{(v_1,v_2)\in \Sob^2(X)\times \Sob^1(X): T_0 v_2 + T_1 v_1 = 0 \}$ is dense in $\Sob^1(X)\times \Sob^0(X)$.
	
The claim about the singular values follows from Lemma \ref{lem:singularvaluesmanifold}. To verify the density claim, let $(u_1,u_2) \in \Sob^1(X)\times \Sob^0(X)$, and take a sequence $(v_1^n,v_2^n) \in \Sob^2(X)\times \Sob^1(X)$ such that $(v_1^n,v_2^n) \rightarrow (u_1,u_2)$ in $\Sob^1(X)\times \Sob^0(X)$ as $n \rightarrow \infty$. Note that 
\[
T_0 v_2^n + T_1 v_1^n \in H^{2-\mu}(\partial X).
\]
Once the sequence $(v_1^n,v_2^n)$ is fixed, choose a sequence $\lambda_n\in\mathbb{C}$ such that $|\lambda_n|$ tends to infinity along one of the rays of ellipticity (say $\Gamma_1$) and 
\begin{equation} \label{eq:lambdagrowth}
|\lambda_n|^{-1} \|T_0 v_2^n + T_1 v_1^n \|_{H^{2-\mu}(\partial X)} \rightarrow 0
\end{equation}
as $n\rightarrow \infty$. According to Theorem \ref{theo:Lambdafredholm}, the operator \[
\mathscr{P}(\lambda_n)^{-1} : \Sob^0(X) \times H^{2-\mu}(X) \rightarrow \Sob^2(X)
\] 
exists for $n$ sufficiently large, where $ \mathrm{ord}_\nu^{(\lambda)}(T(\lambda)) \leq \mu$. Let
\[
w_1^n = \mathscr{P}(\lambda_n)^{-1}(0,-T_0 v_2^n - T_1 v_1^n),
\] 
so $w_1^n$ lies in $\Sob^2(X)$, and set $w_2^n = \lambda_n w_1^n$. Then 
\[
(v_1^n + w_1^n, v_2^n + w_2^n) \in \{(v_1,v_2)\in \Sob^2(X)\times \Sob^1(X): T_0 v_2 + T_1 v_1 = 0 \}.
\]
Furthermore, according to Theorems \ref{theo:bvptheosemiclassical} and \ref{theo:Lambdafredholm} the solution $w_1^n$ satisfies
\[
|\lambda_n|^{2-s} \| w_1^n \|_{\Sob^s(X)} \leq C \|T_0 v_2^n + T_1 v_1^n \|_{H^{2-\mu}(\partial X)}
\]
for $s = 0,1$. Thus $(w_1^n,w_2^n) \rightarrow 0$ in $\Sob^1(X)\times \Sob^0(X)$ by the choice of $\lambda_n$ in \eqref{eq:lambdagrowth}. This shows that $(v_1^n + w_1^n, v_2^n + w_2^n) \rightarrow (u_1,u_2)$, establishing the density.

\item  For $\nu \geq 1$ the singular value estimates remain the same, and the corresponding density of $\Sob^2(X) \times \Sob^1(X)$ in $\Sob^1(X)\times \Sob^0(X)$ is trivial.
\end{inparaenum}
\end{proof}

\section*{Acknowledgements}
This paper is based on work supported by NSF grants DMS-1201417 and DMS-1500852. I would like to thank Maciej Zworski for his encouragement on this project.

\appendix  

\section{Proof of lemma \ref{lem:traceproperty}} \label{appendix1}

The proof of Lemma \ref{lem:traceproperty} is broken up into several stages. Recall in this section that $\gamma_\pm$ are defined as in the beginning of Section \ref{subsect:traces} without any mention of the space $\Cnu$.

\begin{lem} \label{lem:absolutelycontinuoustrace} Let $\nu  > 0$.
	\begin{enumerate} \itemsep6pt
		\item If $u \in \Sob^{1}(\RNP)$ and $\gamma_- u = 0$, then for a.e. $y \in \RN$,
		\[
		u(x,y) = x^{1/2-\nu}\int_0^x t^{\nu-1/2} \TP u(t,y) \,dt.
		\]
		\item Suppose in addition that $0 < \nu < 1$. If $u \in \Sob^{2}(\RNP)$, and $\UL{\gamma}u = 0$, then for a.e. $y \in \RN$,
		\[
		u(x,y) = x^{1/2-\nu}\int_0^x t^{-2\nu + 1} \int_0^t s^{1/2-\nu} \TP^*\TP u(s,y)\, ds\, dt.
		\]
	\end{enumerate}
\end{lem}
\begin{proof}
	
	These two facts follow from the Sobolev embedding for weighted spaces, as in Section \ref{subsect:weighted}. In the first case, for a.e. $y\in \RN$ the function $x \mapsto x^{\nu-1/2}u(x,y)$ is absolutely continuous on $\OL{\RP}$, and $\gamma_- u = 0$ implies that $x^{\nu - 1/2}u(x,y) \rightarrow 0$ as $x\rightarrow 0$ for a.e. $y \in \RN$. The the result follows from the fundamental theorem of calculus. A similar argument applies in the second case, in which the functions $x \mapsto x^{\nu-1/2}u(x,y)$ and  $x \mapsto x^{1/2-\nu}\TP u(x,y)$ are absolutely continuous on $\OL{\RP}$ for a.e. $y\in \RN$, and vanish at $x=0$.
\end{proof}

\begin{lem} \label{lem:mathring}
	Let $0 < \nu < 1$. Then $\dot{\Sob}^{1}(\RNP)= \ker \gamma_-$, and $\dot{\Sob}^{2}(\RNP) = \ker \UL{\gamma}$.
\end{lem}
\begin{proof} The first equality comes from \cite[Proposition 1.2]{grisvard:1963}. It remains to show the second equality.
	\begin{inparaenum}

\item First show that if $u \in \Sob^{2}(\RNP)$ and $\UL{\gamma}u = 0$, then $u \in \dot{\Sob}^{2}(\RNP)$.  Begin by assuming that $u$ has compact support in $\OL{\RNP}$, which is possible by a standard truncation argument. Next, fix $\chi \in C_c^\infty(\RP)$ satisfying 
\[
	0 \leq \chi \leq 1, \quad \chi(x) = 0 \text{ for } 0 \leq x \leq 1, \quad \chi(x) = 1 \text{ for } x\geq 2.
\] 
	Consider the sequence $u_n(x,y) = \chi(nx)u(x,y)$; since $u_n$ has compact support in the interior $\RNP$, it follows that 
	\[
	u_n \in \dot\Sob^2(\RNP) \Longleftrightarrow u_n\in  \dot H^2(\RNP)
	\]
	 by comparability of norms in the interior and the well known characterization of $\dot{H}^2(\RNP)$ as the kernel of the smooth trace maps.
	
	It now suffices to prove that $u_n \rightarrow u$ in $\Sob^{2}(\RNP)$, since $\dot{\Sob}^{2}(\RNP)$ is closed. This is easily deduced from Lemma \ref{lem:absolutelycontinuoustrace}, imitating the proof of \cite[Chapter 5.5, Theorem 2]{evans:2010} for instance.
	
	\item The inclusion $\dot{\Sob}^{2}(\RNP) \subseteq \ker \UL{\gamma}$ is clear, since $\UL{\gamma} = 0$ for each $u \in C_c^\infty(\RNP)$, and hence $\UL{\gamma}u = 0$ for each $u \in \dot{\Sob}^{2}(\RNP)$ by density and continuity.
\end{inparaenum}
\end{proof}

\begin{lem}  \label{lem:tracelift}
	There exists a map 
	\[
	\mathcal{K}: C^\infty(\RN) \times C^\infty(\RN) \rightarrow \Cnu
	\]
	such that $\UL{\gamma} \circ \mathcal{K} = 1$ on $C^\infty(\RN)\times C^\infty(\RN)$ and $\widetilde{\mathcal{K}} = \mathcal{K} \times 1$ extends by continuity to a map
	\[
	\mathcal{K} : H^{s-\UL{\nu}}(\RN) \rightarrow \TSob^{s}(\RNP)
	\]
	for each $s=0,\pm 1,\pm 2$. If $s=2$, then $\mathcal{K}$ extends to a right inverse for $\UL{\gamma}$ acting on $\Sob^2(\RNP)$.
\end{lem}
\begin{proof}
	Let $\varphi \in C^\infty_c(\overline{\RP})$ be such that $\varphi = 1$ near $x =0$, and set
	\[
	v_-(x) = x^{1/2-\nu} \varphi(x^2), \quad v_+(x) = (2\nu)^{-1} x^{1/2+\nu} \varphi(x^2),
	\]
	so $v_\pm \in \Cnu$. Given $(f_-,f_+) \in C^\infty(\RN) \times C^\infty(\RN)$, define $u_\pm(x,y)$ by its Fourier coefficients,
	\[
	\hat{u}_\pm(x,q) = \left<q \right>^{-(1/2\pm \nu)} \hat{f}_\pm(q) v_\pm(\left<q\right>x). 
	\]
	Then $u_\pm\in \Cnu$ and $\gamma_\pm (u_- + u_+)  =  f_\pm$ in the sense of Lemma \ref{lem:directtrace}; set 
	\[
	\mathcal{K}(f_-,f_+) := u_- + u_+.
	\] 
	 Appealing to Section \ref{subsect:fourier} shows that the map defined by
	\[
	\widetilde{\mathcal{K}}(f_-,f_+) := (\mathcal{K}(f_-,f_+),f_-,f_+)
	\]
	extends by continuity to a map $H^{s-\UL{\nu}}(\RN) \rightarrow \TSob^s(\RNP)$. The last statement about the $s=2$ case follows from the natural identification $\Sob^2(X) = \TSob^2(X)$.
\end{proof}

\begin{lem}
	If $\nu > 0$, then $\Cnu$ is dense in $\Sob^{s}(\RNP)$ for each $s = 0,1,2$.
\end{lem}
\begin{proof}
	The proof for $\nu \geq 1$ can be done directly by a mollification argument, so only the case $0 < \nu < 1$ is considered here. Furthermore, the result is obvious when $s=0$. The proof is given here in the case $s= 2$; the case $s=1$ is simpler, and can be handled similarly. 
	
	Suppose that $u \in \Sob^{2}(\RNP)$, and let $\tilde{u} = \mathcal{K}(\UL{\gamma}u)$, viewed as an element of $\Sob^2(X)$. Then $\UL{\gamma}(u - \tilde u) = 0$, so $u - \tilde u \in \dot{\Sob}^{2}(\RNP)$ by Lemma \ref{lem:mathring}. It follows that there exists a sequence $u_j \in C_c^\infty(\RNP)$ such that $u_j \rightarrow u-\tilde{u}$ in $\Sob^{2}(\RNP)$. On the other hand, approximate $\UL{\gamma}u$ by a sequence $\underline{v}_j \in C^\infty(\RN) \times C^\infty(\RN)$, and hence $\tilde{u}_j = \mathcal{K}\underline{v}_j$ satisfies $\tilde{u}_j \in \Cnu$ and $\tilde{u}_j \rightarrow \tilde u$ in $\Sob^{2}(\RNP)$. Therefore, $u_j + \tilde{u}_j \in \Cnu$ and $u_j + \tilde{u}_j \rightarrow u$, which shows that $\Cnu$ is dense in $\Sob^{2}(\RNP)$.
\end{proof}

\section{Singular values}

\begin{lem} \label{lem:singularvaluesmanifold} Let $\OL{X}$ be a compact manifold with boundary. If $\nu > 0$, then the embeddings
	\[
	J_1: \Sob^1(X) \hookrightarrow \Sob^0(X), \quad J_2: \Sob^2(X) \hookrightarrow \Sob^1(X)
	\]
	are compact, and the singular values of $J_i$ satisfy $s_j(J_i) \leq Cj^{-1/n}$.
\end{lem}

To prove the lemma, first let $L$ denote the self-adjoint operator on $\TSHARP = (0,1) \times \RN$ with distributional action given by $\Delta_\nu$ and form domain $\{ u \in {\Sob}^1(\TSHARP): u(1,\cdot) = 0\}$. The remarks following Lemma \ref{lem:graphnorm2}, Green's formula, and Theorem \ref{theo:bvptheo} show that 
\[
D(L) = \begin{cases} u \in \Sob^2(\TSHARP): \gamma_+ u = u(1,\cdot)= 0 & \text{ if } 0 < \nu <1,
\\
u \in \Sob^2(\TSHARP): u(1,\cdot)= 0 & \text{ if } \nu \geq 1,
\end{cases} 
\]
and in either case Theorem \ref{theo:fredholm} guarantees that $L$ has discrete spectrum. Note that $D((L+1)^{1/2}) = \{ u \in {\Sob}^1(\TSHARP): u|_{x=1} = 0\}$. The eigenvalues and eigenvectors are well known. When $0 < \nu < 1$ the eigenvalues are $|q|^2 + y_{\nu,n}^2+1$, where $q \in \mathbb{Z}^{n-1}$ and $y_{\nu,n}$ is the $n$'th positive root of the Bessel function $Y_\nu$. The corresponding eigenfunction is 
\[
\sqrt{x}Y_\nu(y_{\nu,n}x)\otimes e^{i\left<q,y\right>}.
\]
The zeros $y_{\nu,n}$ satisfy the asymptotic formula
\[
y_{\nu,n} = \left( n + \tfrac{1}{2}\nu - \tfrac{3}{4}\right)\pi + \mathcal{O}(n^{-1})
\]
as $n \rightarrow \infty$. The eigenvalues of the compact operator $(L+1)^{-1/2}$ are therefore $(1+|q|^2 +y_{\nu,n}^2)^{-1/2}$, and if they are listed in descending order  with multiplicity, then
\begin{equation} \label{eq:singularbound}
\lambda_j \leq Cj^{-1/n}
\end{equation}
for some $C>0$. If $\nu \geq 1$ then $Y_{\nu}$ should be replaced by $J_\nu$, and $y_{\nu,n}$ by the zeros $j_{\nu,n}$ of $J_\nu$; the bound \eqref{eq:singularbound} remains valid.

\begin{proof}[Proof of Lemma \ref{lem:singularvaluesmanifold}]
	\begin{inparaenum}
\item 	 First consider the operator $J_1$. By passing to a partition of unity, it suffices to bound the singular values of the inclusion
	\[
	J_1 : D((L+1)^{1/2}) \hookrightarrow L^2(\TSHARP).
	\]
	The operator $(L+1)^{-1/2}$ is an isomorphism acting $L^2(\TSHARP) \rightarrow \{u \in {\Sob}^1(\TSHARP): u(1,\cdot)= 0\}$. Write 
	\[
	J_1=  (L + 1)^{-1/2} (L+1)^{1/2}.
	\]
	Now $ (L + 1)^{-1/2}$ is self-adjoint and non-negative on $L^2(\TSHARP)$, so its singular values are the $\lambda_j$ which satisfy $\lambda_j \leq Cj^{-1/n}$. Furthermore, $(L + 1)^{1/2}$ is bounded ${\Sob}^1(\TSHARP) \rightarrow L^2(\TSHARP)$, so the inequality $s_j(AB) \leq s_j(A)\|B\|$ shows that $s_j(J_1) \leq Cj^{-1/n}$. 
	
\item Next, consider $J_2$. First suppose that $\nu \geq 1$. Again passing to a partition of unity, it suffices to bound the singular values of 
\[
J_2 : \{u \in \Sob^2(\TSHARP): u(1,\cdot) = 0\} \hookrightarrow D((L+1)^{1/2}).
\]
Since the space on the left hand side is just $D(L)$ when $\nu \geq 1$, one proceeds as in the first part of the proof.
	
	The proof when $0 < \nu < 1$ proceeds differently.	 In that case, one considers the inclusion $J_2 : \Sob^2(\TSHARP) \hookrightarrow \Sob^1(\TSHARP)$ directly. Note that $\Sob^2(\TSHARP)$ may be identified with a closed subspace $H$ of $\Sob^1(\TSHARP)^{n} \times \Sob^1_{*}(\TSHARP)$ via the mapping
	\[
	u \mapsto (u,\partial_{y_1}u,\ldots, \partial_{y_{n-1}}u, \TP u).
	\]
	With this in mind, the embedding $\Sob^2(\TSHARP) \hookrightarrow \Sob^1(\TSHARP)$ is identified with the embedding 
	\begin{equation} \label{eq:H2compact}
	\Sob^1(\TSHARP)^{n} \times \Sob^1_{*}(\TSHARP) \hookrightarrow L^2(\TSHARP)^{n+1},
	\end{equation}
	restricted to $H$.  Since $0 < 1-\nu < 1$, by the first part of the proof the singular values of the embedding \eqref{eq:H2compact} are bounded by $Cj^{-1/n}$. The same is therefore true of the embedding $\Sob^2(\TSHARP)\hookrightarrow \Sob^1(\TSHARP)$.
\end{inparaenum}
\end{proof}

\bibliographystyle{alphanum}

\bibliography{adsbib}

\end{document}